\documentclass[12pt, reqno]{amsart}
\usepackage{amsmath}
\usepackage{amssymb}
\usepackage{epsfig}
\usepackage{graphicx}
\usepackage{color}
\usepackage{fullpage}
\definecolor{shadecolor}{gray}{0.875}
\usepackage{amscd}
\usepackage{comment}

\numberwithin{equation}{section}

\input xy
\xyoption{all}

\calclayout
\allowdisplaybreaks[3]

\theoremstyle{plain}
\newtheorem{prop}{Proposition}[section]

\newtheorem{theo}[prop]{Theorem}
\newtheorem{coro}[prop]{Corollary}

\newtheorem{lemm}[prop]{Lemma}

\theoremstyle{definition}
\newtheorem{defi}[prop]{Definition}

\newtheorem{conj}[prop]{Conjecture}

\newtheorem{rema}[prop]{Remark}

\newtheorem{exam}[prop]{Example}

\newtheorem{clai}[prop]{Claim}

\newtheorem{cons}[prop]{Construction}
\newtheorem{fact}[prop]{Fact}

\def\Br{\mathrm{Br}}

\def\Eff{\overline{\mathrm{Eff}}}
\def\Pic{\mathrm{Pic}}

\def\Chow{\mathrm{Chow}}
\def\Mor{\mathrm{Mor}}
\def\Nef{\mathrm{Nef}}
\def\Hilb{\mathrm{Hilb}}
\def\Supp{\mathrm{Supp}}

\def\Sym{\mathrm{Sym}}
\def\Pic{\mathrm{Pic}}
\def\neg{\mathrm{neg}}
\def\Sec{\mathrm{Sec}}
\def\Manin{\mathrm{Manin}}
\def\IJ{\mathrm{IJ}}
\def\AJ{\mathrm{AJ}}
\def\Jac{\mathrm{Jac}}
\def\Alb{\mathrm{Alb}}
\def\maxdef{\mathrm{maxdef}}
\def\MBB{\mathrm{MBBbound}}
\def\Gal{\mathrm{Gal}}

\def\expdim{\mathrm{expdim}}

\makeatother
\makeatletter

\author{Brian Lehmann}
\address{Department of Mathematics \\
Boston College  \\
Chestnut Hill, MA \, \, 02467}
\email{lehmannb@bc.edu}

\author{Sho Tanimoto}
\address{Graduate School of Mathematics, Nagoya University, Furocho Chikusa-ku, Nagoya, 464-8602, Japan}
\email{sho.tanimoto@math.nagoya-u.ac.jp}

\title[Classifying sections of del Pezzo fibrations]{Classifying sections of del Pezzo fibrations, I}


\begin{document}

\begin{abstract}
We develop a strategy to classify the components of the space of sections of a del Pezzo fibration over $\mathbb P^1$.  In particular, we prove the Movable Bend-and-Break lemma for del Pezzo fibrations.  Our approach is motivated by Geometric Manin's Conjecture and proves upper bounds on the associated counting function.  We also give applications to enumerativity of Gromov-Witten invariants and to the study of the Abel-Jacobi map.
\end{abstract}

\maketitle

\tableofcontents

\section{Introduction}

Manin's Conjecture predicts the asymptotic growth rate of the number of rational points of bounded height on a smooth Fano variety defined over a global field.  The modern formulation of the conjecture (\cite{FMT89}, \cite{BM}, \cite{Peyre}, \cite{BT}, \cite{LST18}) has been heavily influenced by unpublished notes of Batyrev (\cite{Bat88}).  In these notes Batyrev develops a heuristic for Manin's Conjecture over a global function field that is based on certain assumptions about the geometry of the moduli space of curves.  By translating the principles underlying Batyrev's heuristic from global function fields to complex function fields, we obtain a set of predictions known as Geometric Manin's Conjecture (described in more detail below).

We initiate the study of Geometric Manin's Conjecture for Fano varieties over $\mathbb{C}(t)$ by investigating del Pezzo surfaces.  It is usually more convenient to work directly with an integral model:

\begin{defi}
Let $k$ be an algebraically closed field of characteristic $0$.  A del Pezzo fibration over $\mathbb{P}^{1}$ is an algebraic fiber space $\pi: \mathcal{X} \to \mathbb{P}^{1}$ such that $\mathcal{X}$ has only Gorenstein terminal singularities and the general fiber of $\pi$ is a del Pezzo surface.  
\end{defi}

Our main results address the structure of the space of sections $\Sec(\mathcal{X}/\mathbb{P}^{1})$ of a del Pezzo fibration.  First, we describe an inductive procedure for generating all irreducible components of $\Sec(\mathcal{X}/\mathbb{P}^{1})$ from a finite set of components in low degree.  We illustrate this technique in many examples.  Second, we identify the exceptional set in Geometric Manin's Conjecture for del Pezzo surfaces and prove that it is controlled by the Fujita invariant.  Our work validates Batyrev's heuristic (in characteristic $0$): we define a formal counting function and prove that it has the properties predicted by the heuristic.  Third, we discuss applications to stabilization of the Abel-Jacobi map and to enumerativity of certain Gromov-Witten invariants associated to spaces of sections.

\subsection{Overview of goals and motivation}

Let us review Manin's Conjecture over a global function field for trivial families over $\mathbb P^1$.
Let $X$ be a smooth Fano variety defined over a finite field $\mathbb F_q$. The main protagonist is the moduli space of rational curves $\mathrm{Mor}(\mathbb P^1, X)$. For each numerical curve class $\alpha$, we denote by $\mathrm{Mor}(\mathbb P^1, X, \alpha)$ the fine moduli scheme parametrizing morphisms $f : \mathbb P^1 \to X$ such that $f_*[\mathbb P^1] = \alpha$. In this setting Manin's Conjecture concerns the behavior of the number of $\mathbb F_q$-points
\[
\mathrm{Mor}(\mathbb P^1, X, \alpha)(\mathbb F_q)
\]
as the anticanonical degree $d = -K_X \cdot \alpha$ goes to $+\infty$.

In \cite{Bat88}, Batyrev developed a heuristic to establish the asymptotic formula for this counting function based on the following assumptions:
\begin{enumerate}
\item There exists some proper closed subset $V$ such that any irreducible component of $\mathrm{Mor}(\mathbb P^1, X)$ parametrizing rational curves not contained in $V$ has the expected dimension;
\item for each nef numerical class $\alpha$, $\mathrm{Mor}(\mathbb P^1, X, \alpha)$ is irreducible;
\item the cardinality of $\mathrm{Mor}(\mathbb P^1, X, \alpha)(\mathbb F_q)$ is approximated by $q^{\dim \mathrm{Mor}(\mathbb P^1, X, \alpha)}$.
\end{enumerate}
Although these assumptions are not valid in general, based on \cite{LST18} we can expect that their failure is controlled by an invariant from birational geometry known as the Fujita invariant.  In \cite{LT17} we investigated these assumptions for rational curves on complex Fano varieties.

Manin's Conjecture makes predictions in a more general setting. For a non-trivial family of Fano varieties $\pi : \mathcal X \to \mathbb P^1$, Manin's Conjecture concerns the space of sections
\[
\Sec(\mathcal X/\mathbb P^1)
\]
instead of the space of rational curves $\mathrm{Mor}(\mathbb P^1, X)$. We denote by $\Sec(\mathcal{X}/\mathbb{P}^{1},\alpha)$ the finite type fine moduli scheme parametrizing sections of numerical class $\alpha$.  We expect the following principles to hold for fibrations with Fano fibers over an arbitrary ground field:
\begin{enumerate}
\item There exists some proper closed subset $V$ such that any irreducible component of $\Sec(\mathcal{X}/\mathbb{P}^{1})$ parametrizing rational curves not contained in $V$ has the expected dimension;
\item for each sufficiently positive nef numerical class $\alpha$, $\Sec(\mathcal{X}/\mathbb{P}^{1},\alpha)$ consists of $|\Br(\mathcal{X})|$ irreducible components parametrizing sections not coming from exceptional sets; 
\item the irreducible components of $\Sec(\mathcal{X}/\mathbb{P}^{1},\alpha)$ exhibit motivic or homological stability as the anticanonical degree of $\alpha$ goes to $+\infty$.
\end{enumerate}
(The third principle was proposed by Ellenberg and Venkatesh who noted that homological stability of $\mathrm{Mor}(\mathbb P^1, X, \alpha)$ combined with the Grothendieck-Lefschetz trace formula could prove Batyrev's heuristic on point counting; see \cite{EV05}.)   Furthermore, Geometric Manin's Conjecture predicts that the exceptional set can be explicitly identified using the Fujita invariant as explained in \cite{LST18}.

 In this paper we study Geometric Manin's Conjecture in characteristic $0$ for $\Sec(\mathcal X/\mathbb P^1)$ when $\pi : \mathcal X \to \mathbb P^1$ is a del Pezzo fibration over $\mathbb P^1$. The main results of this paper are the following.
For simplicity suppose that $\mathcal X$ is smooth, the relative anticanonical class $-K_{\mathcal X/\mathbb P^1}$ is relatively ample and the degree of a general del Pezzo fiber is $\geq 2$. Under these assumptions, we prove:
\begin{enumerate}
\item there exists a proper closed $V \subset \mathcal X$ such that any irreducible component of $\Sec(\mathcal X/\mathbb P^1)$ parametrizing sections not contained $V$ has the expected dimension;
\item Movable Bend-and-Break theorem which is an inductive technique to prove the irreducibility of $\Sec(\mathcal X/\mathbb P^1, \alpha)$;
\item a polynomial upper bound $P(d)$ for the number of components of $\Sec(\mathcal X/\mathbb P^1)$ parametrizing sections of height $-K_{\mathcal X/\mathbb P^1} \cdot C \leq d$;
\item a certain stabilization of the Stein factorization of Abel-Jacobi mappings;
\item the enumerativity of certain Gromov-Witten invariants on $\mathcal X$.
\end{enumerate}

The main theorem of this paper is the Movable Bend-and-Break theorem for del Pezzo fibrations.  Recall that Mori's Bend-and-Break lemma states that if we deform a rational curve while fixing two points on it, then it degenerates into a non-integral curve with rational components.  However, in general it is hard to control the properties of the degenerate curve.  For sufficiently positive sections of del Pezzo fibrations we perform a ``controlled'' degeneration so that the resulting curve has exactly two irreducible components and represents a smooth point of the moduli space of stable maps. This enables us to study properties of irreducible components of the space of sections based on induction on the height of sections, and we use this strategy to study multiple examples in Section~\ref{sect:examples}.

We also give a refinement of Batyrev's heuristic in the setting of del Pezzo fibrations in Section~\ref{sect:Manin}. This is the first study of such a heuristic in the setting of non-trivial fibrations.

\subsection{Main results}

We next describe our main results in more detail.

\subsubsection{Classifying sections}

Let $\pi : \mathcal X \to \mathbb P^1$ be a del Pezzo fibration over $\mathbb P^1$. There are two types of sections $C$ of $\pi$:
\begin{itemize}
\item sections which deform to cover $\mathcal X$, or;
\item sections which do not deform to cover $\mathcal X$.
\end{itemize}
Sections of the first type have nice deformation properties. One can show that a general member of an irreducible component of  $\Sec(\mathcal{X}/\mathbb{P}^{1})$ parametrizing such sections is free and in particular the component has the expected dimension. 

Sections of the second type are pathological in some sense. For example, it is possible that a component parametrizing such sections has dimension higher than expected dimension or is everywhere non-reduced.  For this reason such sections should be considered as part of the ``exceptional set'' in Manin's Conjecture and it is important to have some control on these components for Batyrev's heuristic.

Our first main theorem describes the components of $\Sec(\mathcal{X}/\mathbb{P}^{1})$ which parametrize a non-dominant family of sections.

\begin{theo} \label{theo:maintheorem1}
Let $\pi: \mathcal{X} \to \mathbb{P}^{1}$ be a del Pezzo fibration such that $-K_{\mathcal X/\mathbb P^1}$ is relatively nef. Then there is a proper closed subset $V \subsetneq \mathcal{X}$ such that any component $M \subset \Sec(\mathcal{X}/\mathbb{P}^{1})$ parametrizing a non-dominant family of sections will parametrize sections contained in $V$.
\end{theo}


Our proof of Theorem \ref{theo:maintheorem1} is constructive: we show that $V$ is the union of the subvarieties of $\mathcal{X}$ which have larger generic Fujita invariant with a locus swept out by low height sections. 

\begin{rema}
\cite[3.3 Theorem]{Corti} shows that any del Pezzo surface over the function field of $\mathbb{P}^{1}$ admits an integral model with Gorenstein terminal singularities such that $-K_{\mathcal{X}/\mathbb{P}^{1}}$ is relatively nef.  Thus every del Pezzo fibration admits a birational model where Theorem \ref{theo:maintheorem1} applies.
\end{rema}

For dominant families of sections, the key question is the number of components of $\Sec(\mathcal{X}/\mathbb{P}^{1})$ parametrizing sections of a given nef numerical class $\alpha$. In Batyrev's heuristic, it is important to control this number when the degree of $\alpha$ is sufficiently large. 

\begin{conj}\label{conj:GMC_intro}
Let $\pi : \mathcal X \to \mathbb P^1$ be a del Pezzo fibration over $\mathbb P^1$ such that every smooth del Pezzo fiber has degree $\geq 2$. Let $\mathcal X_\eta$ denote the generic fiber of $\pi$ and let $E_1, \cdots, E_r$'s be generators of the pseudo-effective cone of divisors for $\mathcal X_\eta$. Then 
\begin{enumerate} 
\item for a nef numerical class $\alpha$ of sections the number of dominant components of $\Sec(\mathcal{X}/\mathbb{P}^{1})$ parametrizing sections of class $\alpha$ is uniformly bounded; 
\item there exists a constant $C > 0$ such that if $\alpha$ is a nef numerical class of sections such that $E_i \cdot \alpha \geq C$ for any $i$ then the number of dominant components of $\Sec(\mathcal{X}/\mathbb{P}^{1})$ parametrizing sections of class $\alpha$ is equal to $|\mathrm{Br}(\mathcal X)|$.
\end{enumerate}
\end{conj}

We will analyze this conjecture using the inductive approach of \cite{HRS04} based on breaking and gluing rational curves. 
Our second main theorem shows that we can break curves while remaining in the setting of free curves so that the outcome of our controlled Bend-and-Break will be a smooth point of the moduli space:

\begin{theo}[Movable Bend-and-Break for sections] \label{theo:maintheorem2}
Let $\pi: \mathcal{X} \to \mathbb{P}^{1}$ be a del Pezzo fibration such that $-K_{\mathcal{X}/\mathbb{P}^{1}}$ is relatively ample.  There is a constant $Q(\mathcal{X})$ satisfying the following property.  Suppose that $M \subset \Sec(\mathcal{X}/\mathbb{P}^{1})$ is a component that parametrizes a dominant family of sections $C$ satisfying $-K_{\mathcal{X}/\mathbb{P}^{1}} \cdot C \geq Q(\mathcal{X})$.  Then the closure of $M$ in $\overline{M}_{0,0}(\mathcal{X})$ contains a point representing a stable map whose domain has exactly two components each mapping birationally onto a free curve.
\end{theo}

Here $Q(\mathcal{X})$ is an explicit constant determined by the behavior of the Fujita invariant and low degree sections.

\begin{rema}
\cite[1.10 Theorem]{Corti} shows that any del Pezzo fibration whose generic fiber has degree $\geq 3$ will admit a birational model which has Gorenstein terminal singularities and a relatively ample anticanonical divisor.
\end{rema}


Based on the results above, we can use the following strategy to classify sections of a del Pezzo fibration with $-K_{\mathcal{X}/\mathbb{P}^{1}}$ relatively ample:
\begin{enumerate}
\item Classify all non-dominant families of sections using Theorem \ref{theo:maintheorem1}.
\item Classify all dominant families of sections of small height using the explicit geometry of the fibration.
\item Classify all dominant families of sections of large height by induction using Theorem \ref{theo:maintheorem2}.
\end{enumerate}
We illustrate this strategy in several examples:

\begin{exam} \label{intro:cubicexample}
Let $X_{3}$ denote a smooth cubic threefold in $\mathbb{P}^{4}$.  By taking a general pencil of hyperplane sections and resolving the base locus we obtain a del Pezzo fibration $\pi: \mathcal{X} \to \mathbb{P}^{1}$.  In Example \ref{exam:blowupcubic} we will use the strategy above to show that $\Sec(\mathcal{X}/\mathbb{P}^{1})$ admits a unique component representing sections of any height $\geq -1$ and that each of these components has the expected dimension.

A similar argument works for other del Pezzo fibrations constructed by blowing up a Fano $3$-fold, and we give several examples of this type.
\end{exam}

\begin{exam} \label{exam:diagonalcubicexample}
Consider the smooth threefold $\mathcal{X} \subset \mathbb{P}^{1}_{s,t} \times \mathbb{P}^{3}_{x,y,z,w}$ defined by the equation $sx^{3} + ty^{3} + (t+s)z^{3} + (t+2s)w^{3}=0$ equipped with the cubic surface fibration $\pi: \mathcal{X} \to \mathbb{P}^{1}$.  Using the above strategy one can show that for any integer $d \geq -1$ there is a unique component of $\Sec(\mathcal{X}/\mathbb{P}^{1})$ with height $d$.  Since the Galois action on $\mathrm{Pic}(\mathcal X_{\overline{\eta}})$ is $(\mathbb Z/3\mathbb Z)^3$ and under this action the $27$ lines split into orbits of size $(9,9,9)$ it follows from \cite[the Appendix, No. 71]{Jahnel14} that $\mathcal{X}_{\eta}$ has Brauer group $\mathbb{Z}/3\mathbb{Z}$. Thus this example demonstrates that it is the Brauer group of the total space and not the Brauer group of the generic fiber which plays a role in Geometric Manin's Conjecture over a complex curve (see Section \ref{sect:countingcomponents}).
\end{exam}


\subsubsection{Upper bounds in Manin's Conjecture and Geometric Manin's Conjecture}

The following invariant plays a central role in the development of Manin's Conjecture:

\begin{defi}
Let $X$ be a smooth projective variety over a field of characteristic $0$.  Let $L$ be a big and nef $\mathbb{Q}$-Cartier divisor on $X$.  We define the Fujita invariant, or the $a$-invariant, to be
\begin{equation} \label{eq: ainv}
a(X, L) = \min \{ t \in \mathbb R \mid K_{X} + tL \in \overline{\mathrm{Eff}}^1(X)\}.
\end{equation}
When $L$ is nef but not big, we formally set $a(X, L) = +\infty$.

When $X$ is singular, we define the Fujita invariant as the Fujita invariant of the pullback of $L$ to any smooth model. This is well-defined because of \cite[Proposition 2.7]{HTT15}.
\end{defi}

Just as in \cite{LST18}, we expect the ``exceptional set'' of sections in Manin's Conjecture to be controlled by the generic $a$-invariant.  The following definition describes the sections which are allowed to contribute to the counting function.

\begin{defi}
\label{def:Manincomp}
We say that a component $M \subset \Sec(\mathcal{X}/\mathbb{P}^{1})$ is a Manin component if for the universal family $\mathcal{U} \to M$ the evaluation map $\mathcal{U} \to \mathcal X$ does not factor rationally through any proper subvariety $Y$ satisfying
\begin{equation*}
a(Y_{\eta},-K_{\mathcal{X}/\mathbb{P}^{1}}) \geq a(\mathcal{X}_{\eta},-K_{\mathcal{X}/\mathbb{P}^{1}}).
\end{equation*}
We let $\Manin_{i}$ denote the set of Manin components that parametrize sections $C$ satisfying $-K_{\mathcal{X}/\mathbb{P}^{1}} \cdot C = i$.
\end{defi}

\begin{rema}
The above formulation is slightly different from the formulation in \cite{LST18} which predicts the exceptional set for rational points in the general setting.  (The difference is in how we handle subvarieties with the same $a$-value.)  However, in our situation the difference is negligible for the asymptotic formula of the counting function in Geometric Manin's Conjecture, so for the sake of simplicity we will use the definition of a Manin Component given above.
\end{rema}

The counting function encodes the number and dimension of Manin components representing sections of height at most $d$.

\begin{defi}
Fix a real number $q>1$.  For any positive integer $d$ define
\begin{equation*}
N(\mathcal{X},-K_{\mathcal{X}/\mathbb{P}^{1}},q,d) := \sum_{i = 1}^{d} \sum_{M \in \Manin_{i}} q^{\dim M}.
\end{equation*}
\end{defi}



The exceptional set in Geometric Manin's Conjecture is contained in the closed set $V$ of Theorem \ref{theo:maintheorem1}.  The theorem implies that after removing the sections in $V$ every family has the expected dimension, giving the correct exponential term in the asymptotic formula for $N(\mathcal{X},-K_{\mathcal{X}/\mathbb{P}^{1}},q,d)$.  In order to control the subexponential term, we apply ideas from \cite{LT17} to prove a special case of a conjectural heuristic of Batyrev:

\begin{coro} \label{coro:batyrev}
Let $\pi: \mathcal{X} \to \mathbb{P}^{1}$ be a del Pezzo fibration such that $-K_{\mathcal{X}/\mathbb{P}^{1}}$ is relatively ample.  There is a polynomial $P(d)$ which is an upper bound for the number of components of $\Sec(\mathcal{X}/\mathbb{P}^{1})$ that parametrize sections $C$ satisfying
\begin{equation*}
-K_{\mathcal{X}/\mathbb{P}^{1}} \cdot C \leq d.
\end{equation*}
\end{coro}

Corollary \ref{coro:batyrev} implies an upper bound on the counting function of the expected form
\begin{equation*}
N(\mathcal{X},-K_{\mathcal{X}/\mathbb{P}^{1}},q,d) = O(q^{d}d^{r})
\end{equation*}
for some positive integer $r$.

Of course, whenever one can completely classify sections one obtains a precise asymptotic formula for the counting function. Assuming Conjecture \ref{conj:GMC_intro}, we develop a refined heuristic of Batyrev showing the asymptotic formula for the above formal counting function:

\begin{theo} \label{theo:asymptoticformula_intro}
Let $\pi: \mathcal{X} \to \mathbb{P}^{1}$ be a del Pezzo fibration such that $\mathcal X$ is smooth, $-K_{\mathcal{X}/\mathbb{P}^{1}}$ is relatively ample, and the general fiber is a del Pezzo surface of degree $\geq 2$ that is not $\mathbb{P}^{2}$ or $\mathbb{P}^{1} \times \mathbb{P}^{1}$.
Assume that Conjecture \ref{conj:GMC_intro} holds.  Then 
\begin{equation*}
N(\mathcal{X},-K_{\mathcal{X}/\mathbb{P}^{1}},q,d) \mathrel{\mathop{\sim}_{\mathrm{d \to \infty}}} Cq^{d} d^{\rho(\mathcal{X}_{\eta})-1},
\end{equation*}
for some $C>0$.
\end{theo} 

In this way we can analyze the counting function for some low degree del Pezzo surfaces over $k(\mathbb{P}^{1})$.

\subsubsection{Stabilization of the Abel-Jacobi map}

For a del Pezzo fibration $\pi: \mathcal{X} \to \mathbb{P}^{1}$ over $\mathbb{C}$ such that $\mathcal X$ is smooth, we let $\IJ(\mathcal{X})$ denote the intermediate Jacobian of $\mathcal{X}$.  For any reduced component $M \subset \Sec(\mathcal{X}/\mathbb{P}^{1})$ the universal family of curves over $M$ induces a rational map $\AJ_{M}: M \dashrightarrow \IJ(\mathcal{X})$.  Using our main theorems we show that the Stein factorizations of the Abel-Jacobi maps ``stabilize'' as we increase the height.

\begin{theo} \label{theo:maintheoremaj}
Let $\pi: \mathcal{X} \to \mathbb{P}^{1}$ be a del Pezzo fibration such that $-K_{\mathcal{X}/\mathbb{P}^{1}}$ is relatively ample, $\mathcal X$ is smooth, and the general fiber is a del Pezzo surface with degree $\geq 3$.  Consider the set of components $M \subset \Sec(\mathcal{X}/\mathbb{P}^{1})$ such that the Abel-Jacobi map $\AJ_{M}: M \dashrightarrow \IJ(\mathcal{X})$ is dominant.  There are only finitely many morphisms $\{ Z_{i} \to \IJ(\mathcal{X}) \}$ which occur as a Stein factorization of a resolution of a projective compactification of one of these $\AJ_{M}$.
\end{theo}

\begin{exam}
Suppose that $\pi: \mathcal{X} \to \mathbb{P}^{1}$ is constructed by taking a general pencil of hyperplane sections of a Fano threefold of Picard rank $1$, index $2$, and degree $5$.  In this case $\IJ(\mathcal{X})$ is an elliptic curve.  We prove that the Abel-Jacobi map for every family of sections of height $\geq 1$ is dominant with connected fibers, so that the finite part of the Stein factorization is always the identity map.  (We expect that in this example the Abel-Jacobi map coincides with the MRC fibration for every family of sections.)
\end{exam}

Our original motivation for Theorem \ref{theo:maintheoremaj} was to analyze whether the Abel-Jacobi map relates to other possible types of stability results for components of $\Sec(\mathcal{X}/\mathbb{P}^{1})$:
\begin{itemize}
\item Does the MRC fibration for components of $\Sec(\mathcal{X}/\mathbb{P}^{1})$ stabilize as the degree increases? (\cite{dJS04}, \cite{Voisin13}, \cite{Voisin15} provide negative answers to this question in related settings.)
\item Do components of $\Sec(\mathcal{X}/\mathbb{P}^{1})$ exhibit cohomological stability? (See \cite{EV05}, \cite{Bou09}, \cite{CEF14} for positive results and \cite{Cas04} for some negative results in related settings.) 
\end{itemize}
However, Examples \ref{exam: fibrationoverdegree2} and \ref{exam: fibrationoverdegree1} show that for components of $\Sec(\mathcal{X}/\mathbb{P}^{1})$ neither the MRC fibration nor the Albanese map needs to coincide with the Abel-Jacobi map.  

\begin{exam}
Let $S$ be a del Pezzo surface of degree $2$ equipped with a conic fibration $g: S \to \mathbb{P}^{1}$.  Let $\mathcal{X}$ be a conic bundle over $S$ whose discriminant locus $D$ is a very general element of $|-2K_{S}|$.  By composing with $g$ we obtain a degree $4$ del Pezzo surface fibration $\pi: \mathcal{X} \to \mathbb{P}^{1}$.  In this case $\IJ(\mathcal{X})$ is a $2$-dimensional abelian variety. 

We show that:
\begin{itemize}
\item If the Abel-Jacobi map for a component of $\Sec(\mathcal{X}/\mathbb{P}^{1})$ is dominant, then it is not birational to the MRC fibration. (This follows from results of \cite{HKT16b} and \cite{Voisin13}.)
\item There are components of $\Sec(\mathcal{X}/\mathbb{P}^{1})$ whose Albanese variety has dimension larger than $2$.
\end{itemize}
In particular the dimension of the intermediate Jacobian does not provide an a priori bound on the dimension of $H^{1}(\overline{M},\mathbb{C})$ for the resolution $\overline{M}$ of a projective closure of a component $M$ of $\Sec(\mathcal{X}/\mathbb{P}^{1})$.
\end{exam}

\subsubsection{Gromov-Witten theory}

The arguments we use to prove Theorem \ref{theo:maintheorem2} also show that certain Gromov-Witten invariants for classes of sections are enumerative. 

\begin{prop} \label{prop:introgw}
Let $\pi: \mathcal{X} \to \mathbb{P}^{1}$ be a del Pezzo fibration such that $-K_{\mathcal{X}/\mathbb{P}^{1}}$ is relatively ample and $\mathcal X$ is smooth. Let $F$ be a general fiber of $\pi$. There is a constant $C(\mathcal{X})$ such that the following is true.  Fix an integer $n \geq C(\mathcal{X})$.  Let $\beta \in N_{1}(\mathcal{X})_{\mathbb{Z}}$ denote a curve class satisfying $-K_{\mathcal{X}/\mathbb{P}^{1}} \cdot \beta = 2n-2$ and $F \cdot \beta = 1$ for a general fiber $F$ of $\pi$.  Then the GW invariant $\langle [pt]^{n} \rangle_{0,n}^{\mathcal{X},\beta}$ is enumerative: it counts the number of rational curves of class $\beta$ through $n$ general points of $\mathcal{X}$.
\end{prop}

Proposition \ref{prop:gwenumerative2} gives a similar result for classes of odd height and Proposition \ref{prop:balancednormalbundle} gives conditions which guarantee that these GW-invariants do not vanish.  These results complement \cite[Theorem 4.1]{Tian12} which proves the non-vanishing of other GW-invariants for del Pezzo fibrations over $\mathbb{P}^{1}$.

\subsection{Relations with other work}

Here we discuss relations of the current paper to other work in the subject.

\subsubsection{Past work}

As mentioned above, Batyrev developed a heuristic (``Batyrev's dream'') for Manin's Conjecture over finite fields in \cite{Bat88} based on three assumptions. (This heuristic is explained in \cite[Section 4.7]{Tsc09} and \cite[Section 1.2]{Bou}.) This perspective leads to the formulation of the Batyrev-Manin Conjecture in \cite{BM}. We investigated this heuristic for Fano varieties in characteristic $0$ in \cite{LT17} and confirmed that the first assumption is always valid. However, the second assumption fails in general mainly due to the presence of Zariski dense thin exceptional sets which are intensively studied in the case of rational points over number fields. (\cite{HTT15},  \cite{LTT14}, \cite{HJ16}, \cite{LT16}, \cite{Sen17b}, and \cite{LST18}). Moreover in \cite{LT17}, we proposed Geometric Manin's Conjecture which predicts that for a Fano variety $X$ the number of irreducible components of $\mathrm{Mor}(\mathbb P^1, X, \alpha)$ which contribute to the counting function is constant when the degree of $\alpha$ is sufficiently large.  (The appellation ``Geometric Manin's Conjecture'' comes from the fact that this conjecture over complex function fields is derived from Manin's Conjecture over global function fields, just as the geometric Langlands Program over complex function fields is derived from the Langlands Program over global function fields.)
In this paper, we extend this perspective to the settings of del Pezzo fibrations over $\mathbb P^1$. See Section~\ref{sect:Manin} for more details.

The classification of components of $\mathrm{Mor}(\mathbb P^1, X)$ for a smooth Fano variety has a long and rich history. This has been done for homogeneous spaces in \cite{Thomsen98} and \cite{KP01}.  There was also a pioneering work \cite{HRS04} analyzing the inductive structure using Mori's Bend-and-Break when $X$ is a smooth Fano hypersurface. This has been subsequently generalized in \cite{BK13} and \cite{RY16} completing the analysis for most Fano hypersurfaces. (See \cite{BV16} and \cite{BS20} for another approach using an idea from analytic number theory.) Toric varieties have been analyzed in \cite{Bou16} and the moduli spaces of vector bundles on curves have been studied in \cite{Cas04}. Smooth del Pezzo surfaces have been handled by Testa in \cite{Testa09}, and smooth Fano threefolds have been the focus of many studies (\cite{CS09}, \cite{Cas04}, \cite{LT17}, and \cite{LT18}).  

The stabilization of MRC fibrations/Abel-Jacobi mapping has been also well-studied. In positive directions, there are some results \cite{HRS02} for cubic threefolds, \cite{Cas04} for moduli spaces of vector bundles on curves, \cite{dJHS11} for $2$-Fano fibrations, and \cite{Zhu19} for homogenous fibrations. \cite{dJS04} provides examples where MRC fibrations do not stabilize and \cite{Voisin13} characterizes when the Abel-Jacobi mappings coincide with the MRC fibrations for rationally connected threefolds.

Del Pezzo fibrations have been studied extensively due to their prominent role in the minimal model program for threefolds.  \cite{GHS03} guarantees that a del Pezzo fibration admits a section. There is a large body of literature on the Weak Approximation Conjecture for sections of del Pezzo fibrations, see e.g.~\cite{KMM92}, \cite{CTG04}, \cite{HT06}, \cite{HT08}, \cite{Xu12}, \cite{Xu12b}, \cite{Knecht13}, \cite{Tian15}, and \cite{STZ18}.  The Abel-Jacobi map for sections has been studied for quadric surface fibrations by \cite{HT12} and for degree $4$ del Pezzo fibrations by \cite{HT14}.  \cite{Tian12} has studied Gromov-Witten invariants for del Pezzo fibrations over $\mathbb{P}^{1}$. 

\subsubsection{Comparison to the sequel} 
In the sequel paper \cite{LT21}, we study sections of del Pezzo fibrations over higher genus curves.  Using ideas in the current paper, we will generalize Theorem~\ref{theo:maintheorem1} and Theorem~\ref{theo:maintheorem2} to the case when the base curve has a higher genus. However, the various constants (such as $Q(\mathcal X)$ from Theorem \ref{theo:maintheorem2}) are rather impractical compared to the constants in the current paper.  In this paper we are able to analyze multiple examples of non-trivial del Pezzo fibrations, while in \cite{LT21} the only examples are trivial del Pezzo fibrations over a higher genus curve. 
In addition, the applications to stabilization of Abel-Jacobi mappings and to enumerativity of Gromov-Witten invariants are only discussed in this paper (although similar arguments will work for del Pezzo fibrations over higher genus curves).

\

\bigskip

\noindent
{\bf Acknowledgements:}
The authors thank Brendan Hassett for an enlightening discussion and in particular for suggesting Examples \ref{exam: fibrationoverdegree2} and \ref{exam: fibrationoverdegree1}.  They thank John Lesieutre and Yohsuke Matsuzawa for helpful conversations about dynamics and Asher Auel for discussions regarding Example \ref{exam: fibrationoverdegree2}. We thank Eric Jovinelly for pointing out a mistake in Example~\ref{exam: fibrationoverdegree1}. We thank the referees for constructive criticisms which improve the results in the paper. We also thank the referees for detailed suggestions which significantly improved the exposition of the paper. Part of this work was done at an AIM SQuaRE workshop, and the authors thank AIM for the excellent working environment.  

Brian Lehmann was supported by NSF grant 1600875.  Sho Tanimoto was partially supported by MEXT Japan, Leading Initiative for Excellent Young Researchers (LEADER), by Inamori Foundation, by JSPS KAKENHI Early-Career Scientists Grant number 19K14512, by JSPS Bilateral Joint Research Projects Grant number JPJSBP120219935, and by JST FOREST program Grant number JPMJFR212Z.

\section{Preliminaries}

Fix a field $k$ which is algebraically closed and of characteristic $0$.  In this paper our ground field will usually be $k$ or $k(\mathbb{P}^{1})$.  A variety is a reduced irreducible separated scheme of finite type over the ground field.  Unless otherwise we state, a component means an irreducible component. When we take a component of a scheme, we always endow it with its reduced structure.

Throughout the paper we will use $\sim_{rat}$ to denote rational equivalence of cycles and $\sim_{alg}$ to denote algebraic equivalence of cycles.

Let $X$ be a projective separated scheme of finite type over the ground field. We will let $N^{1}(X)_{\mathbb{R}}$ denote the space of $\mathbb{R}$-Cartier divisors up to numerical equivalence on $X$ and let $\Eff^{1}(X)$ and $\Nef^{1}(X)$ denote respectively the pseudo-effective and nef cones of divisors.  Dually, $N_{1}(X)_{\mathbb{R}}$ denotes the space of $\mathbb{R}$-curves up to numerical equivalence and $\Eff_{1}(X)$ and $\Nef_{1}(X)$ denote respectively the pseudo-effective and nef cones of curves. We denote the lattices generated by integral cycles by $N^1(X)_{\mathbb Z} \subset N^1(X)_{\mathbb{R}}$ and $N_1(X)_{\mathbb Z} \subset N_1(X)_{\mathbb{R}}$. We define $\Nef_{1}(X)_{\mathbb Z} := \Nef_{1}(X) \cap N_1(X)_{\mathbb Z}$.

We say that a reduced irreducible curve $C$ is movable on $X$ if $C$ is a member of a family of curves which dominates $X$.

\subsection{Free rational curves on Gorenstein terminal threefolds}

We recall some deformation theory of rational curves on threefolds with Gorenstein terminal singularities.

\begin{defi}
Let $X$ be a Gorenstein threefold with only terminal singularities.
Let $f : \mathbb P^1 \to X$ be a rational curve. We say $f$ is free if the image $f(\mathbb P^1)$ is contained in the smooth locus of $X$ and we have
\[
f^*T_X = \mathcal O(a_1) \oplus \mathcal O(a_2) \oplus \mathcal O(a_3)
\]
with $0 \leq a_1 \leq a_2 \leq a_3$.
\end{defi}

It is well-known that for a smooth quasi-projective variety $Y$ and any component $M \subset \mathrm{Mor}(\mathbb P^1, Y)$ we have
\[
\dim M \geq -K_Y \cdot C + \dim Y
\]
where $C \in M$ is a general member. (See \cite[Theorem II.1.2]{Kollar} for this statement.)
We have the following lemma for Gorenstein terminal threefolds:

\begin{lemm}
\label{lemm:expected}
Let $X$ be a Gorenstein terminal threefold.
Let $M$ be a component of $\mathrm{Mor}(\mathbb P^1, X)$. Then we have
\[
\dim M \geq -K_X \cdot C +3,
\]
where $C \in M$ is a general member.
\end{lemm}

\begin{proof}
\cite[Theorem II.1.3]{Kollar} shows that our assertion holds for locally complete intersection varieties, so we just need to verify this condition for $X$.

Since $X$ has Gorenstein terminal singularities, it has isolated cDV singularities (\cite[Corollary 5.38]{KM98}).  It is well-known that isolated cDV singularities are analytically isomorphic to hypersurface singularities.  For any local ring which is the homomorphic image of a regular local ring the complete intersection property can be detected on the completion (see e.g.~the discussion in Section 18.5 p.462 in \cite{Eisenbud}), and we deduce that $X$ is locally complete intersection in the Zariski topology.
\end{proof}


\begin{lemm}
\label{lemm:dominantimpliesfree}
Let $X$ be a Gorenstein threefold with terminal singularities. Suppose that $M \subset \mathrm{Mor}(\mathbb P^1, X)$ is a component parametrizing a dominant family of curves.  Then a general member of $M$ is free.  
\end{lemm}

\begin{proof}
Suppose that all members of $M$ pass through singular points of $X$.
Let $\phi: Y \to X$ be a resolution. Let $C$ be a general member of $M$ and let $C'$ denote its strict transform on $Y$. Then $C'$ must meet with a $\phi$-exceptional divisor. However, the terminal condition implies that
\[
-K_Y \cdot C < -K_X \cdot C'.
\]
Since deformations of $C'$ dominate $Y$, the family of deformations of $C'$ has the expected dimension.  We conclude that $\dim(M) = -K_{Y} \cdot C' + 3$.  This contradicts Lemma \ref{lemm:expected} which shows that $\dim(M) \geq -K_{X} \cdot C + 3$.
\end{proof}

\subsection{Fano fibrations}

In this paper, a Fano fibration will always be an integral model with mild singularities for a Fano variety over $k(\mathbb{P}^{1})$.  More precisely:

\begin{defi}
A Fano fibration is a morphism $\pi: \mathcal{X} \to \mathbb{P}^{1}$ satisfying the following properties:
\begin{enumerate}
\item $\mathcal{X}$ is a normal projective variety with Gorenstein terminal singularities,
\item $f$ is an algebraic fiber space, and
\item the general fiber of $\pi$ is a smooth Fano variety.
\end{enumerate}
We will always denote the generic point of $\mathbb{P}^{1}$ by $\eta$ and the generic fiber of $\pi$ by $\mathcal{X}_{\eta}$.
\end{defi}

Fano fibrations satisfy the following important properties:

\begin{itemize}
\item Every fiber of $f$ is rationally chain connected (see \cite[IV.3.5 Corollary]{Kollar}).
\item As we vary over all smooth fibers $F$ the dimension of $N^{1}(F)_{\mathbb{R}}$ is constant.  Indeed, the dimension of $N^{1}(F)_{\mathbb{R}}$ does not change under a base change of algebraically closed fields, so one may assume that our ground field $k$ is an algebraic closure of a finitely generated field over $\mathbb Q$. Such a field admits an embedding into $\mathbb C$, so we may assume that $k = \mathbb C$. Now the previous property combined with \cite[2.1]{KMM92b} shows that $F$ is rationally connected.  In particular $\dim(N^{1}(F)_{\mathbb{R}}) = \dim H^{2}(X,\mathbb{R})$ and the latter is constant in smooth families by \cite[pg.~154 Proposition]{Ehresmann95}.
\end{itemize}

Furthermore, Fano fibrations have the nice property that the positivity of divisors on the generic fiber is related to positivity of divisors on a general fiber.

\begin{lemm}[\cite{Wisniewski09} Theorem 1, \cite{dFH11} Theorem 6.8 and preceding discussion] \label{lemm:relativeprops}
Let $\pi: \mathcal{X} \to \mathbb{P}^{1}$ be a Fano fibration.  Let $L$ be a $\mathbb{Q}$-Cartier divisor on $\mathcal X$.  Then the following are equivalent:
\begin{enumerate}
\item $L|_{F}$ is ample for some smooth Fano fiber $F$.
\item $L|_{F}$ is ample for all smooth Fano fibers $F$.
\item $L|_{\mathcal{X}_{\eta}}$ is ample.
\end{enumerate}
The analogous statement is true for nefness, for bigness, and for pseudo-effectiveness.
\end{lemm}



\subsection{Height functions}

Let $\pi: \mathcal{X} \to \mathbb{P}^{1}$ be a projective morphism of finite type of separated schemes over the ground field. Then the choice of an ample $\mathbb{Q}$-Cartier divisor on the generic fiber $\mathcal{X}_{\eta}$ will define a height function on sections of $\pi$.  It is well-known that a Northcott property holds in this setting:

\begin{lemm} \label{lemm:northcott}
Let $\pi: \mathcal{X} \to \mathbb{P}^{1}$ be a dominant projective morphism of finite type of separated schemes over the ground field.  Let $L$ be a $\mathbb{Q}$-Cartier divisor on $\mathcal{X}$ whose restriction to a general fiber is ample.  For any fixed constant $\gamma$, the set of classes of sections $C$ satisfying $L \cdot C \leq \gamma$ lies in a bounded subset of $N_{1}(\mathcal{X})_{\mathbb{R}}$.
\end{lemm}

\begin{proof}
The proof is by induction on the dimension of the generic fiber of $\pi$.  The base case is when the relative dimension is $0$.  Then the set of sections is finite and the statement is obviously true.

Now suppose that the generic fiber of $\pi$ has dimension $n > 0$.  Let $\{ \mathcal{X}_{i} \}_{i=1}^{r}$ denote the set of  (reduced) irreducible components of $\mathcal{X}$ which map dominantly to $\mathbb{P}^{1}$.  The inclusions induce a linear map $\oplus N_{1}(\mathcal{X}_{i})_{\mathbb{R}} \to N_{1}(\mathcal{X})_{\mathbb{R}}$ whose image contains the numerical classes of all sections of $\pi$.  Since any section of $\pi$ will also be a section of $\pi|_{\mathcal{X}_{i}}$ for some $i$, it suffices to prove the statement for each irreducible component $\mathcal{X}_{i}$.  Thus we may assume that $\mathcal{X}$ is a projective variety of dimension $n$ that maps dominantly to $\mathbb{P}^{1}$.

By assumption there is an open subset $U \subset \mathbb{P}^{1}$ such that $L|_{\pi^{-1}(U)}$ is $\pi|_{\pi^{-1}(U)}$-relatively ample.  
In turn this implies that $L$ is $\pi$-relatively big over $\mathbb{P}^{1}$.  Let $F$ denote a general fiber of $\pi$ and choose a positive integer $m$ such that $L + mF$ is big on $\mathcal{X}$.  We can write $L+mF \equiv H + E$ for some ample $\mathbb{Q}$-divisor $H$ and some effective $\mathbb{Q}$-divisor $E$.  Then every section $C$ in the statement of the theorem either satisfies:
\begin{itemize}
\item $E \cdot C \geq 0$, in which case $H \cdot C \leq \gamma + m$, or
\item $E \cdot C < 0$.
\end{itemize}
The numerical classes of sections of the first type lie in a bounded subset of $N_{1}(\mathcal{X})_{\mathbb{R}}$.  Sections of the second type must be contained in $\Supp(E)$.  We then conclude the desired statement by the induction hypothesis applied to $\Supp(E)$ equipped with the $\mathbb{Q}$-Cartier divisor $L|_{E}$.
\end{proof}

One consequence of Lemma \ref{lemm:northcott} is that there is a lower bound on the possible values of $L \cdot C$ as we vary $C$ over all sections.

\begin{defi}
Let $\pi: \mathcal{X} \to \mathbb{P}^{1}$ be a Fano fibration and let $L$ be a $\mathbb{Q}$-Cartier divisor on $\mathcal{X}$ whose restriction to the generic fiber is ample.  We define $\neg(\mathcal{X},L)$ to be the smallest value of $L \cdot C$ as we vary $C$ over all sections of $\pi: \mathcal{X} \to \mathbb{P}^{1}$.
\end{defi}

\subsection{Spaces of sections} \label{sect:spacesofsections}

Let $\pi : \mathcal X \rightarrow \mathbb P^1$ be a Fano fibration.  We let $\Sec(\mathcal{X}/\mathbb{P}^{1})$ denote the open subset of $\Hilb(\mathcal{X})$ parametrizing sections of $\pi$.  Using the functorial definitions, we see that there is an inclusion $\Sec(\mathcal{X}/\mathbb{P}^{1}) \subset \overline{M}_{0,0}(\mathcal{X})$.  In particular, any component $M$ of $\Sec(\mathcal{X}/\mathbb{P}^{1})$ embeds as a dense open subset of a unique component $\overline{M}$ of $\overline{M}_{0,0}(\mathcal{X})$.

The expected dimension of a component $M \subset \Sec(\mathcal{X}/\mathbb{P}^{1})$ is
\begin{equation*}
-K_{\mathcal X/\mathbb P^1} \cdot C + (\dim \mathcal X-1)
\end{equation*}
where $C$ denotes a general section parametrized by $M$.  Lemma~\ref{lemm:expected} shows that the expected dimension is a lower bound for $\dim(M)$ whenever $\dim(\mathcal{X}) \leq 3$.  
If $M$ parametrizes a family of free sections then the actual dimension agrees with the expected dimension.

\begin{lemm} \label{lemm:normalbundleestimate}
Let $\pi: \mathcal{X} \to \mathbb{P}^{1}$ be a Fano fibration with $\dim \mathcal X \leq 3$ and let $n$ be a positive integer.  Suppose that $C$ is a section such that any $n$ general points of $\mathcal{X}$ are contained in a deformation of $C$.  If we let $C'$ be a general deformation of $C$ and write $N_{C'/X} = \oplus \mathcal{O}_{\mathbb{P}^{1}}(a_{i})$ then each $a_{i} \geq n-1$. Moreover we can find such a section $C'$ through any $n$ general points.

Conversely, if $C$ is a free section of $\mathcal{X}$ with $N_{C/\mathcal X} = \oplus \mathcal{O}_{\mathbb{P}^{1}}(a_{i})$ and each $a_{i} \geq n-1$, then any $n$ general points of $\mathcal{X}$ are contained in some deformation of $C$.
\end{lemm}

\begin{proof}
This follows from the commutative diagram (2) in \cite[p. 240]{Shen12}.

First suppose that some deformation of $C$ contains $n$ general points. By Lemma \ref{lemm:dominantimpliesfree} we know that a general deformation of $C$ is free. Consider the sublocus of $\Sec(\mathcal{X}/\mathbb{P}^{1})$ parametrizing free deformations of $C$ through a fixed set of $n-1$ general points. We know that there is some component of this locus whose universal family maps dominantly onto $\mathcal{X}$.  It follows that the differential map of tangent spaces for the universal family is generically full rank. Thus the commutative diagram (2) in \cite{Shen12} implies that $a_i \geq n-1$.

Conversely suppose that $a_i \geq n-1$ for some free section $C$. We prove our assertion by induction on $n$. For the base case when $n = 1$,  $C$ is free so that our assertion is clear. For the induction step, we suppose that the statement is true for $n$ and deduce the statement for $n+1$.  Assume that $a_i \geq n$ for every $i$. By the induction hypothesis, a general deformation of $C$ passes through $n$ general points. We fix $n$ general points and we consider the deformation space of curves passing through $n$ general points. Shen shows that the differential of the map from the universal family to $X$ is generically full rank. This means that this evaluation map must be dominant, so a general deformation of $C$ will pass through $n + 1$ general points.
\end{proof}

\subsection{Rational curves on del Pezzo surfaces}

In this section we review some results of \cite{Testa09} concerning the moduli spaces of rational curves on del Pezzo surfaces defined over an algebraically closed field of characteristic $0$.  When we discuss rational curves on del Pezzo surfaces we will always refer to the anticanonical polarization: lines mean $-K_{S}$-lines, conics mean $-K_{S}$-conics, etc.  We start by discussing the classification of rational curves of low anticanonical degree.

\begin{lemm} \label{lemm:conicscubicsdescription}
Let $S$ be a del Pezzo surface of degree $d$ defined over an algebraically closed field of characteristic $0$. Then: 
\begin{itemize}
\item The lines on $S$ are either $(-1)$-curves or singular elements of $|-K_{S}|$ (when $d=1$).
\item The conics on $S$ are either fibers of a conic fibration or the pullback of rational curves in the anticanonical linear series on a degree $2$ del Pezzo surface via a birational map (when $d=1,2$) or rational curves which lie in $|-2K_{F}|$ (when $d=1$).
\item Suppose $d \geq 2$.  Then the cubics on $S$ are pullbacks of elements of $|\mathcal{O}_{\mathbb{P}^{2}}(1)|$ under birational maps $\phi: S \to \mathbb{P}^{2}$ or the pullback of rational curves in the anticanonical linear series on a degree $3$ del Pezzo surface $S'$ via a birational map (when $d=2,3$).
\end{itemize}
In particular, suppose that $C$ is a cubic curve on a del Pezzo surface $S$ of degree between $2$ and $8$ that contains a point $p_{0}$ not lying on any line in $S$.  Then we can deform $C$ while keeping $p_{0}$ fixed to a stable map whose image is the union of a conic containing $p_{0}$ and a line.
\end{lemm}
\begin{proof}
(1) follows from \cite[Lemma 3.3]{BLRT21}.  (2) follows from \cite[Lemma 3.4]{BLRT21}. (3) follows from \cite[Lemma 4.3]{BLRT21} and its proof.  It only remains to prove the last claim.  If $C$ is smooth, then there is a birational map $\phi: S \to \mathbb{P}^{2}$ such that $C$ is the strict transform of a line.  Let $p_{0}'$ denote the $\phi$-image of $p_0$ let $q'$ denote a point where $\phi^{-1}$ is not defined.  Then $C$ deforms to the union of the strict transfrom of the line passing through $p_0'$ and $q'$ and the $(-1)$-curve over $q'$.  If $C$ is singular, then there is a birational map $\phi: S \to \mathbb{P}^{2}$ such that $C$ is the strict transform of a rational plane cubic $C$ on $\mathbb P^2$ passing through the image $p_0'$ and $6$ points where $\phi^{-1}$ is not defined. Such a curve degenerates to the union of the strict transform of a conic passing through $p_0'$ and $4$ of the points with the strict transform of a line passing through the other $2$ points. Thus our assertion follows.
\end{proof}

By combining with \cite[Proposition 2.4]{Testa09} we obtain the following.

\begin{lemm} \label{lemm:uniruledspaces}
Let $S$ be a del Pezzo surface of degree $\geq 3$.  Then the spaces of conics and cubics on $S$ are rational.
\end{lemm}

Testa's main theorem is the following.

\begin{theo}[\cite{Testa09} Theorem 5.1] 
Let $S$ be a del Pezzo surface of degree $\geq 2$.  Then for every numerical class $\alpha$, there is at most one component of $\overline{M}_{0,0}(S,\alpha)$ that generically parametrizes stable maps that are birational onto their image.
\end{theo}

The following lemma is a key step in his proof and we will need it as well. 

\begin{lemm}
\label{lemm:BBfordelPezzo}
Let $S$ be a del Pezzo surface. Let $C$ be a general free curve on $S$ such that $$-K_S \cdot C \geq 4.$$
Fix two general points $x_0, x_1$ of $C$. Then we can deform $C$ as a stable map, keeping $x_{0}$ and $x_{1}$ fixed, so that the image of the resulting stable map is the union of two free curves $C_0 \cup C_1$ such that $x_0 \in C_0$ and $x_1 \in C_1$.  

Moreover, by continuing this inductively we can deform $C$ as a stable map into a chain of free curves of anticanonical degree $\leq 3$ such that the end curves $C_0$ and $C_1$ contain $x_0, x_1$ respectively.
\end{lemm}

\begin{proof}
To see the first statement, consider the deformations of $C$ which contain a fixed set of  $-K_{S} \cdot C - 2$ general points on $S$.  By Bend-and-Break $C$ deforms into a broken curve through these points.  By \cite[Lemma 1.14]{Testa09} this broken curve must be the union of two free curves, and we deduce the statement. 

We prove the second statement by induction on $-K_S \cdot C = r$.  Let $\overline{M}$ denote the component of $\overline{M}_{0,0}(S)$ containing the stable map represented by the normalization $\mathbb{P}^{1} \to C$.  The first statement shows that $\overline{M}$ contains stable maps $g$ whose domain has two components and whose image in $S$ is $C_0\cup C_1$ where $C_{0}$ and $C_{1}$ are free curves and $x_i \in C_i$.  If both of the $C_{i}$ have degree $\leq 3$ then we are done.  

If exactly one of the $C_{i}$ has degree $\leq 3$, we may suppose without loss of generality that this component is $C_{0}$.  Choose a general point $x_{2}$ on $C_0$.  By choosing $x_0$ general in $S$, we may assume that $x_2$ is also general in $S$. By deforming $C_{1}$ (and deforming the attachment point to $C_{0}$), we may suppose that $C_{1}$ contains both $x_{1}$ and $x_{2}$ so that the attachment is $x_2$. Then we apply the induction hypothesis to $C_1$ with $x_1$ and $x_2$.

Finally, suppose that both of the $C_{i}$ have degree $>3$.  Let $\overline{M}_{0}$ and $\overline{M}_{1}$ denote the components of $\overline{M}_{0,0}(S)$ containing $C_{0}$ and $C_{1}$ and let $\overline{M}_{0}^{(1)}$ and $\overline{M}_{1}^{(1)}$ denote the corresponding components of $\overline{M}_{0,1}(S)$.  Since $g$ is a smooth point of $\overline{M}_{0,0}(S)$, we see that $\overline{M}$ contains a component of $\overline{M}_{0}^{(1)} \times_{S} \overline{M}_{1}^{(1)}$ representing gluings of deformations of $C_{0}$ and $C_{1}$.  Since $C_{0}$ is free and $x_{0}$ is general, if we fix a general point $x_{2} \in S$ then the sublocus of $\overline{M}_{0}$ parametrizing curves through $x_{0}$ and $x_{2}$ has codimension $2$.  By our degree assumption this locus is non-empty, and a general such curve will not contain $x_{1}$.  Similarly, the sublocus of $\overline{M}_{1}$ parametrizing curves through $x_{1}$ and $x_{2}$ has codimension $2$.  By our degree assumption this locus is non-empty, and a general such curve will not contain $x_{0}$.  Thus the corresponding component of $\overline{M}_{0}^{(1)} \times_{S} \overline{M}_{1}^{(1)}$ contains a stable map onto two free curves $\widetilde{C}_{0}$ and $\widetilde{C}_{1}$ such that $\widetilde{C}_{0}$ contains $x_{0}$ but not $x_{1}$ and $\widetilde{C}_{1}$ contains $x_{1}$ but not $x_{0}$ and the image of the node is the general point $x_{2}$. Then we may apply the induction hypothesis to $\widetilde{C}_0$ with $x_0, x_2$ and to $\widetilde{C}_1$ with $x_1, x_2$ and our assertion follows.
\end{proof}

We will also need the following lemma:

\begin{lemm} \label{lemm:nefconedelpezzo}
Let $S$ be a smooth del Pezzo surface of degree $\geq 2$.  Then every element of $\Nef_{1}(S)_{\mathbb{Z}}$ is represented by a sum of free rational curves.
\end{lemm}

\begin{proof}
By \cite[Corollary 2.3]{Testa09}, it suffices to show that for every del Pezzo surface $S$ of degree $\geq 2$ there is a free rational curve in the anticanonical linear series $|-K_{S}|$.  This is clear for $\mathbb{P}^{1} \times \mathbb{P}^{1}$; for any other del Pezzo surface, we can find a birational map $\phi: S \to \mathbb{P}^{2}$.  Then the strict transform $C$ of a rational cubic on $\mathbb{P}^{2}$ which contains the $\phi$-exceptional locus will be a rational curve in $|-K_{S}|$.  Using deformation theory we see that $C$ must deform in at least a $1$-dimensional family on $S$, showing that a general deformation will be free.
\end{proof}

\section{Fujita invariants}

In this section we work over an arbitrary field of characteristic $0$. 
Recall that we have the following definition of the Fujita invariants
\begin{defi}
Let $X$ be a smooth projective variety over a field of characteristic $0$.  Let $L$ be a big and nef $\mathbb{Q}$-Cartier divisor on $X$.  The Fujita invariant, or the $a$-invariant is
\begin{equation}
a(X, L) = \min \{ t \in \mathbb R \mid K_{X} + tL \in \overline{\mathrm{Eff}}^1(X)\}.
\end{equation}
\end{defi}

Note that the $a$-invariant is geometric: it does not change under field extension.  We will be interested in how the Fujita invariant behaves over $k(\mathbb{P}^{1})$.

\begin{defi}
Let $\pi: \mathcal{X} \to \mathbb{P}^{1}$ be a morphism with irreducible generic fiber and let $L$ be a $\mathbb{Q}$-Cartier divisor on $\mathcal{X}$ whose restriction to the generic fiber is big and nef.  The generic $a$-invariant of $\mathcal{X}$ with respect to $L$ is $a(\mathcal X_{\eta},L|_{\mathcal X_{\eta}})$.  
\end{defi}

When $\pi$ is a Fano fibration, we obtain a geometric interpretation of the generic $a$-invariant using Lemma \ref{lemm:relativeprops}.

\begin{lemm}
Let $\pi: \mathcal{X} \to \mathbb{P}^{1}$ be a Fano fibration and let $L$ be a $\mathbb{Q}$-Cartier divisor on $\mathcal X$ whose restriction to a general fiber is big and nef.  Then for any smooth Fano fiber $F$ we have
\begin{equation*}
a(\mathcal X_{\eta},L|_{\mathcal X_{\eta}}) = a(F,L|_{F}).
\end{equation*}
\end{lemm}

\begin{proof}
By Lemma \ref{lemm:relativeprops} we see that $K_{\mathcal{X}_{\eta}} + tL|_{\mathcal{X}_{\eta}}$ is pseudo-effective if and only if $K_{F} + tL|_{F}$ is pseudo-effective for every smooth Fano fiber $F$.  In particular the corresponding Fujita invariants must coincide.
\end{proof}

For del Pezzo surfaces, it is easy to work out the behavior of the $a$-invariant of the anticanonical divisor when restricted to subvarieties.  This leads to the following description:

\begin{lemm} \label{lemm:genericainvfordp}
Let $\pi: \mathcal{X} \to \mathbb{P}^{1}$ be a del Pezzo fibration.  Then:
\begin{itemize}
\item A subvariety $Y$ will have $a(Y_{\eta},-K_{\mathcal{X}/\mathbb{P}^{1}}) > 1$ if and only if its intersection with a general fiber $F$ is a union of curves of the following types: $(-1)$-curves, or rational curves in $|-K_F|$ when $F$ has degree $1$.
\item A subvariety $Y$ will have $a(Y_{\eta},-K_{\mathcal{X}/\mathbb{P}^{1}}) = 1$ if and only if its intersection with a general fiber $F$ is a union of curves of the following types: irreducible fibers of a conic fibration on $F$, the rational curves in $|-K_{F}|$ if $F$ has degree $2$, and the rational curves which lie in $|-2K_{F}|$ or the pullback of the anticanonical linear series on a degree $2$ del Pezzo surface if $F$ has degree $1$.
\end{itemize}
\end{lemm}

\begin{proof}
By \cite[Proposition 4.4]{LST18} every irreducible component of $Y_{\overline{K(B)}}$ will have the same Fujita invariant as $\mathcal{Y}_{\eta}$.  Thus it suffices to show that if $S$ is a del Pezzo surface over an algebraically closed field of characteristic $0$ then any subvariety $Z \subset S$ with Fujita invariant $\geq 1$ must have one of the types described in the statement.  Any such curve $Z$ must be rational.  The desired statement then follows from Lemma \ref{lemm:conicscubicsdescription} and the computation $a(Z,-K_{S}) = \frac{2}{-K_{S} \cdot Z}$. 
\end{proof}

In particular, the subvarieties $Y$ with larger generic $a$-invariant than $\mathcal{X}$ form a closed subset.  The subvarieties $Y$ with the same generic $a$-invariant as $\mathcal{X}$ are a little more complicated; note that they need not form a bounded family on $\mathcal{X}$ (even though the corresponding subvarieties of $\mathcal{X}_{\eta}$ do form a bounded family).

\section{Bend-and-Break}

We will frequently use Bend-and-Break (as in \cite{Mori82} and \cite[Lemma 1.9]{KM98}) for families of rational curves through two fixed points, but we will need a slightly more precise version than is usually stated in the literature.  In order to explain this precise version we review the following lemma (used in the proof of Bend-and-Break) which identifies the exact criterion we will need.

\begin{lemm} \label{lemm:strongerbandb}  
Let $\pi: Y \to C$ be a fibration with general fiber $\mathbb{P}^{1}$ from a smooth projective surface $Y$ to a smooth projective curve $C$. Let $X$ be a projective variety and suppose $f: Y \to X$ is a morphism such that $\dim(f(Y)) = 2$ and two sections $C_{1}, C_{2}$ are contracted to different points $x_{1}, x_{2}$ in $X$.  Then there is some fiber $F$ of $\pi$ admitting two different components $F_{1}, F_{2}$ which satisfy:
\begin{itemize}
\item each $F_{i}$ is not contracted by $f$, and
\item the $f$-image of $F_{1}$ contains $x_{1}$ and the $f$-image of $F_{2}$ contains $x_{2}$.  
\end{itemize}
\end{lemm}

\begin{proof}
We assume otherwise for a contradiction.  Let $D_{1}$ denote the connected component of the $f$-fiber over $x_{1}$ that contains $C_{1}$.  Using the Hodge index theorem, we can assign coefficients to the components of $D_{1}$ so that the resulting divisor $\widetilde{D}_{1}$ satisfies the following properties:
\begin{itemize}
\item $C_{1}$ occurs in $\widetilde{D}_{1}$ with coefficient $1$,
\item every other $\pi$-horizontal component of $D_{1}$ occurs with coefficient $0$ in $\widetilde{D}_{1}$, and
\item $\widetilde{D}_{1} \cdot E = 0$ for every $\pi$-vertical component $E$ of $D_{1}$.
\end{itemize}
In the same way we use $C_{2}$ to construct $D_{2}$ and $\widetilde{D}_{2}$ in the fiber over $x_{2}$.  Note that $\widetilde{D}_{1}$ and $\widetilde{D}_{2}$ do not intersect.

By our assumption, every fiber $F$ of $\pi$ has a unique component $F_{0}$ such that
\begin{itemize}
\item $F_{0}$ is not contracted by $f$,
\item $F_{0}$ is the unique component of $F$ that meets $D_{1}$ but is not contained in $D_{1}$, and
\item $F_{0}$ is the unique component of $F$ that meets $D_{2}$ but is not contained in $D_{2}$.
\end{itemize}
If $F_{0}$ occurs with coefficient $m_{0}$ in $F$, then by construction
\begin{equation*}
m_{0}F_{0} \cdot \widetilde{D}_{1} = F \cdot \widetilde{D}_{1} = 1.
\end{equation*}
Similarly $m_{0}F_{0} \cdot \widetilde{D}_{2} = 1$.  We deduce that $\widetilde{D}_{1} - \widetilde{D}_{2}$ has vanishing intersection against every $\pi$-vertical curve.  Since the general fiber of $\pi$ is a rational curve, this condition implies that $\widetilde{D}_{1} - \widetilde{D}_{2}$ is numerically equivalent to a curve whose support is $\pi$-vertical.  By \cite[Lemma 1-2-10]{Matsuki02} we deduce that $\widetilde{D}_{1} - \widetilde{D}_{2}$ is numerically equivalent to a multiple of the fiber $F$.  Thus
\begin{equation*}
0 = (\widetilde{D}_{1} - \widetilde{D}_{2})^{2} = \widetilde{D}_{1}^{2} + \widetilde{D}_{2}^{2} - 2 \widetilde{D}_{1} \cdot \widetilde{D}_{2}.
\end{equation*}
But since $\widetilde{D}_{1}$ and $\widetilde{D}_{2}$ are contracted by $f$ they both have negative self-intersection.  Furthermore, $\widetilde{D}_{1} \cdot \widetilde{D}_{2} = 0$.  This yields a contradiction.
\end{proof}

Following through the proof of Bend-and-Break, we obtain the following corollary.

\begin{coro} \label{coro:improvedbandb}
Let $X$ be a projective variety.  Let $f: \mathbb{P}^{1} \to X$ be a birational map onto a rational curve in $X$.  Suppose that that there are two different points $p_{1},p_{2} \in \mathbb{P}^{1}$ and a curve $B^{\circ} \subset \Mor(\mathbb{P}^{1},X,f|_{p_{1},p_{2}})$ such that the images of the maps parametrized by $B^{\circ}$ sweep out a surface in $X$.  Then $f$ deforms as a stable map to a morphism $g: C \to X$ such that $C$ has (at least) two components $C_{1}, C_{2}$ which are not contracted by $g$ and which satisfy $f(p_{i}) \in g(C_{i})$ for $i=1,2$.
\end{coro}

\begin{proof}
By construction, there is some curve $B^{\circ} \subset \Mor(\mathbb{P}^{1},X,f|_{p_{1},p_{2}})$ containing $f$ such that we get a morphism $s: \mathbb{P}^{1} \times B^{\circ} \to X$ that contracts two sections.  If we let $B$ denote a projective closure of a normalization of $B^{\circ}$ then $s$ defines a rational map $\mathbb{P}^{1} \times B \dashrightarrow X$.  Let $\phi: S \to \mathbb{P}^{1} \times B$ be a birational map of smooth varieties resolving $s$.  Note that any fiber of $S \to B$ is a union of rational curves.  By applying Lemma \ref{lemm:strongerbandb} to $s$, we find a reducible fiber $T$ of $S \to B$ and two components $T_{1},T_{2}$ that are not contracted by $s|_{T}$ whose images contain $p_{1}$ and $p_{2}$.  Although $s|_{T}$ may be a prestable map that is not stable, after a stabilization procedure as in \cite[Discussion before Proposition 3]{Behrend97} 
we obtain a stable map $g: C \to X$ satisfying the desired properties.
\end{proof}

\subsection{Breaking curves on surfaces}

It will also be helpful to have more precise breaking results for sections of a fibration $f: S \to \mathbb{P}^{1}$ whose generic fiber is $\mathbb{P}^{1}$.  We first need a lemma:

\begin{lemm} \label{lemm:blowingup}
Let $Y$ be a smooth projective surface with a morphism $\pi : Y \rightarrow \mathbb P^1$ such that a general fiber of $\pi$ is isomorphic to $\mathbb P^1$.  Let $F$ be a singular fiber of $\pi$ with components $\{ E_{i} \}_{i=1}^{r}$.  Suppose that $E_{1}$ is a $(-1)$-curve that has multiplicity $1$ in the fiber $F$.  Then there is another $(-1)$-curve in the fiber $F$.
\end{lemm}

\begin{proof}
The argument is by induction on the number of components of $F$.  When $F$ has two components the claim is clear.

We next prove the inductive step.  There is a sequence of contractions of $(-1)$-curves
\begin{equation*}
Y=Y_{1} \xrightarrow{\phi_{1}} Y_{2} \xrightarrow{\phi_{2}} Y_{3} \xrightarrow{\phi_{3}} \ldots \xrightarrow{\phi_{r-1}} Y_{r-1}
\end{equation*}
which contract all the components of $F$ but one.  We may suppose that $E_1$ is the exceptional divisor for $\phi_{1}$ since otherwise the statement is clear.  Let $E_{2}$ denote the exceptional divisor for $\phi_{2}$ and also (by abuse of notation) its strict transform on $Y$.  If $E_{2}$ does not intersect $E_{1}$ on $Y$ then it is another $(-1)$-curve, proving the claim.  If $E_{2}$ does intersect $E_{1}$ on $Y$, then on $Y_{2}$ the curve $E_{2}$ is a $(-1)$-curve and it must have multiplicity $1$ (since the multiplicity of $E_{2}$ on $Y_{2}$ is a lower bound for the multiplicity of $E_{1}$ on $Y$).  By applying the induction hypothesis to $Y_{2}$ we find another $(-1)$ curve $E_{j}$.  The strict transform of $E_{j}$ to $Y$ can not intersect $E_{1}$ (since this would make the multiplicity of $E_{1}$ in the fiber $F$ larger than $1$), showing that $E_{j}$ is a $(-1)$-curve on $Y$ in the fiber $F$.
\end{proof}

\begin{coro} \label{coro:surfacecontraction}
Suppose that $\pi: Y \to \mathbb{P}^{1}$ is a Fano fibration of relative dimension $1$.  Fix a section $C$ of $\pi$.  There is a birational map $\rho$ from $Y$ to a Hirzebruch surface $\mathbb{F}_{e} = \mathbb{P}_{\mathbb{P}^{1}}(\mathcal{O} \oplus \mathcal{O}(-e))$ that is an isomorphism on an open neighborhood of $C$.
\end{coro}

\begin{proof}
Suppose that $\pi$ admits a reducible fiber $F_{0}$.  Then $F_{0}$ must contain a $(-1)$-curve.  This curve may or may not be disjoint from $C$.  However, if it is not disjoint from $C$ then it must have multiplicity $1$ in $F_{0}$ and Lemma \ref{lemm:blowingup} guarantees that there is a different $(-1)$-curve in this fiber.  Thus either way we are guaranteed to have a $(-1)$-curve in $F_{0}$ disjoint from $C$.  Contracting such $(-1)$-curves inductively yields the desired Hirzebruch surface.
\end{proof}

The following theorem gives us two breaking results for sections on surfaces, depending on whether or not we want the resulting section to be movable.

\begin{theo} \label{theo:surfacebreaking}
Let $\pi: Y \to \mathbb{P}^{1}$ be a Fano fibration of relative dimension $1$ and let $F$ denote a general fiber of $\pi$.  Fix a dominant family of sections $C$ on $Y$ and set $q = -K_{Y/\mathbb{P}^{1}} \cdot C$.  Let $\rho: Y \to \mathbb{F}_{e}$ be the birational map constructed by Corollary \ref{coro:surfacecontraction} applied to $C$.  Then $e \leq q$ and
\begin{enumerate}
\item There exists a section $C_{0}$ such that $C \sim_{rat} C_{0} + \frac{q+e}{2}F + T$ for some $\pi$-vertical effective cycle $T$.
\item There exists a dominant family of sections $C_{1}$ such that $C \sim_{rat} C_{1} + \frac{q-e}{2}F.$
\end{enumerate}
In both situations the coefficient of $F$ is an integer.
\end{theo}

\begin{proof}
Pushing forward general deformations of $C$ to $\mathbb{F}_{e}$, we obtain a dominant family of sections $\widetilde{C}$ on $\mathbb{F}_{e}$ of height $q$.  Since the minimal height of a moving section of $\mathbb{F}_{e}$ is $e$, we have $e \leq q$.

We have $\widetilde{C} \sim_{rat} \widetilde{C}_{0} + \frac{q+e}{2} \widetilde{F}$ where $\widetilde{C}_{0}$ is the rigid section and $\widetilde{F}$ is a fiber of the projective bundle structure.  We also have $\widetilde{C} \sim_{rat} \widetilde{C}_{1} + \frac{q-e}{2} \widetilde{F}$ where $\widetilde{C}_{1}$ is the minimal free section of $\mathbb{F}_{e}$.  Pulling back general deformations to $Y$ gives us the two desired expressions.
\end{proof}

\section{Expected dimension of sections}

\cite{LT17} proves that if $X$ is a smooth Fano variety and $M$ is a component of $\Mor(\mathbb{P}^{1},X)$ with dimension larger than expected then the curves parametrized by $M$ will all be contained in a subvariety $Y \subset X$ with $a(Y,-K_{X}) > a(X,-K_{X})$. We would like to formulate an analogous result for Fano fibrations using the generic $a$-invariant with respect to $-K_{\mathcal{X}/\mathbb{P}^{1}}$.  However, the analogy fails in two different ways.

First, there can be components with higher than expected dimension that are not observed by the $a$-invariant at all.

\begin{exam}
Consider the Fano fibration $f: \mathbb{F}_{e} \to \mathbb{P}^{1}$ where $\mathbb{F}_{e} = \mathbb{P}_{\mathbb{P}^{1}}(\mathcal{O} \oplus \mathcal{O}(-e))$.  Since $\mathbb{F}_{e,\eta} = \mathbb{P}^{1}_{k(\mathbb{P}^{1})}$ there are no subvarieties with larger generic $a$-invariant.  However, the rigid section has higher than expected dimension as soon as $e > 1$.
\end{exam}

Examples of this type indicate that our proposed analogy should only address sections with large anticanonical degree.  (Note that this setting is sufficient for applications to Manin's Conjecture.)

Second, even sections which have large anticanonical degree and which deform more than expected need not be contained in subvarieties with higher generic $a$-invariant.  The following example shows that such sections can sweep out a subvariety $Y$ whose generic $a$-invariant is the same as the generic $a$-invariant of $\mathcal{X}$.  As above we will let $\mathbb{F}_{e}$ denote the Hirzebruch surface defined by $\mathcal{O} \oplus \mathcal{O}(-e)$.

\begin{exam} \label{exam:sameainv}
Let $h: \mathbb{F}_{2} \to \mathbb{P}^{1}$ denote the projection map, let $C_{0}$ denote the rigid section of $h$ and let $D$ denote a fiber of $h$.  We define $g: \mathcal{X} \to \mathbb{F}_{2}$ to be the $\mathbb{P}^{1}$-bundle defined by $\mathcal{O} \oplus \mathcal{O}(-C_{0}-D)$.  The composed map $\pi: \mathcal{X} \to \mathbb{P}^{1}$ realizes $X$ as an $\mathbb{F}_{1}$-fibration over $\mathbb{P}^{1}$.  We let $\zeta$ denote the unique effective section of $\mathcal{O}_{\mathcal{X}/\mathbb{F}_{2}}(1)$.  Note that $K_{\mathcal{X}/\mathbb{P}^{1}} = -2\zeta + g^{*}(-3C_{0} - 3D)$.

Let $Y \subset \mathcal{X}$ denote the $g$-preimage of $C_{0}$.  Then $Y$ is isomorphic to $\mathbb{F}_{1}$ and we have $K_{\mathcal{X}/\mathbb{P}^{1}} |_{Y} = K_{Y/\mathbb{P}^{1}} + 2T$ where $T$ is the fiber of $Y \to \mathbb{P}^{1}$.  Note that the generic a-invariants of $\mathcal{X}$ and $Y$ with respect to $-K_{\mathcal{X}/\mathbb{P}^{1}}$ are equal.

Let $C$ be any section of $Y \to \mathbb{P}^{1}$ whose deformations cover $Y$.  The calculation above shows that the expected dimension of the moduli space of deformations of $C$ on $\mathcal{X}$ is always one less than the actual dimension.  Furthermore we can find such sections $C$ with arbitrarily large height. 
\end{exam}

In light of the examples above, the following conjectural statements should be seen as the correct analogues of \cite[Theorem 1.1]{LT17}:
\begin{enumerate}
\item Any section with sufficiently large anticanonical degree that deforms more than expected will sweep out a subvariety $Y \subset \mathcal{X}$ whose generic $a$-invariant with respect to $-K_{\mathcal{X}/\mathbb{P}^{1}}$ is at least as large as that of $\mathcal{X}$.
\item If $Y$ is an ``accumulating subvariety'' (in the sense that sections on $Y$ will dominate the expected exponential term in Geometric Manin's Conjecture), then the generic $a$-invariant of $Y$ with respect to $-K_{\mathcal{X}/\mathbb{P}^{1}}$ is strictly larger than that of $\mathcal{X}$.
\end{enumerate}

In this section we will prove statements of this type for del Pezzo fibrations.

\begin{rema}
Another distinction with the case of trivial families is that one does not expect any sort of converse to hold.  That is, there can be subvarieties $Y$ satisfying $a(Y_{\eta},-K_{\mathcal{X}/\mathbb{P}^{1}}) > a(X_{\eta},-K_{\mathcal{X}/\mathbb{P}^{1}})$ which do not admit any families of sections with higher than the expected dimension.  This is a consequence of the fact that the resolution of the structure map $Y \to \mathbb{P}^{1}$ need not have connected fibers.
\end{rema}

\begin{theo} \label{theo:toomuchdeforming}
Let $\pi: \mathcal{X} \to \mathbb{P}^{1}$ be a del Pezzo fibration such that $-K_{\mathcal{X}/\mathbb{P}^{1}}$ is relatively nef.  Let $M$ denote a component of $\Sec(\mathcal{X}/\mathbb{P}^{1})$ parametrizing sections $C$ satisfying
\begin{equation*}
-K_{\mathcal{X}/\mathbb{P}^{1}} \cdot C \geq  \sup\{-2\mathrm{neg}(\mathcal{X},-K_{\mathcal{X}/\mathbb{P}^{1}}) - 1, 1 \}
\end{equation*}
Let $Y$ denote the closure of the locus swept out by the corresponding sections.  Suppose that either
\begin{itemize}
\item $Y \subsetneq \mathcal{X}$, or
\item $\dim(M) > -K_{\mathcal{X}/\mathbb{P}^{1}} \cdot C + 2$.
\end{itemize}
Then $a(Y_\eta, -K_{\mathcal X/\mathbb P^1}|_{Y}) \geq a(\mathcal{X}_\eta, -K_{\mathcal X/\mathbb P^1})$.
\end{theo}

\begin{proof}
Note that every rigid section $C$ will satisfy $-K_{\mathcal{X}/\mathbb{P}^{1}} \cdot C \leq -2$.  Due to our restriction on height we see that in either of the two cases of the theorem we must have $\dim(Y) = 2$. Let $\phi: \widetilde{Y} \to Y$ denote a resolution of singularities.  Let $\widetilde{C}$ denote the strict transform of a general deformation of the section $C$.  By assumption the deformations of $\widetilde{C}$ are Zariski dense on $\widetilde{Y}$; thus the natural map $\psi: \widetilde{Y} \to \mathbb{P}^{1}$ is an algebraic fiber space. Moreover our height bound guarantees that there is at least a $1$-parameter family of deformations of $\widetilde{C}$ through two general points of $\widetilde{Y}$.  By Bend-and-Break we see that $\widetilde{Y}$ is generically a $\mathbb P^1$-bundle over the base.  Since the $\widetilde{C}$ dominate $\widetilde{Y}$, we have
\begin{equation*}
\dim(M) = -K_{\widetilde{Y}/\mathbb{P}^{1}} \cdot \widetilde{C} + 1.
\end{equation*}
Thus in either of the two cases in the statement of the theorem we are guaranteed to have
\begin{equation*}
(K_{\widetilde{Y}/\mathbb{P}^{1}} - \phi^{*}K_{\mathcal{X}/\mathbb{P}^{1}}|_{Y}) \cdot \widetilde{C}  < 0.
\end{equation*}

By applying Theorem~\ref{theo:surfacebreaking} to $\widetilde{Y}$, we find a Hirzebruch surface $\mathbb{F}_{e}$ and a birational map $\rho: \widetilde{Y} \to \mathbb{F}_{e}$ which is an isomorphism on a neighborhood of $\widetilde{C}$.  As in the statement of Theorem \ref{theo:surfacebreaking} there exists a section $\widetilde{C}_{0}$ such that $\widetilde{C} \sim_{rat} \widetilde{C}_{0} + \frac{q+e}{2}F + T$ for some $\pi$-vertical effective cycle $T$, where
\begin{align*}
q & = -K_{Y/\mathbb{P}^{1}} \cdot \widetilde{C} \\
& > -K_{\mathcal{X}/\mathbb{P}^{1}} \cdot C \\
& \geq  -2\mathrm{neg}(\mathcal{X},-K_{\mathcal{X}/\mathbb{P}^{1}}) - 1
\end{align*}
the last inequality following from our height bounds.
In particular, since the coefficient of $F$ is an integer we have $\frac{q+e}{2} \geq -\neg(\mathcal{X},-K_{\mathcal{X}/\mathbb{P}^{1}})$.  Note that
\begin{align}
0 & > (K_{\widetilde{Y}/\mathbb{P}^{1}} - \phi^{*}K_{\mathcal{X}/\mathbb{P}^{1}}|_{Y}) \cdot \widetilde{C}\nonumber \\
& = -q - \phi^{*}K_{\mathcal{X}/\mathbb{P}^{1}}|_{Y} \cdot \left(\widetilde{C}_{0} + \frac{q+e}{2}F + T \right) \nonumber\\
&  = e + \frac{q+e}{2} ( -  \phi^{*}K_{\mathcal{X}/\mathbb{P}^{1}}|_{Y} \cdot F - 2)  -  \phi^{*}K_{\mathcal{X}/\mathbb{P}^{1}} \cdot (\widetilde{C}_{0} + T) \nonumber\\
&  \label{eqn:toomuchdeformingequation} \geq e + \frac{q+e}{2} ( -  \phi^{*}K_{\mathcal{X}/\mathbb{P}^{1}}|_{Y} \cdot F - 2)  -  K_{\mathcal{X}/\mathbb{P}^{1}} \cdot \phi_{*}\widetilde{C}_{0} 
\end{align}
where we use the relative nefness of $-K_{\mathcal{X}/\mathbb{P}^{1}}$ at the last step.  Suppose for a contradiction that there is an inequality $a(Y_\eta, -K_{\mathcal X/\mathbb P^1}|_{Y}) < a(\mathcal{X}_\eta, -K_{\mathcal X/\mathbb P^1})$.  This is equivalent to saying that $-\phi^{*}K_{\mathcal{X}/\mathbb{P}^{1}}|_{Y} \cdot F \geq 3$.  Then by combining the two inequalities above we obtain 
\begin{equation*}
0 > e + \frac{q+e}{2} + \mathrm{neg}(\mathcal{X},-K_{\mathcal{X}/\mathbb{P}^{1}}) \geq e \geq 0
\end{equation*}
yielding a contradiction.
\end{proof}

The most interesting situation in Theorem \ref{theo:toomuchdeforming} is when $a(Y_{\eta},-K_{\mathcal{X}/\mathbb{P}^{1}}|_{Y}) = a(\mathcal{X}_{\eta},-K_{\mathcal{X}/\mathbb{P}^{1}})$.  In this case the same argument gives a little more:

\begin{theo} \label{theo:sameainvprop}
Let $\pi: \mathcal{X} \to \mathbb{P}^{1}$ be a del Pezzo fibration such that $-K_{\mathcal{X}/\mathbb{P}^{1}}$ is relatively nef.  Let $M$ denote a component of $\Sec(\mathcal{X}/\mathbb{P}^{1})$.
Suppose that the sections $C$ parametrized by $M$ satisfy 
\begin{equation*}
-K_{\mathcal{X}/\mathbb{P}^{1}} \cdot C \geq  \sup\{-2\mathrm{neg}(\mathcal{X},-K_{\mathcal{X}/\mathbb{P}^{1}}) - 1, 1 \}
\end{equation*}
and sweep out a surface $Y \subsetneq X$ satisfying
\begin{equation*}
a(Y_\eta, -K_{\mathcal{X}/\mathbb P^1}|_{Y}) = a(\mathcal{X}_\eta, -K_{\mathcal X/\mathbb P^1}).
\end{equation*}
Then there is a component $M'$ of $\Sec(\mathcal{X}/\mathbb{P}^{1})$ parametrizing sections $C_{1}$ satisfying $-K_{\mathcal{X}/\mathbb{P}^{1}} \cdot C_{1} < -\neg(\mathcal X, -K_{\mathcal X/\mathbb P^1}) - 1$ such that the corresponding sections sweep out $Y$.  Furthermore, the difference between the expected and actual dimension of $M$ is the same as the difference between the expected and actual dimension of $M'$.
\end{theo}

\begin{proof}
We return to the setting of the proof of Theorem \ref{theo:toomuchdeforming}, keeping the constructions and notation established there.
In this situation we have $-\phi^{*}K_{\mathcal{X}/\mathbb{P}^{1}}|_{Y} \cdot F = 2$ and
\[
(K_{\widetilde{Y}/\mathbb{P}^{1}} - \phi^{*}K_{\mathcal{X}/\mathbb{P}^{1}}|_{Y}) \cdot F = 0.
\]
Let $\widetilde{C}_1$ denote the strict transform on $\widetilde{Y}$ of a general moving section of $\mathbb{F}_{e}$ of minimal height. Since $\rho$ is an isomorphism on a neighborhood of $\widetilde{C}_{1}$, we have
\[
0 > (K_{\widetilde{Y}/\mathbb{P}^{1}} - \phi^{*}K_{\mathcal{X}/\mathbb{P}^{1}}|_{Y}) \cdot \widetilde{C} = (K_{\widetilde{Y}/\mathbb{P}^{1}} - \phi^{*}K_{\mathcal{X}/\mathbb{P}^{1}}|_{Y}) \cdot \widetilde{C}_1 = -e - \phi^{*}K_{\mathcal{X}/\mathbb{P}^{1}}|_{Y} \cdot \widetilde{C}_1 .
\]
On the other hand by Equation \eqref{eqn:toomuchdeformingequation} in the proof of Theorem \ref{theo:toomuchdeforming} we also have
\[
0 >  e + \neg(\mathcal X, -K_{\mathcal X/\mathbb P^1}).
\]
Altogether it follows that
\[
- \phi^{*}K_{\mathcal{X}/\mathbb{P}^{1}}|_{Y} \cdot \widetilde{C}_1 < -\neg(\mathcal X, -K_{\mathcal X/\mathbb P^1}) - 1.
\]
Let $C_{1}$ be the image in $\mathcal{X}$ of $\widetilde{C}_{1}$.  We let $M'_{Y}$ denote the component of $\Sec(\widetilde{Y}/\mathbb{P}^{1})$ parametrizing deformations of $\widetilde{C}_{1}$ and let $M'$ denote the component of $\Sec(\mathcal{X}/\mathbb{P}^{1})$ which parametrizes deformations of $C_{1}$.  Since $\widetilde{C}$ and $\widetilde{C}_{1}$ differ by $mF$ for some integer $m$, we have
\begin{align*}
\dim(M) = -K_{\widetilde{Y}/\mathbb{P}^{1}} \cdot \widetilde{C} + 1 & =  -K_{\widetilde{Y}/\mathbb{P}^{1}} \cdot \widetilde{C}_{1} + 2m + 1 = \dim(M'_{Y}) + 2m \\
\expdim(M) = -K_{\mathcal{X}/\mathbb{P}^{1}} \cdot C + 2 & =  -K_{\mathcal{X}/\mathbb{P}^{1}} \cdot C' + 2m + 2 = \expdim(M') + 2m
\end{align*}
Note that there is a natural pushforward map $M'_{Y} \to M'$.  We claim that in fact this map is birational.  In other words, we must show that $Y$ is the closure of the locus swept out by the sections defined by $M'$.  It is clear that this locus contains $Y$, and we only must verify that it is not a larger set.  We know that $\dim(M) - \expdim(M) \geq 0$, since any family of curves on $\mathcal{X}$ deforms at least as much as the expected dimension.  By the equations above $\dim(M'_{Y}) - \expdim(M') \geq 0$ as well.  Suppose for a contradiction that $M'$ parametrizes a dominant family of sections.  Then $\dim(M') = \expdim(M')$ so that $\dim(M'_{Y}) - \dim(M') \geq 0$.  Since the reverse inequality is also true, we deduce $\dim(M'_{Y}) = \dim(M')$.  But then a general deformation of $C_{1}$ in $\mathcal{X}$ should be coming from a deformation of $\widetilde{C}_{1}$ in $Y$, contradicting the fact that $M'$ parametrizes a dominant family.  Thus the curves parametrized by $M'$ cannot sweep out all of $\mathcal{X}$ and must be contained in $Y$.  The final statement is a consequence of the identification $\dim(M'_{Y}) = \dim(M')$ and the earlier equations.
\end{proof}

\begin{exam}
We return to Example \ref{exam:sameainv}, keeping the notation established there.  Recall that the fibers $T$ of $Y \to \mathbb{P}^{1}$ satisfy $-K_{\mathcal{X}/\mathbb{P}^{1}} \cdot T = -K_{Y/\mathbb{P}^{1}} \cdot T$.  Thus if we take any fixed section $C$ in $Y$, glue on copies of the free curve $T$, and smooth the resulting curve, the difference between the expected dimension of deformations in $\mathcal{X}$ and the actual dimension of deformations in $Y$ stays constant.

Conversely, starting from any movable section in $Y$, we can repeatedly break off copies of $T$.  Eventually we will obtain a movable section of smallest degree, namely $\zeta|_{Y}$.  This section satisfies $-K_{\mathcal{X}/\mathbb{P}^{1}} \cdot \zeta|_Y = -1$.  This verifies Theorem \ref{theo:sameainvprop} which claims that there is a section of height less than $-\neg(\mathcal X, -K_{\mathcal X/\mathbb P^1}) - 1= 2$ whose deformations cover $Y$.
\end{exam}

\begin{rema} \label{rema:accumulatingsubvars}
Note that in the situation of Theorem \ref{theo:sameainvprop} the difference between the actual and expected dimension is constant for the sections which sweep out $Y$.  This means that the contributions of $Y$ to the counting function in Geometric Manin's Conjecture have the same exponential term as the expected value.  In other words, the ``accumulating subvarieties'' which contribute a larger exponential term are exactly the subvarieties with larger generic $a$-value. 
\end{rema}

\subsection{Constructing a closed set}
Combining previous results, we prove Theorem \ref{theo:maintheorem1}.

\begin{proof}[Proof of Theorem \ref{theo:maintheorem1}:]
By Theorem \ref{theo:toomuchdeforming}, any component $M \subset \Sec(\mathcal{X}/\mathbb{P}^{1})$ parametrizing a non-dominant family will satisfy one of the following properties: 
\begin{enumerate}
\item $M$ will parametrize sections $C$ satisfying $-K_{\mathcal{X}/\mathbb{P}^{1}} \cdot C <  \sup\{ - 2 \neg(\mathcal{X},-K_{\mathcal{X}/\mathbb{P}^{1}})-2, 1 \}$, or
\item $M$ will parametrize sections $C$ satisfying $-K_{\mathcal{X}/\mathbb{P}^{1}} \cdot C \geq  \sup\{ - 2 \neg(\mathcal{X},-K_{\mathcal{X}/\mathbb{P}^{1}})-2, 1 \}$ which sweep out a $2$-dimensional subvariety $Y$ satisfying $a(Y_{\eta},-K_{\mathcal{X}/\mathbb{P}^{1}}|_{Y}) > a(\mathcal{X}_{\eta},-K_{\mathcal{X}/\mathbb{P}^{1}})$, or
\item $M$ will parametrize sections $C$ satisfying $-K_{\mathcal{X}/\mathbb{P}^{1}} \cdot C \geq  \sup\{ - 2 \neg(\mathcal{X},-K_{\mathcal{X}/\mathbb{P}^{1}})-2, 1 \}$  which sweep out a $2$-dimensional subvariety $Y$ satisfying $a(Y_{\eta},-K_{\mathcal{X}/\mathbb{P}^{1}}|_{Y}) = a(\mathcal{X}_{\eta},-K_{\mathcal{X}/\mathbb{P}^{1}})$.
\end{enumerate}
Lemma \ref{lemm:northcott} shows that curves of the first type lie in a bounded family.  Thus, the union of the subvarieties swept out by the non-dominant families will be a proper closed subset of $\mathcal{X}$. Lemma \ref{lemm:genericainvfordp} shows that the subvarieties defined by the components of the second type will lie in a proper closed subset of $\mathcal{X}$.  We still need to address the third type.

Suppose that $Y$ is a $2$-dimensional subvariety satisfying $a(Y_{\eta},-K_{\mathcal{X}/\mathbb{P}^{1}}|_{Y}) = a(\mathcal{X}_{\eta},-K_{\mathcal{X}/\mathbb{P}^{1}})$ that is swept out by the sections parametrized by $M$.  By Theorem \ref{theo:sameainvprop}, we know that there will be a component $M' \subset \Sec(\mathcal{X}/\mathbb{P}^{1})$ parametrizing a non-dominant family of sections $C$ which satisfy $-K_{\mathcal{X}/\mathbb{P}^{1}} \cdot C< -\neg(\mathcal X, -K_{\mathcal X/\mathbb P^1}) - 1$ and which sweep out $Y$.  By Lemma \ref{lemm:northcott} we conclude that the union of all such subvarieties $Y$ is a proper closed subset of $\mathcal{X}$.
\end{proof}

\section{Stable maps through general points}

Let $\pi: \mathcal{X} \to \mathbb{P}^{1}$ be a del Pezzo fibration.  Suppose that $f: C \to \mathcal{X}$ is a genus $0$ stable map whose image contains $n$ general points of $\mathcal{X}$.  In this section we study what restrictions this condition imposes on $C$ and $f$.  We first discuss the bound for irreducible curves:

\begin{lemm}
Let $\pi: \mathcal{X} \to \mathbb{P}^{1}$ be a del Pezzo fibration.  Let $M \subset \Sec(\mathcal{X}/\mathbb{P}^{1})$ denote a component such that for any $n$ general points of $\mathcal{X}$ there exists a member of $M$ containing those points.  Then the curves $C$ parametrized by $M$ have height $\geq 2n-2$ with respect to $-K_{\mathcal{X}/\mathbb{P}^{1}}$. Furthermore, if the height is exactly $2n-2$ then there are only finitely many curves parametrized by $M$ through $n$ fixed general points and if the height is $2n-1$ then there is sublocus of $M$ of dimension at most $1$ parametrizing curves through $n$ fixed general points.
\end{lemm}

\begin{proof}
Since the points are general they will impose independent conditions on $M$, so that $\dim(M) \geq 2n$.  Since $M$ parametrizes a dominant family of curves, it has the expected dimension, showing that
\begin{equation*}
2n \leq -K_{\mathcal{X}/\mathbb{P}^{1}} \cdot C + 2.
\end{equation*}
The last two statements are clear.
\end{proof}

To understand reducible curves takes more preparation.

\begin{defi}
Let $\pi: \mathcal{X} \to \mathbb{P}^{1}$ be a del Pezzo fibration.  Fix an integer $d$.  We let $\maxdef(d)$ denote the maximum dimension of any component $M \subset \Sec(\mathcal{X}/\mathbb{P}^{1})$ parametrizing sections of height $d$. When there is no section of height $d$, we simply set $\maxdef(d) = -\infty$. 
We also define
\begin{equation*}
\maxdef(\mathcal{X}) = \sup_{d < 0} \maxdef(d).
\end{equation*}
\end{defi}

\begin{lemm} \label{lemm:maxdefbound}
Let $\pi: \mathcal{X} \to \mathbb{P}^{1}$ be a del Pezzo fibration such that $-K_{\mathcal{X}/\mathbb{P}^{1}}$ is relatively ample.  Fix an integer $d < 0$ and a positive integer $n \geq \maxdef(d)$.  Suppose that we have a connected effective curve $C = C_{0} + \sum a_{i}T_{i}$ where $C_{0}$ is a section of height $d$ and each $T_{i}$ is a $\pi$-vertical curve.  If $C$ contains $n$ general points $\{x_{j}\}_{j=1}^{n}$ of $\mathcal{X}$ then $-K_{\mathcal{X}/\mathbb{P}^{1}} \cdot C \geq d + 3n - \maxdef(d)$.
\end{lemm}

\begin{proof}
Let $M$ be the component of $\Sec(\mathcal{X}/\mathbb{P}^{1})$ containing $C_{0}$.  Since $d < 0$ by assumption we know that $M$ does not parametrize a dominant family.  First suppose that the curves parametrized by $M$ sweep out a surface $Y \subset \mathcal{X}$.  Note that deformations of $C_{0}$ can not contain any general points of $\mathcal{X}$.  Thus each general point $x_{j}$ is contained in a different free vertical curve $T_{j}$.

Suppose that $-K_{\mathcal{X}/\mathbb{P}^{1}} \cdot T_{j} = 2$.  There is only a finite set of conics in a del Pezzo surface through a fixed general point $x_j$.  By generality of $x_{j}$ the intersection points of these conics with $Y$ form a finite set.  To ensure that $C$ is connected, $C_{0}$ must contain one of this finite set of points.  By generality, the conditions imposed on $C_{0}$ by insisting that $C_{0}$ contain the attachment points for conics $T_{j}$ will be independent for different $j$.  Thus we see that there can be at most $\maxdef(d)$ such conics $T_{j}$ in the curve $C$.  This proves that
\begin{align*}
-K_{\mathcal{X}/\mathbb{P}^{1}} \cdot C \geq d + 2 \maxdef(d) + 3(n-\maxdef(d))
\end{align*}

When $C_{0}$ is a rigid curve, the argument is similar but easier since every component $T_{j}$ containing a general point must satisfy $-K_{\mathcal{X}/\mathbb{P}^{1}} \cdot T_{j} \geq 3$.
\end{proof}

The two following propositions describe which sections can pass through the maximal number of general points of $\mathcal{X}$.  The proof of \cite[Theorem 4.1]{Tian12} establishes a related statement for a particular type of curves of minimal height. The first proposition handles the case of even height.

\begin{prop} \label{prop:evendegreebound}
Let $\pi: \mathcal{X} \to \mathbb{P}^{1}$ be a del Pezzo fibration such that $-K_{\mathcal{X}/\mathbb{P}^{1}}$ is relatively ample.  Fix a positive integer
\begin{equation*}
n \geq \maxdef(\mathcal{X}) + 2 + \sup\{0, - \neg(\mathcal{X},-K_{\mathcal{X}/\mathbb{P}^{1}})\}.
\end{equation*}
Suppose that $f: C \to \mathcal{X}$ is a genus $0$ stable map whose image has anticanonical height $2n-2$ such that the unique component of $C$ whose image is not $\pi$-vertical maps birationally to a section. Then:
\begin{enumerate}
\item Suppose the image of $C$ contains $n$ general points of $\mathcal{X}$.  Then $f$ is a birational map to a free section.
\item Fix a general curve $Z$ in a basepoint free linear series in a general fiber of $\pi$.  Suppose the image of $C$ contains $n-1$ general points of $\mathcal{X}$ and intersects $Z$.  Suppose also that the image of $C$ is reducible and at least one of the general points is contained in a $\pi$-vertical component of $C$.  Then $C$ has exactly two components and $f$ maps one component birationally onto a free section and the other birationally onto either a free $-K_{\mathcal X}$-conic or a free $-K_{\mathcal X}$-cubic in a general fiber of $\pi$.
\end{enumerate}
\end{prop}

\begin{cons}
Before giving the proof, we clarify what ``general'' means in the statement of the proposition.  We may ensure that $n$ general points and a general member $Z$ of a basepoint free linear series in a general fiber of $\pi$ satisfy the following conditions: 
\begin{enumerate}
\item We may ensure that the $n$ points and $Z$ are contained in different fibers of $\pi$ and that every such fiber is a smooth del Pezzo surface $F$.  We may furthermore ensure that each point does not lie on any $-K_F$-line in $F$. 
\item Fix a positive integer $d$ so that $d \leq n$.  Suppose we fix any subset of our set of points of size $d$.  Then we may ensure that the set of sections of height $2d-2$ which contain this subset of points is finite (possibly empty) and that every such section is free.  Indeed, a free section of height $2d-2$ will deform in dimension $2d$.  Since there are only finitely many components of $\Sec(\mathcal{X}/\mathbb{P}^{1})$ parametrizing free sections of height $2d-2$, we can choose $d$ points general with respect to these families such that there are only finitely many sections in these families containing all $d$ points.  Since the finitely many sections satisfying these incidence correspondences are general in the moduli, we conclude that they are free.

A similar argument shows that we may ensure there are finitely many loci of dimension $\leq 1$ in $\overline{M}_{0,0}(X)$ that parametrize sections of height $2d-1$ which contain this set of $d$ points and a general point in each locus is a free section.  
\item We may ensure that for each of our points $p$ the finite set of $-K_F$-conics in $F$ through $p$ has the expected dimension of intersection against the sections described in (2); namely, for any $0$-dimensional component of the parameter space consisting of sections of height $2d-2$ passing through $d$ general points which are different from $p$, the corresponding section is disjoint from each $-K_F$-conic, and for any $1$-dimensional component of the parameter space consisting of sections of height $2d-1$ through $d$ general points which are different from $p$ there are only finitely many sections intersecting each $-K_F$-conic. Moreover we may ensure that these finitely many sections are all free.

To see this, consider the incidence correspondence $I \subset T \times S$ where $T$ parametrizes sets of $d$ points contained in a section of degree $2d-2$, $S$ parametrizes a choice of a point $p$ and a $\pi$-vertical conic containing it, and $I$ represents the condition that the section meet the conic.  Since in a family of free curves it is a codimension $1$ condition to intersect any fixed curve, we see that the general fiber of the second projection $I \to S$ has codimension $1$ in $T$. In particular $I$ is a proper closed subscheme of $T \times S$.  A similar argument works for the other condition.
\item We may ensure that the sections in (2) and (3) meet $Z$ in the expected dimension: for any $0$-dimensional component of the parameter space the corresponding section is disjoint from $Z$, and for any $1$-dimensional component of the parameter space there are only finitely many sections intersecting $Z$. Moreover we may ensure that these finitely many sections are all free.  The argument is the same as for (3).
\item Suppose we fix a subset of our $n$ points of size $d-1$.  We may ensure that there are only finitely many sections of height $2d-2 \leq 2n-2$ passing through all $d-1$ points which also intersect $Z$ and intersect a $-K_F$-conic containing a general point which is distinct from our chosen subset of $d-1$ points. Moreover we may ensure that these finitely many sections are all free.

To see this, note that sections of height $2d-2$ containing $d-1$ general points can deform in dimension at most $2$.  Among them sections meeting with $Z$ will form a locus of dimension at most $1$.  Using an incidence correspondence as above, we see there are only finitely many sections passing through the $d-1$ points that intersect $Z$ and intersect a $-K_F$-conic passing through  point different from the $d-1$ general points. Since such sections are general in moduli, we conclude that they are free.
\item For a point $p$ in our set we may ensure that there are only finitely many $-K_F$-cubics in $F$ through $p$ that intersect any section that is parametrized by a $0$-dimensional subset as in (2), (3), (4), (5) with respect to a subset of points not containing $p$. The argument is the same as for (3).
\item We may ensure that there is no section as in (3) and (5) intersecting $-K_F$-conics through two different points in our set. Furthermore, we may ensure that sections parametrized by a $0$-dimensional component of the parameter space as in (2), (3), (4), and (5) do not intersect with a vertical line meeting $Z$.  The argument is the same as before.
\item Since $n \geq \maxdef(\mathcal{X})$, we may ensure that for any $d<0$ and any subset of our set of points of size $\geq \maxdef(d)$ the points are general in the sense of Lemma \ref{lemm:maxdefbound}.
\end{enumerate}
\end{cons}

\begin{proof}[Proof of Proposition \ref{prop:evendegreebound}:]
(1).  It suffices to show that the image of $f$ is irreducible.  Suppose otherwise, so that $f(C) = C_{0} + \sum_{i \in I} a_{i}T_{i}$ for some $\pi$-vertical curves $T_{i}$.   Let $d$ denote the height of $C_{0}$ and set $t_{i} = -K_{\mathcal{X}/\mathbb{P}^{1}} \cdot T_{i}$, so that
\begin{equation*}
2n-2 = d + \sum_{i \in I} a_{i}t_{i}.
\end{equation*}
Note that $C_{0}$ can contain at most $\sup \{ 0, \lfloor \frac{d}{2} \rfloor + 1\}$ general points of $\mathcal{X}$.   An irreducible $\pi$-vertical curve $T_{i}$ can contain at most $1$ general point and if it does then $t_{i} \geq 2$.   Let $I' \subset I$ denote the set of vertical curves that contain one of the general points.  

We now break the argument into several cases.

\noindent \textbf{Case 1: $d \geq 0$.}  Then the number of general points contained in $f(C)$ is bounded above by the number of general points contained in $C_{0}$ and in the $T_{i}$.  Thus:
\begin{align*}
\sup \left\{ 0, \left\lfloor \frac{d}{2} + 1 \right\rfloor \right\} + |I'| & \geq n \\
& = \frac{d}{2} + 1 + \sum_{i \in I} a_{i}t_{i}/2.
\end{align*}
Since $d \geq 0$ the RHS is an upper bound for the LHS.  Thus the inequality above must be an equality.  This means that $d$ is even and $C_{0}$ goes through the maximum number of points possible, that each $a_{i}=1$, and that each component of $T$ is a free vertical curve through one of the general points with $t_{i} = 2$.  In particular, the set of $\frac{d}{2}+1$ general points determines a finite number of possibilities for $C_{0}$, and each vertical curve is also determined by a general point up to a finite set of possibilities.  If there are any vertical components, then for general choices $f(C)$ will not be connected, an impossibility.  Thus $f(C)$ is irreducible. Then $f(C)$ contains all $n$ general points so there are only finitely many choices of $f : C \to \mathcal X$. Thus by generality it must be free.

\noindent \textbf{Case 2: $d < 0$.}
Due to our lower bound on $n$ we may apply Lemma \ref{lemm:maxdefbound}.  It shows that $C$ must have height
\begin{align*}
-K_{\mathcal{X}/\mathbb{P}^{1}} \cdot C & \geq \neg(\mathcal{X},-K_{\mathcal{X}/\mathbb{P}^{1}}) + 3n - \maxdef(\mathcal{X}) \\
& > 2n-2
\end{align*}
contradicting our assumption on the height of $C$.

\bigskip

(2).  Write $f(C) = C_{0} + \sum_{i \in I} a_{i}T_{i}$ for some $\pi$-vertical curves $T_{i}$.   Let $d$ denote the height of $C_{0}$ and set $t_{i} = -K_{\mathcal{X}/\mathbb{P}^{1}} \cdot T_{i}$, so that
\begin{equation*}
2n-2 = d + \sum_{i \in I} a_{i}t_{i}.
\end{equation*}
 Let $I' \subset I$ denote the set of vertical curves that contain one of the general points.  Again we separate into cases:

\noindent \textbf{Case 1: $d \geq 0$ and $C_{0}$ intersects $Z$.} 
Since $C_{0}$ intersects $Z$ it can contain at most $\sup \{ 0, \lfloor \frac{d+1}{2} \rfloor \}$ general points of $\mathcal{X}$.   The number of general points contained in $f(C)$ is bounded above by the number of general points contained in $C_{0}$ and in the $T_{i}$.  Thus:
\begin{align*}
\sup \left\{ 0, \left\lfloor \frac{d+1}{2} \right\rfloor \right\} + |I'| & \geq n-1 \\
& \geq \frac{d}{2} + \sum_{i \in I} a_{i}t_{i}/2.
\end{align*}
If $d$ is even, then we must have equality everywhere.  This means that set of deformations of $C_{0}$ which meet $Z$ and go through the maximal number of points is at most $1$-dimensional, that every $a_{i}=1$, and that each component of $T$ is a free vertical curve through one of the general points with $t_{i} = 2$.  However, for $C_{0}$ to meet a vertical conic through a general point is a codimension $1$ condition, so by generality there can be at most one vertical component. Moreover there are only finitely many such sections going through the maximum number of points and meeting with $Z$ and a conic going through a general point, thus $C_0$ must be free by generality.  Thus we obtain the desired expression. 

If $d$ is odd, then $C_{0}$ must contain the maximal number of points $\lfloor \frac{d+1}{2} \rfloor$.  This means that the difference between $|I'|$ and $\sum a_{i}t_{i}/2$ can be at most $1/2$.  Thus there are only three options for the vertical components:
\begin{enumerate}
\item every component of $T$ has anticanonical degree $2$ and contains a general point, or
\item every component of $T$ but one has anticanonical degree $2$ and contains a general point, and the last one has anticanonical degree $3$ and contains a general point, or,
\item every component of $T$ but one has anticanonical degree $2$ and contains a general point, and the last component has anticanonical degree $1$.
\end{enumerate}
Note that there are only finitely many deformations of $C_{0}$ which meet $Z$ and go through $\frac{d+1}{2}$ points. Thus $C_{0}$ must be free by generality.  Recall that by assumption there is a vertical component of $f(C)$ through a general point.  However, since there are only finitely many vertical conics through a general point, by generality no such conic can intersect $C_{0}$.  This rules out the first and third situations, showing that $C$ must be the union of a free section and a cubic in a fiber.

\noindent \textbf{Case 2: $d \geq 0$ and $C_{0}$ does not intersect $Z$.}
Just as before, the number of general points contained in $f(C)$ is bounded above by the number of general points contained in $C_{0}$ and in the $T_{i}$.  Thus:
\begin{align*}
\sup \left\{ 0, \left\lfloor \frac{d}{2} \right\rfloor + 1 \right\} + |I'| & \geq n - 1 \\
& \geq \frac{d}{2} + \sum_{i \in I} a_{i}t_{i}/2.
\end{align*}
Since $C_{0}$ does not intersect $Z$, there must be a vertical curve that does intersect $Z$ but does not contain any general points, so that
\begin{equation*}
\sum_{i \in I} a_{i}t_{i}/2 - |I'| \geq \sum_{i \in I} a_{i}t_{i}/2 - \sum_{i \in I'} a_{i}t_{i}/2 \geq 1/2.
\end{equation*}
Thus when $d$ is even, $C_{0}$ must contain the maximal number of points $\frac{d}{2}+1$.  In this situation the difference between $|I'|$ and $\sum_{i \in I'} a_{i}t_{i}/2$ is at most $1/2$.  There are four options for the vertical components:
\begin{enumerate}
\item every component of $T$ but one has anticanonical degree $2$ and contains a general point, and the last has anticanonical degree $1$ and meets $Z$.
\item every component of $T$ but two has anticanonical degree $2$ and contains a general point, one has anticanonical degree $1$ and meets $Z$, and the last one has anticanonical degree $3$ and contains a general point,
\item every component of $T$ but two has anticanonical degree $2$ and contains a general point, and the last two components have anticanonical degree $1$, one of which meets $Z$,
\item every component of $T$ has anticanonical degree $2$, all but one contain a general point, and the last component meets $Z$.
\end{enumerate}
Since by generality there are only finitely many deformations of $C_{0}$ through the required number of points, such $C_{0}$ can not intersect a vertical line meeting $Z$ or a vertical conic through a general point.  This rules out the first three cases immediately, and the fourth is also ruled out since by assumption there exists at least one vertical curve which contains some general point. 

When $d$ is odd then $C_{0}$ must contain the maximal number of points $\lfloor \frac{d}{2} \rfloor+1$.  In this case there is only one option: every component of $T$ but one has anticanonical degree $2$ and contains a general point, and the last has anticanonical degree $1$ and meets $Z$.  However, since by generality $C_{0}$ can only deform in a one-parameter family while containing the maximal number of points, it is impossible for $C_{0}$ to meet both a line intersecting $Z$ and a conic through a general point.  Since by assumption $f(C)$ contains a vertical component through a general point, this case is also ruled out.

\noindent \textbf{Case 3: $d < 0$.}
In this case Lemma \ref{lemm:maxdefbound} shows that $C$ must have height
\begin{align*}
-K_{\mathcal{X}/\mathbb{P}^{1}} \cdot C & \geq \neg(\mathcal{X},-K_{\mathcal{X}/\mathbb{P}^{1}}) + 3(n-1) - \maxdef(\mathcal{X}) \\
& > 2n-2
\end{align*}
proving the impossibility of this case.
\end{proof}

The next proposition is the analogue of Proposition \ref{prop:evendegreebound} for sections of odd height.

\begin{prop} \label{prop:odddegreebound}
Let $\pi: \mathcal{X} \to \mathbb{P}^{1}$ be a del Pezzo fibration such that $-K_{\mathcal{X}/\mathbb{P}^{1}}$ is relatively ample.  Fix a positive integer
\begin{equation*}
n \geq \maxdef(\mathcal{X}) + \sup\{0, - \neg(\mathcal{X},-K_{\mathcal{X}/\mathbb{P}^{1}})\}.
\end{equation*}
Suppose that $f: C \to \mathcal{X}$ is a genus $0$ stable map whose image has anticanonical height $2n-1$ such that the unique component of $C$ whose image is not $\pi$-vertical maps birationally to a section. Then:
\begin{enumerate}
\item Fix a general curve $Z$ in a basepoint free linear series in a general fiber of $\pi$. Suppose the image of $C$ contains $n$ general points of $\mathcal{X}$ and intersects $Z$. Then $f$ is a birational map to a free section.
\item Suppose the image of $C$ contains $n$ general points of $\mathcal{X}$.  Suppose also that the image of $C$ is reducible and at least one of the general points is contained in a $\pi$-vertical component of $C$.  Then $C$ has exactly two components and $f$ maps one component birationally onto a free section and the other birationally onto a free $-K_{\mathcal X}$-conic or a free $-K_{\mathcal X}$-cubic in a general fiber of $\pi$.
\end{enumerate}
\end{prop}

The proof is essentially the same as the proof of Proposition \ref{prop:evendegreebound}, but slightly easier.

\section{Movable Bend-and-Break for del Pezzo fibrations}

The following conjecture is essential for understanding sections of Fano fibrations.

\begin{conj}[Movable Bend-and-Break for sections] \label{conj:movablebandb}
Let $\pi: \mathcal{X} \to \mathbb{P}^{1}$ be a Fano fibration.  There is a constant $Q = Q(\mathcal{X})$ such that the following holds.  Suppose that $C$ is a movable section of $\pi$ satisfying $-K_{\mathcal{X}/\mathbb{P}^{1}} \cdot C > Q(\mathcal{X})$.  Then $C$ deforms (as a stable map) to a chain of free curves with at least $2$ components.
\end{conj}

We establish Movable Bend-and-Break for sections of del Pezzo fibrations such that $-K_{\mathcal{X}/\mathbb{P}^{1}}$ is relatively ample.  
The proof is based on techniques from \cite{Shen12} and \cite{LT18}.  The following statement is a stronger version of Theorem \ref{theo:maintheorem2} where we state a precise value for the bound $Q(\mathcal{X})$.

\begin{theo} \label{theo:mbbmain}
Let $\pi: \mathcal{X} \to \mathbb{P}^{1}$ be a del Pezzo fibration such that $-K_{\mathcal{X}/\mathbb{P}^{1}}$ is relatively ample.  Define
\begin{align*}
Q(\mathcal{X}) = \sup \{& 3, -2 \neg(\mathcal X, -K_{\mathcal X/\mathbb P^1}) -5, \\
& -\neg(\mathcal X, -K_{\mathcal X/\mathbb P^1}) + 3, \\ & 2\maxdef(\mathcal{X}) - 5\neg(\mathcal{X},-K_{\mathcal{X}/\mathbb{P}^{1}}) - 5, \\
& 2\maxdef(\mathcal{X})-\neg(\mathcal{X},-K_{\mathcal{X}/\mathbb{P}^{1}}) - 3, \\
& 2\maxdef(\mathcal{X}) + 2 + 2\sup\{0, - \neg(\mathcal{X},-K_{\mathcal{X}/\mathbb{P}^{1}})\} \}.
\end{align*}
Suppose that $M \subset \Sec(\mathcal{X}/\mathbb{P}^{1})$ is a component that parametrizes a dominant family of sections $C$ satisfying $-K_{\mathcal{X}/\mathbb{P}^{1}} \cdot C \geq Q(\mathcal{X})$.  Then the closure of $M$ in $\overline{M}_{0,0}(\mathcal{X})$ contains a point representing a stable map whose domain has exactly two components each mapping birationally onto a free curve.
\end{theo}

For later applications it will be convenient to introduce the following definition:

\begin{defi}
Let $\pi: \mathcal{X} \to \mathbb{P}^{1}$ be a del Pezzo fibration such that $-K_{\mathcal{X}/\mathbb{P}^{1}}$ is relatively ample.  We define
$\MBB(\mathcal{X})$ to be the smallest non-negative integer $r$ such that any component of $\overline{M}_{0,0}(\mathcal{X})$ which generically parametrizes free sections of height at least $r$  will also contain a stable map whose domain has exactly two components each mapping birationally to a free curve.
\end{defi}

Thus Theorem \ref{theo:mbbmain} establishes that $\MBB(\mathcal{X})$ exists and gives an explicit upper bound $\MBB(\mathcal{X}) \leq Q(\mathcal{X})$.

\begin{proof}[Proof of Theorem \ref{theo:mbbmain}:]
Suppose that $M$ is a component of $\Sec(\mathcal{X}/\mathbb{P}^{1})$ parametrizing a dominant family of sections $C$ that satisfy
\begin{align*}
-K_{\mathcal{X}/\mathbb{P}^{1}} \cdot C \geq Q(\mathcal{X}).
\end{align*}
In particular, the general section in the family is free.  Since each section is smooth we can consider its normal bundle, and we separate into several cases based on the normal bundle of the general curve $C$.  Write $N_{C/\mathcal{X}} = \mathcal{O}(a) \oplus \mathcal{O}(b)$ with $0 \leq a \leq b$.  Our height bound implies that if $a = 0$ then $b \geq 3$ and if $a=1$ then $b \geq 2$.

\textbf{Case 1:} $a \geq 1$ and $b \geq a+2$.

Suppose we fix $a+1$ general points of $X$ and choose a curve $C$ with the generic normal bundle containing these points.  According to \cite[Remark 2.2]{Shen12}, the locus parametrizing curves with the generic normal bundle that contain all $a+1$ points has a unique smooth irreducible component that contains $C$.  The closure of the locus swept out by these curves is an irreducible surface $\Sigma$.  Let $\Sigma' \to \Sigma$ denote the normalization map.  \cite[Corollary 2.7]{Shen12} shows that if $C$ is a general curve through these points then $\Sigma'$ is smooth along the strict transform $C'$ of $C$ and the normal bundle of $C'$ in $\Sigma'$ is $\mathcal{O}(b)$.  Let $\widetilde{\Sigma}$ denote a minimal resolution of singularities of $\Sigma'$, let $\nu: \widetilde{\Sigma} \to \Sigma$ denote the induced birational map, and let $\widetilde{C}$ denote the strict transform of $C'$ in $\widetilde{\Sigma}$.  Since $\Sigma'$ is smooth along $C'$, we see that $\widetilde{C}$ still has normal bundle $\mathcal{O}(b)$ because $\widetilde{\Sigma} \to \Sigma'$ is locally an isomorphism along $\widetilde{C}$. Note that $\widetilde{\Sigma}$ is covered by curves with degree one over the base $\mathbb{P}^{1}$, so that the map $\psi: \widetilde{\Sigma} \to \mathbb{P}^{1}$ is an algebraic fiber space.  Furthermore, since $b \geq 3$ by the usual Bend-and-Break we know that the general fiber of $\psi$ must be $\mathbb{P}^{1}$.  Since $\Sigma$ contains a general point, we may assume that $-K_{\mathcal X/\mathbb P^1} \cdot F \geq 2$ where $F$ is a general fiber of $\psi$.

As in the statement of Theorem \ref{theo:surfacebreaking} there exists a birational map $\rho: \widetilde{\Sigma} \to \mathbb{F}_{e}$, a moving section $\widetilde{C}_{1}$ of $\widetilde{\Sigma}$ such that $\widetilde{C} \sim_{rat} \widetilde{C}_{1} + \frac{b-e}{2}F$, and a section $\widetilde{C}_{0}$ such that $\widetilde{C} \sim_{rat} \widetilde{C}_{0} + \frac{b+e}{2}F + T$ for some $\pi$-vertical curve $T$.  Thus
\begin{align}
a & = -b -\nu^{*}K_{\mathcal{X}/\mathbb{P}^{1}} \cdot \widetilde{C} \nonumber \\
& = - b + \frac{b+e}{2} \left(-\nu^{*}K_{\mathcal{X}/\mathbb{P}^{1}} \cdot F \right) - \nu^{*}K_{\mathcal{X}/\mathbb{P}^{1}} \cdot (\widetilde{C}_{0} + T) \nonumber \\
& = e + \frac{b+e}{2} \left(-\nu^{*}K_{\mathcal{X}/\mathbb{P}^{1}} \cdot F - 2 \right)- \nu^{*}K_{\mathcal{X}/\mathbb{P}^{1}} \cdot (\widetilde{C}_{0} + T) \nonumber \\
& \label{eqn:mbbargumentequation} \geq e + \frac{b+e}{2} \left(-\nu^{*}K_{\mathcal{X}/\mathbb{P}^{1}} \cdot F - 2 \right) - \nu^{*}K_{\mathcal{X}/\mathbb{P}^{1}} \cdot \widetilde{C}_{0}
\end{align}
where the last inequality follows from the fact that $-K_{\mathcal{X}/\mathbb{P}^{1}}$ is relatively ample.

\begin{clai}  \label{claim1} We have $b>e$.
\end{clai}

The statement of Theorem \ref{theo:surfacebreaking} guarantees that $b \geq e$; suppose for a contradiction that $b=e$.  We separate the argument into two cases.  First suppose that $-K_{\mathcal{X}/\mathbb{P}^{1}} \cdot F \geq 3$.  Then the inequality \eqref{eqn:mbbargumentequation} yields
\begin{equation*}
a \geq 2b - \nu^{*}K_{\mathcal{X}/\mathbb{P}^{1}} \cdot \widetilde{C}_{0} \geq 2b + \neg(\mathcal X, -K_{\mathcal X/\mathbb P^1}).
\end{equation*}
Since also $b-2 \geq a$, we deduce that $b \leq -\neg(\mathcal X, -K_{\mathcal X/\mathbb P^1}) -2$. Thus
\[
-K_{\mathcal X/\mathbb{P}^{1}} \cdot C = a+b \leq 2b -2 \leq -2\neg(\mathcal X, -K_{\mathcal X/\mathbb P^1}) -6
\]
which contradicts with our height bounds on $C$.

Now suppose that $-K_{\mathcal{X}/\mathbb{P}^{1}} \cdot F = 2$.  Then the inequality \eqref{eqn:mbbargumentequation} yields
\begin{equation}  \label{eqn:mbbcomparison}
a - b \geq - \nu^{*}K_{\mathcal{X}/\mathbb{P}^{1}} \cdot \widetilde{C}_{0}
\end{equation}
and in particular, $-2 \geq  - \nu^{*}K_{\mathcal{X}/\mathbb{P}^{1}} \cdot \widetilde{C}_{0}$.  Recall that $\Sigma$ contains $a+1$ general points of $\mathcal{X}$, and that their preimages give a set of $a+1$ distinct points in $\widetilde{\Sigma}$.  Since $b > a+1$, we can deform the $b$ fibers $F$ in our broken curve to ensure that they meet these $a+1$ points.  By taking images in $\mathcal{X}$, we find a broken curve which is a deformation of $C$ through $a+1$ general points of $\mathcal{X}$ such that the section $\nu_{*}\widetilde{C}_{0}$ in this broken curve has height $\leq -2$.  According to Lemma \ref{lemm:maxdefbound}, this implies that either
\begin{equation*}
-K_{\mathcal{X}/\mathbb{P}^{1}} \cdot C \geq \neg(\mathcal{X},-K_{\mathcal{X}/\mathbb{P}^{1}}) + 3(a+1) - \maxdef(\mathcal{X})
\end{equation*}
or $a +1 < \maxdef(\mathcal{X})$.  Furthermore by the inequality \eqref{eqn:mbbcomparison}
\begin{align*}
-K_{\mathcal{X}/\mathbb{P}^{1}} \cdot C & = a+b \\
& \leq 2a - \neg(\mathcal X, -K_{\mathcal X/\mathbb P^1})
\end{align*}
Comparing the two bounds we see that
\begin{equation*}
a \leq \sup\{ \maxdef(\mathcal{X}) - 2\neg(\mathcal{X},-K_{\mathcal{X}/\mathbb{P}^{1}}) - 3, \maxdef(\mathcal{X}) - 2 \}
\end{equation*}
which in turn implies that
\begin{equation*}
-K_{\mathcal{X}/\mathbb{P}^{1}} \cdot C \leq  \sup \{ 2\maxdef(\mathcal{X}) - 5\neg(\mathcal{X},-K_{\mathcal{X}/\mathbb{P}^{1}}) - 6, 2\maxdef(\mathcal{X})-\neg(\mathcal{X},-K_{\mathcal{X}/\mathbb{P}^{1}}) - 4\}.
\end{equation*}
This possibility is ruled out by our degree bounds and finishes the proof of the claim.

We now return to the main argument.  Since $\frac{b-e}{2}$ is an integer, Claim \ref{claim1} shows that it is at least $1$.  If $\frac{b-e}{2} > 1$, then we smooth $\widetilde{C}_1 + (\frac{b-e}{2}-1)F$ to obtain a section $\widetilde{C}_2$.  If $\frac{b-e}{2}=1$, we set $\widetilde{C}_{2} = \widetilde{C}_{1}$. 
In either case we have $\widetilde{C} \sim_{rat} \widetilde{C}_2 + F$.
We claim that the $\nu$-images of $\widetilde{C}_{2}$ and $F$ are movable curves in $\mathcal{X}$.  Recall that $\Sigma$ contains a general point of $\mathcal{X}$, and in particular, $\Sigma$ is not contained in the closed set constructed by Theorem \ref{theo:maintheorem1}.  Since $\nu_*\widetilde{C}_{2}$ deforms to cover $\Sigma$, its image in $\mathcal{X}$ must deform to cover $\mathcal{X}$.  We also claim that $\nu_{*}\widetilde{C}_{2}$ avoids any singularity of $\mathcal X$. Indeed, the preimage of the singular locus of $\mathcal{X}$ is a union of points and $\psi$-vertical divisors.  Since $\widetilde{C}_{2}$ and $\widetilde{C}$ have the same intersection numbers against $\psi$-vertical divisors, the fact that $\widetilde{C}$ avoids the preimage of the singular locus of $\mathcal{X}$ implies that $\widetilde{C}_{2}$ does as well. 
Similarly, $F$ deforms to cover $\Sigma$, and thus contains a general point of $\mathcal{X}$.  Since the $-K_{\mathcal{X}/\mathbb{P}^{1}}$-degree of $F$ admits an a priori upper bound in terms of the height of $C$ and in terms of $\neg(\mathcal{X},-K_{\mathcal{X}/\mathbb{P}^{1}})$, this implies that $F$ is a movable curve.  It is also contained in the smooth locus of $\mathcal{X}$.

\textbf{Case 2:} $a = 0$ and $b \geq 3$.  

The argument is very similar to Case $1$.  Fix a general point $x$ and consider a component $\Sigma$ of the locus swept out by the curves in $M$ through $x$ with generic normal bundle. Lemma \ref{lemm:normalbundleestimate} shows that $\Sigma \subsetneq \mathcal{X}$.   Since there is a $b$-dimensional family of curves through $x$ with generic normal bundle, $\Sigma$ must be a surface.

Let $\nu: \widetilde{\Sigma} \to \Sigma$ be a resolution; by possibly blowing up further we may suppose that the preimage of $x$ is divisorial, so that the condition for a curve on $\widetilde{\Sigma}$ to go through $x \in \Sigma$ is a numerical condition.  The strict transforms $\widetilde{C}$ of the general sections sweeping out $\Sigma$ will define a dominant family of sections of $\widetilde{\Sigma}$ that deforms in dimension $b$.  Thus we have $-K_{\widetilde{\Sigma}/\mathbb{P}^1} \cdot \widetilde{C} = b-1$.  Note that the induced map $\psi: \widetilde{\Sigma} \to \mathbb{P}^{1}$ has connected fibers since $\widetilde{\Sigma}$ is dominated by sections of $\psi$.  Also, since $b - 1 \geq 2$ by the usual Bend-and-Break we know that the general fiber of $\psi$ is $\mathbb{P}^{1}$.

As in the statement of Theorem \ref{theo:surfacebreaking} there exists a birational map $\rho: \widetilde{\Sigma} \to \mathbb{F}_{e}$, a moving section $\widetilde{C}_{1}$ of $\widetilde{\Sigma}$ such that $\widetilde{C} \sim_{rat} \widetilde{C}_{1} + \frac{b-e-1}{2}F$, and a section $\widetilde{C}_{0}$ such that $\widetilde{C} \sim_{rat} \widetilde{C}_{0} + \frac{b+e-1}{2}F + T$ for some $\pi$-vertical curve $T$. The analogue of inequality \eqref{eqn:mbbargumentequation} is
\begin{equation} \label{eqn:mbbargumentequation2}
0 \geq e-1 + \frac{b+e-1}{2} \left(-\nu^{*}K_{\mathcal{X}/\mathbb{P}^{1}} \cdot F - 2 \right) - \nu^{*}K_{\mathcal{X}/\mathbb{P}^{1}} \cdot \widetilde{C}_{0}.
\end{equation}

We claim that $b -1 > e$.  Assume for a contradiction that $b-1=e$. First suppose that $-K_{\mathcal X/\mathbb P^1} \cdot F \geq 3$.  By inequality \eqref{eqn:mbbargumentequation2} we deduce that $2e \leq 1 - \neg(\mathcal{X},-K_{\mathcal{X}/\mathbb{P}^{1}})$ and thus
\[
-K_{\mathcal X/\mathbb P^1} \cdot C \leq -\frac{1}{2}\neg(\mathcal X, -K_{\mathcal X/\mathbb P^1}) + \frac{3}{2}
\]
contradicting our height bound.  Next suppose that $-K_{\mathcal X/\mathbb P^1} \cdot F =2$.  Then by inequality \eqref{eqn:mbbargumentequation2} we deduce $e \leq 1 - \neg(\mathcal{X},-K_{\mathcal{X}/\mathbb{P}^{1}})$ and thus
\[
-K_{\mathcal X/\mathbb P^1} \cdot C \leq -\neg(\mathcal X, -K_{\mathcal X/\mathbb P^1})+ 2
\]
contradicting our height bound. 

Since $\frac{b-e-1}{2}$ must be an integer our claim implies that it is at least $1$, and we conclude the argument in the same way as before.

\textbf{Case 3:} $a \geq 1$ and $b-a \leq 1$.  

First suppose that $a=b$; in this case our height bounds imply that
\begin{equation*}
a \geq \maxdef(\mathcal{X}) + 1 + \sup\{0, - \neg(\mathcal{X},-K_{\mathcal{X}/\mathbb{P}^{1}})\}.
\end{equation*}
Fix $a$ general points of $\mathcal{X}$ and fix a general curve $Z$ in a basepoint free linear series in a general fiber of $\pi$.  There is a union of one-parameter family of deformations of $C$ containing $a$ general points and meeting $Z$.  Fix one component of this parameter space.

By applying the Bend-and-Break theorem (as in Corollary \ref{coro:improvedbandb}), we can deform $C \sim_{alg} C_{0} + T$, where: 
\begin{itemize}
\item $C_{0}$ is a section,
\item $T$ is an effective $\pi$-vertical curve and at least one component contains one of our $a$ general points.
\end{itemize}
By Proposition \ref{prop:evendegreebound} we see that the stable map corresponding to $C_{0}+T$ has the desired form.

The case when $b-a=1$ is similar.  In this case there is a one-parameter family of sections containing $a+1$ general points.  By considering the broken curve $C_{0} + T$ and applying Proposition \ref{prop:odddegreebound}, we can repeat the argument to obtain the same conclusion.
\end{proof}

\subsection{Consequences of Movable Bend-and-Break}

\begin{defi}
A comb is a nodal curve consisting of one ``central curve'' $C_{0}$ and of $r \geq 0$ ``teeth'' $T_{1},\ldots,T_{r}$ such that each $T_{i}$ intersects $C_{0}$ at a single point and there are no other intersections.  For the purposes of this paper every tooth of a comb will be a rational curve.
\end{defi}

By repeating the breaking argument of Theorem \ref{theo:maintheorem2} inductively, we obtain:

\begin{coro} \label{coro:breakstocomb}
Let $\pi: \mathcal{X} \to \mathbb{P}^{1}$ be a del Pezzo fibration such that $-K_{\mathcal{X}/\mathbb{P}^{1}}$ is relatively ample.  Suppose that $C$ is a section of $\pi$ satisfying $-K_{\mathcal{X}/\mathbb{P}^{1}} \cdot C \geq \MBB(\mathcal{X})$ whose deformations dominate $\mathcal{X}$.  Then $C$ deforms as a stable map to a morphism $f: C \to \mathcal{X}$ whose domain $C$ is a comb $C_{0} \cup T_{1} \cup \ldots \cup T_{r}$ where
\begin{itemize}
\item $f$ maps each component of $C$ birationally onto a free curve, 
\item the image of $C_{0}$ is a section satisfying $-K_{\mathcal{X}/\mathbb{P}^{1}} \cdot C_0 < \MBB(\mathcal{X})$, and
\item the image of each $T_{i}$ is a free $\pi$-vertical curve.
\end{itemize}
\end{coro}

\begin{proof}
Let $\overline{M}$ be a component of $\overline{M}_{0,0}(\mathcal{X})$ containing a stable map $f: C \to X$ such that $C$ is a comb, the central curve $C_{0}$ is mapped birationally to a free section of $\pi$, and the teeth $T_{1},\ldots,T_{r}$ are mapped birationally to free $\pi$-vertical curves.  (Our original assumption is that $\overline{M}$ contains such a curve with zero teeth.)  Suppose that $-K_{\mathcal{X}/\mathbb{P}^{1}} \cdot C_{0} \geq \MBB(\mathcal{X})$. We show that $\overline{M}$ also contains a comb $C'$ of the same form such that the central curve $C_{0}'$ has strictly smaller degree against $-K_{\mathcal{X}/\mathbb{P}^{1}}$.

Let $\overline{M}_{i}$ denote the component of $\overline{M}_{0,0}(\mathcal{X})$ containing the curve $C_{0}$ (when $i=0$) or $T_{i}$ (when $i \geq 1$).  Each such component generically parametrizes irreducible free curves.  Let $\overline{M}_{i}^{(k)}$ denote the unique component of $\overline{M}_{0,k}(\mathcal{X})$ lying over $\overline{M}_{i}$.  Since the comb $f: C \to \mathcal{X}$ is a smooth point of the moduli space, there is a component $\widetilde{M}$ of
\begin{equation*}
\overline{M}_{0}^{(r)} \times_{\mathcal{X}^{\times r}} \left( \prod_{i \geq 1} \overline{M}_{i}^{(1)} \right)
\end{equation*}
that maps into $\overline{M}$.  Note that $\widetilde{M}$ dominates each component of the product under the corresponding projection map: indeed, since the evaluation map
\begin{equation*}
\prod_{i \geq 1} \overline{M}_{i}^{(1)} \to \mathcal{X}^{\times r}
\end{equation*}
is flat on the locus parametrizing free curves, every component of the product containing a union of free curves will have the expected dimension and will thus map dominantly onto every factor.  By Theorem \ref{theo:maintheorem2}, the component $\overline{M}_{0}^{(r)}$ contains in its smooth locus a general point of a component of $\overline{N}_{0}^{(r+1)} \times_{\mathcal{X}} \overline{N}_{1}^{(1)}$ where $\overline{N}_{0}$ generically parametrizes free sections of smaller $-K_{\mathcal{X}/\mathbb{P}^{1}}$-degree and $\overline{N}_{1}$ generically parametrizes free $\pi$-vertical curves.  Since the locus $\overline{N}_{0}^{(r+1)} \times_{\mathcal{X}} \overline{N}_{1}^{(1)}$ has codimension $1$ in $\overline{M}_{0}^{(r)}$ and the projection map to $\overline{M}_{0}^{(r)}$ is flat on the locus of free curves, a dimension count shows that a general point in the preimage of $\overline{N}_{0}^{(r+1)} \times_{\mathcal{X}} \overline{N}_{1}^{(1)}$ in $\widetilde{M}$ represents a comb of free curves.
\end{proof}

The next statement is similar to Theorem \ref{theo:maintheorem2}, but we additionally impose an upper bound on the height of $C_{0}$.  

\begin{theo} \label{theo:precisemovablebandb}
Let $\pi: \mathcal{X} \to \mathbb{P}^{1}$ be a del Pezzo fibration such that $-K_{\mathcal{X}/\mathbb{P}^{1}}$ is relatively ample.  Suppose that $C$ is a section of $\pi$ whose deformations dominate $\mathcal{X}$ and satisfies $-K_{\mathcal{X}/\mathbb{P}^{1}} \cdot C \geq \MBB(\mathcal{X})$.  Then $C$ deforms (as a stable map) to a union of a free section $C_{0}$ satisfying $-K_{\mathcal{X}/\mathbb{P}^{1}} \cdot C_{0} < \MBB(\mathcal{X})$ with a free $\pi$-vertical moving curve.
\end{theo}

We prove this by regluing the $\pi$-vertical curves in the setting of Corollary \ref{coro:breakstocomb}.

\begin{proof}
Let $\overline{M} \subset \overline{M}_{0,0}(\mathcal X)$ be a component containing a point $f: \mathbb{P}^{1} \to \mathcal X$ which is a birational map onto a section as in the statement of Theorem \ref{theo:maintheorem2}.  Corollary \ref{coro:breakstocomb} shows that $\overline{M}$ contains a smooth point representing a comb of $r+1$ free curves such that the section $C_{0}$ satisfies the desired bound.  We prove by induction on $r$ that $\overline{M}$ contains a point representing a union of two free curves such that the section $C_{0}$ satisfies the desired bound.  In the base case when $r=1$ there is nothing to prove.

For the induction step, label the $\pi$-vertical free curves as $T_{1},\ldots,T_{r}$.  Since all the components are free, we may deform the comb so that $C_{0}$ and $T_{1}$ intersect at a general point of $X$.  Suppose we then deform $T_{2}$ along $C_{0}$ until its attaching point coincides with that of $T_{1}$.  Due to the generality of $C_{0} \cap T_{1}$, we may ensure that the corresponding deformation of $T_{2}$ through this point is an irreducible free curve $T_{2}'$.  Thus, the resulting stable map $f': C' \to X$ will have $r+2$ components: there will be one curve $E \subset C'$ which is contracted by $f$ and which is attached to the components mapping to $C_{0}, T_{1}, T_{2}'$.  Note that $f'$ represents a smooth point of $\overline{M}_{0,0}(X)$.  Thus the component $M$ also contains a point representing a stable map whose image is a comb $C_{0} \cup T_{1}\cup T_{2}'' \cup T_{3} \cup \ldots \cup T_{r}$ where we attach an irreducible free deformation $T_{2}''$ of $T_{2}$ to a general point of $T_{1}$.   Note that this is again a smooth point of $\overline{M}_{0,0}(X)$.  Since $T_{1}$ and $T_{2}''$ are free curves in a fiber, we can glue them while fixing the attaching point with $C_{0}$ by  \cite[II.7.6.1 Theorem]{Kollar}, yielding a comb of the form $C_{0} \cup T_{*} \cup T_{3} \cup \ldots \cup T_{r}$ where $T_{*}$ is a free curve in a fiber.  This point is also contained in $\overline{M}$, and we conclude by induction on the number of components of the comb.
\end{proof}

Another version of the result is:

\begin{prop} \label{prop:breakofflowdegree}
Let $\pi: \mathcal{X} \to \mathbb{P}^{1}$ be a del Pezzo fibration such that $-K_{\mathcal{X}/\mathbb{P}^{1}}$ is relatively ample.  Suppose that $C$ is a section of $\pi$ whose deformations dominate $\mathcal{X}$ and satisfies $-K_{\mathcal{X}/\mathbb{P}^{1}} \cdot C \geq \MBB(\mathcal{X})$.  Then $C$ deforms (as a stable map) to a union of a free section $C_{0}$ with a free $\pi$-vertical moving curve $T$ such that $-K_{\mathcal{X}/\mathbb{P}^{1}} \cdot T \leq 3$.
\end{prop}

\begin{proof}
By Theorem \ref{theo:precisemovablebandb}, $\overline{M}$ contains a point representing the union of a free section $C_{0}$ satisfying $-K_{\mathcal{X}/\mathbb{P}^{1}} \cdot C_{0} \leq \MBB(\mathcal{X})-1$ with a free $\pi$-vertical curve.  We may assume that the gluing point of the two curves is a general point of the fiber containing it.  By Lemma~\ref{lemm:BBfordelPezzo} we can further deform the free $\pi$-vertical curve to a chain of free $\pi$-vertical curves, each of which has anticanonical degree at most $3$, while preserving the attachment point with $C_{0}$.  After smoothing all but one of the $\pi$-vertical components into the section, we obtain the desired breaking.
\end{proof}

Finally, by breaking again we can obtain a polynomial upper bound on the number of height $d$ components of $\Sec(\mathcal{X}/\mathbb{P}^{1})$ as predicted by Batyrev.

\begin{proof}[Proof of Corollary \ref{coro:batyrev}:]
It suffices to prove the existence of a polynomial upper bound on the number of components of rigid families, the number of components of families of sections which sweep out surfaces, and the number of components of dominant families of sections.  The first claim is clear since there are only finitely many rigid curves.

Suppose that $M \subset \Sec(\mathcal{X}/\mathbb{P}^{1})$ parametrizes sections of height $d$ which sweep out a closed surface $Y$.  Let $\widetilde{Y}$ denote a resolution of $Y$.  If $d \geq 0$ then there is a one-parameter family of rational curves through a general point on $\widetilde{Y}$. Thus $\widetilde{Y}$ is rationally connected so that two sections on $\widetilde{Y}$ will be numerically equivalent if and only if they are linearly equivalent.  Thus, the number of components of $\Sec(\mathcal{X}/\mathbb{P}^{1})$ that sweep out $Y$ and have height at most $d$ is bounded above by the number of numerical classes of sections of $\widetilde{Y}$ of height at most $d$.  This latter number is a polynomial in $d$.  By Theorem \ref{theo:maintheorem1} only finitely many surfaces $Y$ can be obtained in this way, giving a universal polynomial bound on all components of sections of this type. 

Finally, suppose that $M$ is a component of $\Sec(\mathcal{X}/\mathbb{P}^{1})$ that parametrizes a dominant family of sections.  Let $\overline{M}$ denote the corresponding component of $\overline{M}_{0,0}(\mathcal X)$.  By Theorem \ref{theo:precisemovablebandb}, if the height of the sections parametrized by $M$ is at least $\MBB(\mathcal{X})$ then $\overline{M}$ contains a point representing the union of a free section $C_{0}$ satisfying $-K_{\mathcal{X}/\mathbb{P}^{1}} \cdot C_{0} \leq \MBB(\mathcal{X})-1$ with a free $\pi$-vertical curve.  We may assume that the gluing point of the two curves is a general point of the fiber containing it.  By Lemma~\ref{lemm:BBfordelPezzo} we can further deform the free $\pi$-vertical curve to a chain of free $\pi$-vertical curves, each of which has anticanonical degree at most $3$, while preserving the attachment point with $C_{0}$.

Choose a positive integer $m$ such that $A := -K_{\mathcal{X}/\mathbb{P}^{1}} + mF$ is an ample divisor where $F$ denotes a general fiber of $\pi$.  We have shown that each component $M$ parametrizing a dominant family of sections yields a component $\overline{M}$ which contains a chain of free curves where each component has $A$-degree no more than $\sup \{ 3, m+\MBB(\mathcal{X})\}$.

Now we take a smooth birational model $\beta : \widetilde{\mathcal X} \to \mathcal X$ such that $\beta$ is an isomorphism on the smooth locus of $\mathcal X$. Any free section is contained in this smooth locus, so by taking strict transforms of dominant components $M$ of $\Sec(\mathcal{X}/\mathbb{P}^{1})$ we obtain dominant components $\widetilde{M}$ of $\Sec(\widetilde{\mathcal X}/\mathbb P^1)$. Moreover by taking the strict transforms of chains of free curves, we conclude that for each such component the Zariski closure $\overline{\widetilde{M}} \subset\overline{M}_{0,0}(\widetilde{\mathcal X})$ contains a chain of free curves where each component has $\beta^*A$-degree no more than $\sup \{ 3, m+\MBB(\mathcal{X})\}$.
\cite[Theorem 5.13]{LT17} proves that there is a polynomial upper bound $P(d)$ on the number of components $\overline{\widetilde{M}} \subset \overline{M}_{0,0}(\widetilde{\mathcal X})$ of $\beta^*A$-degree $d$ which contain a chain of free curves where each component has $\beta^*A$-degree $\leq \sup\{3,m + \MBB(\mathcal{X})\}$.  
This yields a polynomial upper bound on the number of components $M$ as desired.
\end{proof}

\subsection{Fibers of evaluation maps}

We now apply Movable Bend-and-Break to show that for sections with large height the evaluation map for the universal family has connected fibers.  This is the analogue in our setting of  \cite[Proposition 5.15]{LT17}.

\begin{lemm}
\label{lemm:evaluationmap}
Let $\pi: \mathcal{X} \to \mathbb{P}^{1}$ be a del Pezzo fibration such that $-K_{\mathcal{X}/\mathbb{P}^{1}}$ is relatively ample.  Let $M \subset \Sec(\mathcal{X}/\mathbb{P}^{1})$ denote a component that parametrizes a dominant family of sections.  Let $p: \mathcal{U} \to M$ denote the universal family and let $s: \mathcal{U} \to \mathcal X$ denote the evaluation map.  

Suppose that there is a component $Q$ of $\overline{M}_{0,0}(\mathcal{X})$ parametrizing a dominant family of reduced $\pi$-vertical curves $C$ with $-K_{\mathcal{X}/\mathbb{P}^{1}} \cdot C \geq 3$ such that a general curve parametrized by $Q$ appears as a component of some stable map in the closure $\overline{M}$ of $M$ in $\overline{M}_{0,0}(\mathcal{X})$.  For any birational model $\mathcal{U}'$ of $\mathcal{U}$ the induced map $s': \mathcal{U}' \to \mathcal{X}$ has irreducible general fiber.
\end{lemm}

\begin{proof}
By taking projective closures and taking resolutions we find smooth projective varieties $W$, $M_{W}$ which admit maps $\overline{s}: W \to \mathcal{X}$ and $\overline{p}: W \to M_{W}$ that agree with $s$ and $p$ on open subsets.  Let $f: Y \to \mathcal{X}$ denote a resolution of the Stein resolution of $\overline{s}$.  Our goal is to show that $f$ is a birational map.

By construction $Y$ is dominated by sections of the natural map $\psi = \pi \circ f: Y \to \mathbb{P}^{1}$, so that $Y$ has the structure of an algebraic fiber space over $\mathbb{P}^{1}$.  Let $\overline{N}$ denote the component of $\overline{M}_{0,0}(Y)$ whose general point is the image of a general fiber of $\overline{p}$.  We know that $\overline{N}$ admits a birational morphism to the corresponding component $\overline{M}$ of $\overline{M}_{0,0}(\mathcal{X})$.    In particular, there is a dominant family of $\psi$-vertical curves $P$ mapping onto the corresponding curves in $Q$.  Since the induced map from $P$ to $Q$ is dominant, the expected dimensions of $P$ and $Q$ agree.  This implies that the curves $C$ parametrized by $P$ satisfy $(K_{Y/\mathbb{P}^{1}} -f^{*}K_{\mathcal{X}/\mathbb{P}^{1}}) \cdot C = 0$.  By the Riemann-Hurwitz formula $K_{Y/\mathbb{P}^{1}} - f^{*}K_{\mathcal{X}/\mathbb{P}^{1}}$ is relatively pseudo-effective, and since $C$ is a movable $\psi$-vertical curve we see that this divisor is not relatively big.  This means that
\begin{equation*}
a(Y_{\eta},-f^{*}K_{\mathcal{X}/\mathbb{P}^{1}}) = 1 = a(X_{\eta},-K_{\mathcal{X}/\mathbb{P}^{1}}).
\end{equation*}
Consider the map of geometric generic fibers $\overline{f}_{\eta}: Y_{\overline{K(B)}} \to \mathcal{X}_{\overline{K(B)}}$.  By \cite[Proposition 4.4]{LST18} the Fujita invariant of a geometrically integral variety over a field of characteristic $0$ is preserved by base change of the ground field.  Then \cite[Theorem 6.2]{LT16} implies that either $\overline{f}_{\eta}$ is birational or that $K_{Y_{\overline{K(B)}}} -\overline{f}_{\eta}^{*}K_{\mathcal{X}_{\overline{K(B)}}}$ has Iitaka dimension $1$.  However, in the latter case the curves $\overline{C}$ contracted by the Iitaka fibration will satisfy $-\overline{f}_{\eta}^{*}K_{\mathcal{X}_{\overline{K(B)}}} \cdot \overline{C} = 2$.  This contradicts the fact that $-K_{\mathcal{X}/\mathbb{P}^{1}} \cdot f_{*}C \geq 3$.  We deduce that $\overline{f}_{\eta}$ is birational, hence that $f_{\eta}$ is birational, hence that $f$ is birational. 
\end{proof}

\begin{coro}
\label{coro:evaluationmap2}
Let $\pi: \mathcal{X} \to \mathbb{P}^{1}$ be a del Pezzo fibration such that $-K_{\mathcal{X}/\mathbb{P}^{1}}$ is relatively ample.  Let $M \subset \Sec(\mathcal{X}/\mathbb{P}^{1})$ denote a component that parametrizes a dominant family of sections.  Suppose that the curves $C$ parametrized by $M$ satisfy $-K_{\mathcal{X}/\mathbb{P}^{1}} \cdot C \geq \MBB(\mathcal{X}) + 2$.  Then any resolution of the universal family map over $M$ has irreducible general fiber.
\end{coro}

\begin{proof}
Let $\overline{M}$ denote the corresponding component of $\overline{M}_{0,0}(\mathcal{X})$.  Applying Theorem \ref{theo:precisemovablebandb} to $\overline{M}$ on $\mathcal{X}$, we see that it contains points which are the union of a free $\pi$-vertical curve of $-K_{\mathcal{X}/\mathbb{P}^{1}}$-degree at least $3$ and a horizontal moving section.  Since such a stable map is a smooth point of the moduli space, a general deformation of this $\pi$-vertical curve must appear as a component of some curve parametrized by $\overline{M}$.  We then apply Lemma \ref{lemm:evaluationmap}.
\end{proof}

\section{Examples} \label{sect:examples}

In this section we analyze several explicit examples of del Pezzo fibrations.

\subsection{Improving previous bounds}
In many situations the upper bound on $\MBB$ given in Theorem \ref{theo:mbbmain} is much larger than the actual value.  In order to compute examples it will be useful to give more precise bounds for del Pezzo fibrations satisfying certain properties.

\begin{itemize}
\item Suppose that we have classified the sections on $\mathcal{X}$ of negative height.  Then we can hope to prove a more precise statement in place of Lemma \ref{lemm:maxdefbound}.  This in turn will improve the bounds in Cases 1 and 2 of Theorem \ref{theo:mbbmain} in the $d<0$ situation, and in Case 3 of Theorem \ref{theo:mbbmain} via Propositions \ref{prop:evendegreebound} and \ref{prop:odddegreebound}.
\item Suppose that $\mathcal{X}_{\eta}$ does not admit any rational curves of anticanonical degree $2$ defined over the ground field $k(\mathbb{P}^{1})$.  This implies that the surface $\widetilde{\Sigma}$ constructed in the proof of Theorem \ref{theo:mbbmain} is not ruled by $-K_{\mathcal{X}/\mathbb{P}^{1}}$-conics.  In particular, this allows us to give more precise bounds in inequalities \eqref{eqn:mbbargumentequation} and \eqref{eqn:mbbargumentequation2}.
\item When the height is $3$ or $4$ and the conditions above hold we can mimic the proof of Theorem \ref{theo:mbbmain} using more explicit constructions to prove a Movable Bend-and-Break theorem by hand.
\end{itemize}

By implementing these improvements we have:

\begin{lemm} \label{lemm:improvedbounds}
Let $\pi: \mathcal{X} \to \mathbb{P}^{1}$ be a del Pezzo fibration such that $-K_{\mathcal X/\mathbb P^1}$ is relatively ample.  Suppose that $\maxdef(d)-d \leq 2$ for every $d<0$ such that there is a section of height $d$.  If $\mathcal{X}_{\eta}$ does not admit any rational curves defined over $k(\mathbb{P}^{1})$ of anticanonical degree $2$ then
\begin{equation*}
\MBB(\mathcal{X}) \leq 3.
\end{equation*}
\end{lemm}

The following theorem improves Proposition \ref{prop:breakofflowdegree} by giving conditions where we are guaranteed to be able to break off a $\pi$-vertical conic.  This simplifies inductive arguments, since it means that we can understand dominant families of sections of height $n$ using only the dominant families of sections of height $n-2$.

\begin{theo} \label{theo:anothermovablebandb}
Let $\pi: \mathcal{X} \to \mathbb{P}^{1}$ be a del Pezzo fibration such that $-K_{\mathcal{X}/\mathbb{P}^{1}}$ is relatively ample.  Assume that a general fiber of this fibration is a del Pezzo surface with $2 \leq (-K_S)^2 \leq 8$ and that every surface swept out by the $\pi$-vertical lines contained in smooth fibers is a big divisor.  Suppose that $C$ is a section of $\pi$ of height $d$ whose deformations dominate $\mathcal{X}$ and such that $d \geq \MBB(\mathcal{X})$. 
Then $C$ deforms (as a stable map) to a union of a free section $C_{0}$ and a free $\pi$-vertical conic.
\end{theo}

\begin{proof}
By Proposition \ref{prop:breakofflowdegree} we know that we can find a breaking $C \sim_{alg} C_{0} + T$ where the vertical curve $T$ is either a free $\pi$-vertical conic or a free $\pi$-vertical cubic.  In the first case we are done.  In the second case, let $\overline{P}$ denote the component of $\overline{M}_{0,0}(\mathcal{X})$ parametrizing the deformations of $T$.  By Lemma \ref{lemm:conicscubicsdescription} we see that every general fiber of $\pi$ contains a curve parametrized by $\overline{P}$ which is the union of a conic and a line.  Since such curves are smooth points of $\overline{P}$, in fact there is an entire component $\overline{Q}$ of the parameter space of $\pi$-vertical lines contained in smooth fibers such that every stable map parametrized by $\overline{Q}$ occurs as the restriction of a stable map parametrized by $\overline{P}$ to a subcurve.

By assumption the surface swept out by the curves parametrized by $\overline{Q}$ is big.  Thus every deformation of $C_{0}$ intersects this surface.  Consider the point of $\overline{P}$ obtained by gluing a general $\pi$-vertical line parametrized by $\overline{Q}$ to a $\pi$-vertical conic.  Then there is a free deformation $C_{1}$ of $C_{0}$ that intersects the image of this curve along the line.  By gluing, we obtain a stable map whose image has the form $C_{1} \cup \ell \cup T_{1}$ where $C_{1}$ is a free deformation of $C_{0}$, $\ell$ is a $\pi$-vertical line, and $T_{1}$ is a free $\pi$-vertical conic.  This is a smooth point of $\overline{M}_{0,0}(\mathcal{X})$ that lies in the same component of $\overline{M}_{0,0}(\mathcal{X})$ as $C_{0} \cup T$.  We can then glue $C_{1}$ to $\ell$ to conclude the desired statement.
\end{proof}

\subsection{Blow-ups of Fano threefolds}

One of the most common ways to construct a del Pezzo fibration is by blowing up a Fano threefold.  In this section we analyze several examples of this type.

\begin{exam} \label{exam:blowupcubic}
Let $X_{3}$ be a smooth cubic $3$-fold in $\mathbb{P}^{4}$.
Let $Z$ denote a smooth elliptic curve which is the base locus of a general pencil of hyperplane sections.  Let $\pi: \mathcal{X} \to \mathbb{P}^{1}$ denote the resolution of the pencil.  We have $\Pic(\mathcal{X}) = \mathbb{Z}H \oplus \mathbb{Z}E$ where $H$ is the pullback of the hyperplane class from $X_{3}$ and $E$ is the exceptional divisor of the blow-up $\phi: \mathcal{X} \to X_{3}$.  In this notation $-K_{\mathcal{X}/\mathbb{P}^{1}} = E$, so that $\neg(\mathcal{X},-K_{\mathcal{X}/\mathbb{P}^{1}}) = -1$ and the sections achieving this bound will exactly be the sections contained in $E \cong Z \times \mathbb{P}^{1}$.

\textbf{Computation of $\pi$-vertical curves:}  
The monodromy action on the N\'eron-Severi space of the general fiber is the full Weyl group $W(E_{6})$. Indeed, since our pencil is general, it will be a Lefschetz pencil in the sense of \cite[Definition 2.6]{Voi03}.  It follows from \cite[Proposition 2.27]{Voi03} that the primitive cohomology of a smooth member of the pencil is generated by vanishing spheres. The theory of Picard-Lefschetz tells us that vanishing spheres are roots and the monodromy is generated by orthogonal reflections with respect to these roots. Moreover since the monodromy action on the lattice generated by vanishing spheres is irreducible, the lattice generated by vanishing spheres is an irreducible root lattice. However, there is no irreducible root lattice contained in the lattice $E_6$ other than itself. Thus all roots are vanishing spheres, proving the claim. 
This implies that for $\pi$-vertical curves of anticanonical degree $1$ or $2$ there is a unique family of rational curves on $\mathcal{X}$ that dominates the base $\mathbb{P}^{1}$.

\textbf{Computation of the closed set:} Since any rigid section will satisfy $-K_{\mathcal{X}/\mathbb{P}^{1}} \cdot C \leq -2$, there are no rigid sections of $\pi$.

By Lemma \ref{lemm:genericainvfordp}, if $Y \subset \mathcal{X}$ is a subvariety with $a(Y_{\eta},-K_{\mathcal{X}/\mathbb{P}^{1}}) > a(\mathcal{X}_{\eta},-K_{\mathcal{X}/\mathbb{P}^{1}})$ then the intersection of $Y$ with a general fiber is a union of $(-1)$-curves.  Since the monodromy action is maximal, every $(-1)$-curve in a general fiber is equivalent to every other $(-1)$-curve under the monodromy action.  This implies that there is a unique irreducible surface $V$ whose intersection with a general fiber is a union of $(-1)$-curves. If $\widetilde{V}$ denotes a resolution of $V$ then the map $\widetilde{V} \to \mathbb{P}^{1}$ is not an algebraic fiber space, showing that $V$ can not be covered by sections.

On a cubic del Pezzo surface the subvarieties with the same $a$-value are exactly the conics.  Since the monodromy action is maximal, the set of all conics in general fibers is parametrized by an irreducible surface.  In particular, the Stein factorization of the evaluation map for the universal family of $\pi$-vertical conics is non-trivial.  This implies that any $1$-parameter family of conics that dominates the base $\mathbb{P}^{1}$ will fail to be an algebraic fiber space over $\mathbb{P}^{1}$ and thus can not be dominated by sections.

By combining the results above with Theorem \ref{theo:toomuchdeforming} we see that the only possible non-dominant families of sections will have height $\leq 0$. 
We analyze sections in this range below and we deduce that there is only one non-dominant family consisting of the sections contained in $E$.

\textbf{Computation of low degree sections:}
We analyze the sections of low height with respect to $-K_{\mathcal{X}/\mathbb{P}^{1}}$:
\begin{itemize}
\item Height -1: as discussed above these are exactly the rational curves in $E \cong Z \times \mathbb{P}^{1}$.
\item Height 0: these will be the strict transforms of lines in $X_{3}$ not meeting $Z$.  Thus this space is irreducible.
\item Height 1: these will be the strict transforms of conics in $X_{3}$ which meet $Z$ once.  \cite[Theorem 7.6]{LT17} shows that for the universal family of conics $\mathcal{U}$ on $X_{3}$ the evaluation map $s: \mathcal{U} \to X_{3}$ has equidimensional irreducible fibers over an open neighborhood of $Z$.  Thus the space of height $1$ sections is irreducible.
\item Height 2: these will be the strict transforms of twisted cubics in $X_{3}$ which meet $Z$ in two points. Any twisted cubic $C$ in $X_{3}$ is contained in a unique hyperplane section $H \subset \mathbb P^4$ and $C$ is a smooth rational curve of degree $3$ in the cubic surface $S = H \cap X_{3}$. Such a curve is the strict transform of a line via a birational morphism $S \rightarrow \mathbb P^2$. Since the space of twisted cubics in $X_{3}$ is irreducible (see \cite{HRS05}), we know that the monodromy action on classes of twisted cubics in $S$ is transitive as $H$ varies in an open set $U \subset (\mathbb P^4)^*$. Otherwise, each orbit corresponds to an irreducible component of the space of twisted cubics, contradicting with the irreducibility of such a space. Choose two points $p_1, p_2$ on $Z$ and consider the hyperplanes $H \subset \mathbb P^4$ containing $\ell = \mathrm{span}(p_1, p_2)$. For each $H$ there are only finitely many twisted cubics passing through $p_1, p_2$.  By the Lefschetz hyperplane theorem the monodromy action of the hyperplanes containing $\ell$ is still transitive, so these curves are parametrized by an irreducible variety. The irreducibility of the space of sections of height $2$ follows.
\end{itemize}

\textbf{Classification of sections:} Lemma \ref{lemm:improvedbounds} shows that $\MBB(\mathcal{X}) \leq 3$.  
We can now classify all families of sections. As explained above, there is a unique family of sections of height $-1$, $0$, $1$, and $2$.   We prove the irreducibility of the moduli space of sections of height $\geq 3$ by induction.  As discussed earlier we know that any component of $\Sec(\mathcal{X}/\mathbb{P}^{1})$ in this height range will parametrize a dominant family of sections.  By Lemma \ref{lemm:improvedbounds} and Theorem \ref{theo:anothermovablebandb} we know that for $d \geq 3$ any free section of height $d$ will deform as a stable map to the union of a free section of height $d-2$ and a free $\pi$-vertical conic.  So it suffices to show that if $M$ denotes the unique component of $\Sec(\mathcal{X}/\mathbb{P}^{1})_{d-2}$ then the space of curves obtained by gluing free $\pi$-vertical conics to curves in $M$ is also irreducible.  By the fact that the monodromy is maximal, we know that the space of conics which meet a fixed section is irreducible, and the desired conclusion follows immediately.
\end{exam}

An analogous argument holds for other similar constructions.

\begin{exam} \label{exam:blowupX5}  
Let $X_{5}$ be a general codimension $3$ linear section of $\mathrm{Gr}(2, 5)\subset \mathbb P^9$.  Let $Z$ denote a smooth elliptic curve which is the base locus of a general pencil of hyperplane sections.  Let $\pi: \mathcal{X} \to \mathbb{P}^{1}$ denote the resolution of the pencil.  We have $\Pic(\mathcal{X}) = \mathbb{Z}H \oplus \mathbb{Z}E$ where $H$ is the pullback of the hyperplane class and $E$ is the exceptional divisor of the blow-up $\phi: \mathcal{X} \to X_{5}$.  In this notation $-K_{\mathcal{X}/\mathbb{P}^{1}} = E$.

Since the fibration is constructed by resolving a Lefschetz pencil, the monodromy action will again be full.  Thus, by arguing just as in Example \ref{exam:blowupcubic}, we can classify all components of $\Sec(\mathcal{X}/\mathbb{P}^{1})$ by analyzing the behavior of sections of height $\leq 2$.

\textbf{Computation of low degree sections:}
We analyze the sections of low height with respect to $-K_{\mathcal{X}/\mathbb{P}^{1}}$:
\begin{itemize}
\item Height -1: these are the rational curves in $E \cong Z \times \mathbb{P}^{1}$.
\item Height 0: these will be the strict transforms of lines in $X_{5}$ not meeting $Z$.  As described by \cite{Iskov79}, this space is irreducible. 
\item Height 1: these will be the strict transforms of conics in $X_{5}$ which meet $Z$ once.  \cite[Theorem 7.6]{LT17} shows that for the universal family of conics $\mathcal{U}$ on $X_{5}$ the evaluation map $s: \mathcal{U} \to X_{5}$ has equidimensional irreducible fibers over an open neighborhood of $Z$.  Thus the space of height $1$ sections is irreducible.
\item Height 2: these will be the strict transforms of twisted cubics in $X_{5}$ which meet $Z$ in two points. Consider the incidence correspondence $I$ of pairs $(H,C)$ where $H$ is a hyperplane such that $S = H \cap X_{5}$ is smooth and $C$ is a twisted cubic in $S$.  Note that each $S$ is a degree $5$ del Pezzo surface and that $C$ is the pullback of a line under a birational map $S \to \mathbb{P}^{2}$.  The correspondence $I$ admits a map to an open subset $U \subset \mathbb{P}^{6*}$ and the monodromy action of $U$ on the classes of twisted cubics in a fixed surface $S$ is transitive. Choose two points $p_1, p_2$ on $Z$. They are parametrized by $\mathrm{Sym}^2(Z)$. Then we consider hyperplanes $H \subset \mathbb P^6$ containing $\ell = \mathrm{span}(p_1, p_2)$. For each $H$ there are only finitely many twisted cubics contained in $H \cap X_{5}$ passing through $p_1, p_2$.  By the Lefschetz hyperplane theorem the monodromy action of the hyperplanes containing $\ell$ is still transitive on the set of cubics in each surface $S$, so the set of twisted cubics meeting $Z$ in two points are parametrized by an irreducible variety.  The irreducibility of the space of sections of height $2$ follows.
\end{itemize}
By mimicking the argument in Example \ref{exam:blowupcubic}, we see that there is a unique irreducible family of sections of every height $\geq -1$.
\end{exam}

Similar computations hold for
\begin{itemize}
\item the blow up of the intersection of two quadrics along the base locus of a general pencil of hyperplanes, and
\item the blow up of a quadric along the base locus of a general pencil of quadrics.
\end{itemize}
In these cases the behavior of low degree sections is described by \cite{HT14}.

\subsection{Subvarieties of products of projective spaces}

Another common way of constructing a del Pezzo fibration is by taking a complete intersection inside a product of a threefold with a curve.  The following example illustrates how one can compute sections in this situation.

\begin{exam}
Let $\mathcal{X}$ be a general hypersurface in $\mathbb{P}^{1} \times \mathbb{P}^{3}$ of degree $(2,3)$.  The projection $\pi: \mathcal{X} \to \mathbb{P}^{1}$ is a fibration by cubic surfaces.  We denote the pullbacks of the hyperplane class from the two projections by $H_{1}$ and $H_{2}$ respectively.  Then $-K_{\mathcal{X}/\mathbb{P}^{1}} = H_{2}-2H_{1}$.

\textbf{Computation of $\pi$-vertical curves:}  
Let $U \subset \mathbb{P}^{19}$ denote the parameter space of smooth cubic hypersurfaces in $\mathbb{P}^{3}$.  \cite[Th\'eor\`eme 2]{Beauville86} shows that the action of the fundamental group of $U$ on the N\'eron-Severi space of the hypersurfaces $H$ is the entire Weyl group (denoted in Beauville's notation by $O_{h}(H^{2}(H,\mathbb{Z}))$).  By \cite[Corollary 6]{Kol15} a general conic curve $T \subset \mathbb{P}^{19}$ will have the property that $\pi_{1}(T \cap U) \to \pi_{1}(U)$ is surjective.  Since $\mathcal{X}$ is a family of cubic hypersurfaces defined by a general conic we deduce that the monodromy action for $\mathcal{X}$ is the entire Weyl group. 
This implies that the generic fiber does not admit a curve of anticanonical degree $1$ or $2$ that is defined over the ground field. 

\textbf{Computation of the closed set:} 
The minimal possible height of a section is $-2$, and any rigid section will have this height.  We will verify below that there is a finite set of sections of height $-2$ and all are rigid.

The union of the subvarieties of $\mathcal{X}$ with larger generic $a$-invariant is swept out by the $\pi$-vertical lines.  Since the monodromy action is full, this space is irreducible.  Similarly, the union of the subvarieties with the same generic $a$-invariant is swept out by the $\pi$-vertical conics.  Again this is an irreducible subset.  Since the Stein factorization of the map to $\mathbb{P}^{1}$ does not have connected fibers, this surface cannot admit a dominant family of sections.

By combining the results above with Theorem \ref{theo:toomuchdeforming} we see that the only possible non-dominant families of sections will have height $\leq 2$. We analyze sections in this range below.

\textbf{Computation of low degree sections:}
We analyze the sections of low height with respect to $-K_{\mathcal{X}/\mathbb{P}^{1}}$.  We claim that every component of the moduli space for the space of sections of height $\leq 2$ has the expected dimension and, if the height is between $-1$ and $2$, is also irreducible.  An easy dimension count proves that there are only finitely many sections of height $-2$.  A lengthier dimension count proves that a general section of height $d$ with $-1 \leq d \leq 2$ will map birationally onto a general rational curve of degree $d+2$ in $\mathbb{P}^{3}$, and that every component parametrizing such sections has the expected dimension.

It only remains to verify that for each $d \leq 2$ the space of sections of height $d$ which map birationally onto a general rational curve of degree $d+2$ in $\mathbb{P}^{3}$ is irreducible.  For some fixed $d$ let $N^{\circ}$ denote the space of such sections in $\mathbb{P}^{1} \times \mathbb{P}^{3}$ and let $\mathcal{C}$ denote the universal curve over $N^{\circ}$.  Consider the diagram
\begin{equation*}
\xymatrix{\mathcal{C} \ar[r]^-{s} \ar[d]_{p} &  \mathbb{P}^{1} \times \mathbb{P}^{3} \\
N^{\circ} & }
\end{equation*}
Consider $\mathcal{E} := p_{*}s^{*}\mathcal{O}(2,3)$.  Since $s$ has connected fibers we can identify sections of $s^{*}\mathcal{O}(2,3)$ with sections of $\mathcal{O}(2,3)$.  Thus a general section of $\mathcal{E}$ will parametrize the sublocus of curves parametrized by $N^{\circ}$ that are contained in a general section of $\mathcal{O}(2,3)$.  We claim that $\mathcal{E}$ is a globally generated vector bundle; this can be proved by writing down the ideal sheaf of a general curve $C$ parametrized by $N^{\circ}$ and computing that $H^{1}(\mathbb{P}^{1} \times \mathbb{P}^{3}, I_{C}(2,3)) =0$.  This means that a general section of $\mathcal{E}$ is irreducible.  In other words, the open dense locus of $\Sec(\mathcal{X}/\mathbb{P}^{1})_{d}$ parametrizing sections which map onto a general curve in $\mathbb{P}^{3}$ is irreducible for this degree range.

\textbf{Classification of sections:}  By Lemma \ref{lemm:improvedbounds} $\MBB(\mathcal{X}) \leq 3$.  As explained earlier, we can deduce that every component of the space of sections has the expected dimension, and that there is a unique component of any height $> -2$.
\end{exam}

\begin{exam}[Diagonal cubic surface (1,t,t+1,t+2)]
Consider the smooth threefold $\mathcal{X} \subset \mathbb{P}^{1}_{s,t} \times \mathbb{P}^{3}_{x,y,z,w}$ defined by the equation $sx^{3} + ty^{3} + (t+s)z^{3} + (t+2s)w^{3}=0$.  The projection map to $\mathbb{P}^{3}$ realizes $\mathcal{X}$ as the blow-up of the base locus of a pencil of cubic surfaces on $\mathbb{P}^{3}$.   We let $Z$ denote the   smooth genus $10$ curve which is the base-locus of this pencil and let $E$ denote the exceptional divisor.  The projection $\pi: \mathcal{X} \to \mathbb{P}^{1}$ is a cubic surface fibration.  
We will see below that the Galois action on $\mathrm{Pic}(\mathcal X_{\overline{\eta}})$ is $(\mathbb Z/3\mathbb Z)^3$ and under this action the $27$ lines split into orbits of size $(9,9,9)$.  It follows from \cite[the Appendix, No. 71]{Jahnel14} that $\mathcal{X}_{\eta}$ has Brauer group $\mathbb{Z}/3\mathbb{Z}$. However, since $\mathcal{X}$ is birational to $\mathbb{P}^{3}$ its Brauer group is trivial.  Let $H_{1}$ and $H_{2}$ denote the pullbacks of the hyperplane class under the two projection maps.  Then $-K_{\mathcal{X}/\mathbb{P}^{1}} = -H_{1} + H_{2}$ is relatively ample.

\textbf{Computation of $\pi$-vertical curves:}  
The Galois group acts on $\Pic(\mathcal{X}_{\overline{\eta}})$ through $(\mathbb{Z}/3\mathbb{Z})^{3}$  and the rank of $\Pic(\mathcal X_\eta)$ is $1$. Indeed, let $L = k(t^{1/3}, (t+1)^{1/3}, (t+2)^{1/3})$ be the splitting field of $\Pic(\mathcal{X}_{\overline{\eta}})$. Then the Galois action on $L$ is the action of $(\mathbb Z/3\mathbb Z)^3$ given by $(i, j, k) \cdot (t^{1/3},  (t+1)^{1/3}, (t + 2)^{1/3}) = (\zeta_3^i t^{1/3},  \zeta_3^j(t + 1)^{1/3}, \zeta_3^k(t + 2)^{1/3})$ where $\zeta_3$ is a cubic root of unity. There are three orbits of lines, each of size $9$. There are three orbits of classes of conics, each of size $9$, since every conic fibration has the form $X_{\overline{\eta}} \ni (x:y:z:w) \to (s^{1/3}x + t^{1/3}y: (t+s)^{1/3}z + (t + 2s)^{1/3}w) \in \mathbb P^1$ up to permutation and Galois action. This implies that $\mathcal{X}_{\eta}$ does not admit a curve of anticanonical degree $1$ or $2$ that is defined over the ground field.

The computation above shows that $\overline{X}$ admits three families of $\pi$-vertical lines which meet a general fiber.  There are also four singular fibers of $\pi$, each of which is a cone over a cubic curve.  There is a one-parameter family of $\pi$-vertical lines contained in every singular fiber of $\pi$.  In all seven families of $\pi$-vertical lines a general line has restricted tangent bundle $\mathcal{O}(2) \oplus \mathcal{O} \oplus \mathcal{O}(-1)$.

\textbf{Computation of the closed set:} 
A section must have intersection number $1$ against $H_{1}$.  So its degree under the pushforward to $\mathbb{P}^{3}$ will just be $1$ more than the height.  The lowest possible height is $-1$, and each such section is contracted by the map $\mathcal{X} \to \mathbb{P}^{3}$.  As explained above, we know that the generic fiber of $\pi$ has no line or conic defined over the ground field, so by Theorem \ref{theo:toomuchdeforming} we see that the only possible non-dominant families of sections will have height $\leq 0$.  We analyze sections in this range below.

\textbf{Computation of low degree sections:}
We next classify sections of low height.

\textbf{Height $-1$:} These will be the sections contained in $E$.  There is an irreducible one parameter family of such sections.

\textbf{Height $0$:} These will be strict transforms of lines in $\mathbb{P}^{3}$.  Such a line must meet every hyperplane in our family with multiplicity $2$ along $Z$, so these lines are the bisecants to $Z$.  They are parametrized by an irreducible space birational to $\Sym^{2}(Z)$.

\textbf{Height $1$:} These will be strict transforms of smooth conics which meet each cubic in our family with multiplicity $5$ along $Z$.  Each generic plane $P$ contains a finite number of such conics.  Indeed, $P \cap Z$ will be nine points, and any conic which contains 5 of these points will yield a section.  We claim there is a unique component of $\Sec(\mathcal{X}/\mathbb{P}^{1})$ whose generic section has this form.  Indeed, by \cite[Proposition 1.1]{Kad19} as we vary the plane $P$ the monodromy action on $P \cap Z$ is the full symmetric group.  But if there were more than one irreducible component then the monodromy action could not be transitive. 

We show there are no other components using a dimension count.  Suppose that $P$ is a plane containing a family of height $1$ sections which has dimension $\geq 1$.  This means that there is a length $\leq 4$ dimension $0$ subscheme of $P$ such that every conic containing this subscheme meets $Z$ with multiplicity $5$ along these points.  In particular the restriction of every cubic in our family to $P$ must be singular. However, there is no such plane $P$, since the tangent planes of our cubics at each point $z \in Z$ vary.

\textbf{Height $2$:} These will be strict transforms of cubics in $\mathbb{P}^{3}$.  We first show that there is a unique component $M \subset \Sec(\mathcal{X}/\mathbb{P}^{1})$ whose general curve $C$ corresponds to a twisted cubic $C'$ in $\mathbb{P}^{3}$.  Our strategy is to show that any such component $M$ will contain broken curves of a certain type.  Recall that there are two irreducible components $M_0, M_1$ of the moduli space of rational curves which parametrize free sections of height $0$ and height $1$ respectively.  We let $L_0, L_1, L_2$ denote the three components which parametrize a family of $\pi$-vertical lines which dominate $\mathbb{P}^{1}$, and we let $S_1, S_2, S_3, S_4$ denote the four components which parametrize vertical lines in the $4$ singular fibers of $\pi$.  Finally are three components $N_0, N_1, N_2$ parametrizing vertical free conics.

Let $M$ be any component of $\Sec(\mathcal{X}/\mathbb{P}^{1})$ consisting of the strict transforms of twisted cubics $C'$.  Then the general $C'$ meets with $Z$ at $8$ points.  Fix one general point $x$ on $\mathcal X$ and one general point $z$ on $Z$; the set of sections $C$ passing through $x$ and meeting the exceptional curve in $\mathcal{X}$ corresponding to $z$ in $Z$ form a family of dimension at least $1$.   The pushforwards of deformations of $C$ to $\mathbb{P}^{3}$ are twisted cubics $C'$ meeting the $2$ fixed points $x,z$.  Applying Bend and Break to this family of twisted cubics we obtain a breaking curve $\sum_{i = 1}^k C'_i$ in $\mathbb P^3$.  The breaking curve must be either the union of a line and a conic or the union of three lines.  However a connected chain of three lines cannot meet $x$ and $z$ while meeting $Z$ at $7$ other points.  So the broken curve must be a line plus a conic. There are two possibilities: (i) the line meets with $Z$ at $3$ points and the conic meets with $Z$ at $5$ points; (ii) the line meets with $Z$ at $2$ points and the conic meets with $Z$ at $6$ points. Going back to $\mathcal X$, this means that we break a curve $C$ to either (i) a vertical line and a section of height $1$, or (ii) a vertical conic and a section of height $0$. Moreover by construction our breaking curve must meet with $x$ and with $1$ general section of height $-1$, and in fact by Lemma \ref{lemm:strongerbandb} each component of the broken curve must meet either $x$ or the general section.

Next we show that both types of breaking curve are smooth points of $\overline{M}_{0,0}(\mathcal X)$.
In case (i), there are only finitely many $\pi$-vertical lines meeting a fixed general section of height $-1$.  If the section is general then the line must also be general so its normal bundle has the form $\mathcal{O} \oplus \mathcal{O}(-1)$.  Furthermore, the sections of height $1$ passing through $x$ form a $1$ dimensional locus. Among them only finitely many meet with a fixed general vertical line.  This implies that our section of height $1$ must be free.  We conclude that the broken curve is a smooth point of the moduli space.
In case (ii), $x$ meets with either a section of height $0$ or a vertical conic.  If the former holds, there are finitely many sections of height $0$ passing through $x$ and by generality any such conic must be free. Then there are only finitely many choices of a $\pi$-vertical conic meeting fixed general sections of height $0$ and height $-1$, so the conic must also be free.
If the latter holds, there are only finitely many choices for a vertical conic and by generality it must be free. Then there are only finitely many sections of height $0$ meeting with $1$ general section of height $-1$ and with one of our chosen conics, so the section must also be free.

Before continuing, we note that for any $j=1,2,3,4$ there is a unique component of $M_1^{(1)} \times_{\mathcal X} S_j^{(1)}$ which generically parametrizes the union of a free section and a general line in $S_{j}$.  Indeed, since any section meets the corresponding singular fiber of $\pi$ in one point, the parameter space of such glued curves is birational to $M_{1}$.

Suppose that $M$ contains the union of a general vertical line in some $L_i$ and a free section of height $1$.  Then the locus of broken curves in $M$ of this type must admit a dominant map to $M_1$.  Since the space of height $1$ sections is irreducible, its closure in $\overline{M}_{0,0}(\mathcal{X})$ contains a two-dimensional locus parametrizing unions of a free height $0$ section and a general line in $S_{1}$.  Correspondingly, the closure of $M$ in $\overline{M}_{0,0}(\mathcal{X})$ parametrizes a curve which is a chain $\ell_{1} \cup C_{0} \cup \ell_{2}$ where $\ell_{1}$ is a general line in $S_{1}$, $C_{0}$ is a free height $0$ section, and $\ell_{2}$ is a line in $L_{i}$ which must be general by a dimension count. 
Note that this is a smooth point of $\overline{M}_{0,0}(\mathcal{X})$.  We can smooth $C_{0} \cup \ell_{2}$, and furthermore we can deform $\ell_{1}$ and the intersection point along with the smoothing.  Thus we see that the closure of $M$ in $\overline{M}_{0,0}(\mathcal{X})$ must contain $M_1^{(1)} \times_{\mathcal X} S_1^{(1)}$.  Using analogous reasoning, it also contains $M_1^{(1)} \times_{\mathcal X} S_j^{(1)}$ for any $j=1,2,3,4$.

Suppose that $M$ contains the union of a general vertical line in some $S_{j}$ and a free section of height $1$.  Since this is a smooth point of $\overline{M}_{0,0}(\mathcal{X})$, $M$ must contain the entire locus $M_1^{(1)} \times_{\mathcal X} S_j^{(1)}$.

Finally, suppose that $M$ contains the union of a general vertical conic in some $N_i$ and a free section of height $0$.  Then the locus of broken curves in $M$ of this type must admit a dominant map to $M_0$.  Since the space of height $0$ sections is irreducible, its closure in $\overline{M}_{0,0}(\mathcal{X})$ contains a 1-dimensional locus parametrizing unions of a general height $-1$ section and a general line in $S_{1}$.  
Since this locus is $1$-dimensional, an incidence correspondence argument shows that the closure of $M$ in $\overline{M}_{0,0}(\mathcal{X})$ parametrizes a chain $\ell \cup C_{0} \cup T$ where $\ell$ is a general line in $S_{1}$, $C_{0}$ is a general height $-1$ section, and $T$ is a free conic. 
This is a smooth point of $\overline{M}_{0,0}(\mathcal{X})$.
Note that we can smooth $C_{0} \cup T$, and furthermore we can deform $\ell$ and the intersection point along with the smoothing.  Thus we see that the closure of $M$ in $\overline{M}_{0,0}(\mathcal{X})$ must contain $M_1^{(1)} \times_{\mathcal X} S_1^{(1)}$.  Using analogous reasoning, it also contains $M_1^{(1)} \times_{\mathcal X} S_j^{(1)}$ for any $j=1,2,3,4$.

We have shown that any component $M$ parametrizing the strict transform of twisted cubics must contain $M_1^{(1)} \times_{\mathcal X} S_j^{(1)}$ for some $j=1,2,3,4$.  Furthermore, there must be at least one component which contains a smoothing of a general height $1$ section and a general line in $L_{i}$.  This particular component contains $M_1^{(1)} \times_{\mathcal X} S_j^{(1)}$ for every $j$.  Since for each $j$ the space $M_1^{(1)} \times_{\mathcal X} S_j^{(1)}$ is irreducible and contains smooth points of $\overline{M}_{0,0}(\mathcal{X})$, altogether we deduce that there is a unique such component $M$.

We next show that there is no component $M \subset \Sec(\mathcal{X}/\mathbb{P}^{1})$ whose general point corresponds to a planar cubic $C$ in $\mathbb{P}^{3}$.  Let $P$ be the plane containing $C$.  If $P$ is generic, then either $C$ contains $8$ points of $P \cap Z$ or the singular point of $C$ lies on $P \cap Z$ and $C$ must contain $6$ other points of $P \cap Z$.  Note that there are only finitely many possible such $C$ in any fixed $P$.  Thus a dimension count shows that the collection of such curves cannot be dense in a component of $\Sec(\mathcal{X}/\mathbb{P}^{1})$.  Similar dimension counts work for non-generic $P$ to show that we get no new components in this way.

\textbf{Classification of sections:}
By Lemma \ref{lemm:improvedbounds} $\MBB(\mathcal{X}) \leq 3$. Let $d \geq 3$ and assume that for every $-1 \leq k < d$ there is a unique component $M_k$ parametrizing free sections of height $k$. Suppose we fix a component $M$ of $\Sec(\mathcal{X}/\mathbb{P}^{1})$ generically parametrizing free sections of height $d \geq 3$.
Then $\overline{M}$ contains a stable map consisting of a free section $C_{d-2}$ of height $d-2$ and a free conic $T$.  Fix some $j \in \{1,2,3\}$. Since the space of height $d-2$ sections is irreducible, a free section of height $d-2$ can be deformed into the union of a general section $C_{d-4}$ of height $d-4$ and a general vertical free conic $T'$ in the component $N_j$ parametrizing vertical conics.   Furthermore since $T$ is free we can deform $T$ and the attachment point to $C_{d-2}$ along with the deformation above to obtain a chain $T \cup C_{d-4} \cup T'$ where $T$ is still a free vertical conic.  Then we smooth $T \cup C_{d-4}$ and obtain an element of $M_{d-2}^{(1)} \times_{\mathcal X} N_j^{(1)}$ which is a smooth point of $\overline{M}_{0,0}(\mathcal X)$. Since $M_{d-2}^{(1)} \times_{\mathcal X} N_j^{(1)}$ is irreducible, we conclude that $M$ is unique.
\end{exam}

\section{Bounds in Geometric Manin's Conjecture}
\label{sect:Manin}

\subsection{Algebraic and numerical equivalence for curves}

We start by recalling some facts about algebraic and numerical equivalence for curve classes on del Pezzo fibrations.

\subsubsection{Numerical equivalence}
By dualizing the surjective restriction map $N^{1}(\mathcal{X})_{\mathbb{R}} \to N^{1}(\mathcal{X}_{\eta})_{\mathbb{R}}$ we obtain an inclusion $N_1(\mathcal X_\eta)_{\mathbb{R}} \subset N_1(\mathcal X)_{\mathbb{R}}$.  Let $i: F \hookrightarrow \mathcal{X}$ denote the inclusion of a general fiber of $\pi$.  By applying the Invariant Cycle Theorem (see \cite[Th\'eor\`eme (4.1.1).(ii)]{Deligne} and the following discussion) to divisors and dualizing, we see that
\begin{equation*}
N_{1}(\mathcal{X}_{\eta})_{\mathbb{R}} = i_{*}(N_{1}(F)_{\mathbb{R}}) = i_{*}(N_{1}(F)_{\mathbb{R}}^{mon})
\end{equation*}
where $N_{1}(F)_{\mathbb{R}}^{mon}$ denotes the monodromy invariant part of $N_{1}(F)_{\mathbb{R}}$. Under this identification we have $\Nef_{1}(\mathcal{X}_{\eta}) = i_{*}\Nef_{1}(F)$.

The behavior of lattices of curves in these spaces is more subtle.  We have inclusions
\begin{equation*}
N_{1}(\mathcal X_{\eta})_{\mathbb{Z}} \subset i_{*}(N_{1}(F)_{\mathbb{Z}}) \subset N_{1}(\mathcal{X})_{\mathbb{Z}} \cap N_{1}(\mathcal{X}_{\eta})_{\mathbb{R}}
\end{equation*}
as full-rank sublattices.  While the first inclusion can be strict, the second inclusion is an equality for del Pezzo fibrations which have a smooth total space.

\begin{prop}
Let $\pi: \mathcal{X} \to \mathbb{P}^{1}$ be a Fano fibration such that $\mathcal X$ is smooth.  Let $F$ denote a general fiber of $\pi$.  Suppose that the pairing of $N^{1}(F)_{\mathbb{Z}}$ and $N_{1}(F)_{\mathbb{Z}}$ is unimodular.  Then $i_{*}(N_{1}(F)_{\mathbb{Z}}) = N_{1}(\mathcal{X})_{\mathbb{Z}} \cap N_{1}(\mathcal{X}_{\eta})_{\mathbb{R}}$.
\end{prop}

\begin{proof} 
First note that the restriction map $N^1(\mathcal X)_{\mathbb Z} \twoheadrightarrow N^1(\mathcal{X}_\eta)_{\mathbb Z}$ is surjective, since any Cartier divisor on $\mathcal{X}_{\eta}$ can be extended to a Weil divisor on $\mathcal{X}$ and $\mathcal{X}$ is smooth.  The Hochschild-Serre spectral sequence for $H^{2}(\mathcal{X}_{\eta},\mathbb{G}_{m})$ (as in \cite[Corollary 6.7.8]{Poonen}) reads
\begin{equation*}
0 \to \Pic(\mathcal{X}_{\eta}) \to \Pic(\mathcal{X}_{\overline{\eta}})^{\Gal(k(\mathbb{P}^{1}))} \to \Br(k(\mathbb{P}^1)) \to H^{2}(\mathcal{X}_{\eta},\mathbb{G}_{m})
\end{equation*}
Since the Brauer group $\mathrm{Br}(k(\mathbb P^1))$ is equal to $0$, we conclude that $N^1(\mathcal{X}_\eta)_{\mathbb Z} \cong N^{1}(\mathcal{X}_{\overline{\eta}})_{\mathbb{Z}}^{\Gal(k(\mathbb{P}^{1}))}$. This latter group is in turn isomorphic to $N^1(F)_{\mathbb Z}^{mon}$. Indeed, we have $N^1(F)_{\mathbb Z} \cong N^{1}(\mathcal{X}_{\overline{\eta}})_{\mathbb{Z}}$ and the Galois group is acting on this space via through the monodromy action. Combining, we have maps
\[
N^1(\mathcal X)_{\mathbb Z} \twoheadrightarrow N^1(\mathcal{X}_\eta)_{\mathbb Z} \cong N^1(F)_{\mathbb Z}^{mon} \hookrightarrow N^1(F)_{\mathbb Z}
\]
where the first map is the pullback and the third is the inclusion.
Taking the dual we obtain
\[
N_1(F)_{\mathbb Z} \twoheadrightarrow (N^1(F)_{\mathbb Z}^{mon})^{\vee} \cong N^1(\mathcal{X}_\eta)_{\mathbb Z}^\vee \hookrightarrow N^1(\mathcal X)_{\mathbb Z}^\vee
\]
Finally we note that $N^1(\mathcal{X}_\eta)_{\mathbb Z}^\vee = N_1(\mathcal{X}_\eta)_{\mathbb{R}} \cap N^1(\mathcal X)_{\mathbb Z}^\vee$.  Indeed, since $N^1(\mathcal X)_{\mathbb Z}$ is torsion free, the kernel of $N^1(\mathcal X)_{\mathbb Z} \twoheadrightarrow N^1(\mathcal{X}_\eta)_{\mathbb Z}$ is also a free abelian group.  Thus the cokernel of $N^1(\mathcal{X}_\eta)_{\mathbb Z}^\vee \hookrightarrow N^1(\mathcal X)_{\mathbb Z}^\vee$ is torsion free.  This proves that $N_{1}(\mathcal{X}_{\eta})_{\mathbb{R}} \cap N_{1}(\mathcal{X})_{\mathbb{Z}} \subset i_{*}(N_{1}(F)_{\mathbb{Z}})$ as desired.
\end{proof}

\subsubsection{Algebraic equivalence}
For a smooth rationally connected threefold $\mathcal{X}$ over $\mathbb{C}$ we have
\begin{equation*}
\Br(\mathcal{X}) \cong H^{3}(\mathcal{X}, \mathbb{Z})_{\mathrm{tors}}.
\end{equation*}
According to the universal coefficient theorem for cohomology (as in \cite[Theorem 3.2]{Hatcher}), we can equally well think of $\Br(\mathcal{X})$ as the torsion classes of $H_{2}(\mathcal{X},\mathbb{Z})$.  Let $Q_{1}(\mathcal{X})$ denote the set of algebraic equivalence classes of curves of $\mathcal{X}$.  \cite[Theorem 1]{BS83} shows that algebraic and homological equivalence coincide for curve classes on $\mathcal{X}$ and \cite[Theorem 2]{Voisin06} proves the integral Hodge conjecture for $\mathcal{X}$.  Together these show:

\begin{theo}[\cite{BS83}, \cite{Voisin06}] \label{theo:algequiv}
Let $\mathcal{X}$ be a smooth rationally connected threefold over $\mathbb{C}$.  Then $|\Br(\mathcal{X})|$ is the size of the kernel of the quotient map $q: Q_{1}(\mathcal{X}) \to N_{1}(\mathcal{X})_{\mathbb{Z}}$.
\end{theo}

\subsection{Counting components} \label{sect:countingcomponents}
The key to Geometric Manin's Conjecture is to count the number of components of $\Sec(\mathcal{X}/\mathbb{P}^{1})$.  We separate this into two questions:
\begin{enumerate}
\item Which numerical classes are represented by a section?
\item How many components of $\Sec(\mathcal{X}/\mathbb{P}^{1})$ represent a fixed numerical class?  How many such components parametrize a dominant family of curves?
\end{enumerate}
In this section we analyze these questions using heuristic arguments.  The goal is to develop a precise conjecture describing what we will need for Geometric Manin's Conjecture. For simplicity we assume that $\mathcal X$ is smooth.

For Question (1) the key input is the Weak Approximation Conjecture.  A section can only intersect components of fibers which occur with multiplicity $1$.  Conversely, since the Brauer group of $k(\mathbb{P}^{1})$ is trivial the Weak Approximation Conjecture predicts that this is the only restriction on the possible intersection numbers with components of fibers (see e.g.~\cite[Conjecture 2]{HT06}).
Let $\Gamma_{\mathcal X}$ be the (finite) set of all possible combinations of intersection numbers of a section against $\pi$-vertical divisors, i.e., the set parametrizing the ways of choosing one component of multiplicity $1$ in every fiber of $\pi$.  We call an element $\lambda$ of $\Gamma_{\mathcal X}$ an ``intersection profile''.  We will let $N_{\lambda} \subset N_{1}(\mathcal{X})_{\mathbb{R}}$ denote the affine linear subspace of classes with a given intersection profile $\lambda$.  Note that once we have one section with a given intersection profile, we can obtain many more by gluing on $\pi$-vertical free curves.

We next turn to Question (2).  Each component of $\Sec(\mathcal{X}/\mathbb{P}^{1})$ naturally determines an algebraic equivalence class.  Conversely, we expect that each ``sufficiently positive'' algebraic equivalence class is represented by at most one dominant family of sections.  By combining with Theorem \ref{theo:algequiv}, we expect that each  ``sufficiently positive'' numerical equivalence class is represented by $|\Br(\mathcal{X})|$ different dominant families of sections.  Precisely:

\begin{conj} \label{conj:irreducibility}
Let $\pi: \mathcal{X} \to \mathbb{P}^{1}$ be a del Pezzo fibration such that $\mathcal X$ is smooth.  Fix an intersection profile $\lambda$ and let $\Nef_{\lambda} = \Nef_{1}(\mathcal{X}) \cap N_{\lambda}$.  Then:
\begin{enumerate}
\item There is an upper bound on the number of dominant families of sections representing any fixed numerical class in $\Nef_{\lambda}$. 
\item There is some translate $\mathcal{T}$ of $\Nef_{\lambda}$ in $N_{\lambda}$ such that every numerical class in $\mathcal{T}_{\mathbb{Z}}$ is represented by exactly $|\Br(\mathcal{X})|$ different dominant families of sections.
\end{enumerate}
\end{conj}

The following example shows that when $|\Br(\mathcal{X})| > 1$ we can indeed find multiple dominant components of $\Sec(\mathcal{X}/\mathbb{P}^{1})$ representing some sufficiently positive numerical classes.

\begin{exam}
Consider the Artin-Mumford example of a unirational threefold $X$ with non-trivial Brauer group from \cite{AM72}.  The construction starts with a linear subspace $\mathbb{P}^{3}$ in the parameter space for quadrics on $\mathbb{P}^{3}$.  We then take a double cover $Y \to \mathbb{P}^{3}$ ramified over the locus of singular quadrics.  $Y$ has $10$ double point singularities, and by resolving them we obtain our threefold $X$.

Let $p \in \mathbb{P}^{3}$ denote the image of one of the singular points on $Y$.  By composing the map $X \to \mathbb{P}^{3}$ with projection from $p$, we obtain a morphism $g: X \to \mathbb{P}^{2}$.  This is generically a conic bundle, and \cite[Section 2]{AM72} shows that the discriminant locus consists of two elliptic curves $D_{1}, D_{2} \subset \mathbb{P}^{2}$.  Furthermore, if we take one component $T_{1}$ of a reducible conic over $D_{1}$ and one component $T_{2}$ of a reducible conic over $D_{2}$, then $T_{1} - T_{2}$ is a class that is numerically equivalent to $0$ but not algebraically equivalent to $0$ on $X$.

Now compose the map $X \to \mathbb{P}^{3}$ with projection from a general line.  By resolving this map, we obtain $\pi: \mathcal{X} \to \mathbb{P}^{1}$.  The general fiber is a double cover of $\mathbb{P}^{2}$ ramified along a smooth quartic, so that $\pi$ gives $\mathcal{X}$ the structure of a degree $2$ del Pezzo fibration.  The curves $T_{1}, T_{2}$ above yield two sections of $\pi$ that are numerically equivalent but not algebraically equivalent on $\mathcal{X}$.   If we glue these two curves to any sufficiently positive $\pi$-vertical curve and smooth, we will find two different dominant components of $\Sec(\mathcal{X}/\mathbb{P}^{1})$ representing different algebraic classes but the same numerical class.
\end{exam}

\subsection{Geometric Manin's Conjecture} 
Throughout this section $\pi: \mathcal{X} \to \mathbb{P}^{1}$ denotes a del Pezzo fibration polarized by $-K_{\mathcal{X}/\mathbb{P}^{1}}$ such that $\mathcal{X}$ is smooth.  For simplicity, we will assume that the general fibers are not isomorphic to $\mathbb{P}^{2}$ or $\mathbb{P}^{1} \times \mathbb{P}^{1}$.  (In these two cases one must adjust the counting function slightly to reflect the fact that the general fiber does not contain any curve class with anticanonical degree $1$.)  

Recall that we have the following definition of our counting function:

\begin{defi}
Fix a real number $q>1$.  For any positive integer $d$ define
\begin{equation*}
N(\mathcal{X},-K_{\mathcal{X}/\mathbb{P}^{1}},q,d) := \sum_{i = 1}^{d} \sum_{M \in \Manin_{i}} q^{\dim M}.
\end{equation*}
For a definition of $\Manin_i$ see Definition~\ref{def:Manincomp}.
\end{defi}

For the remainder of the section, we study the asymptotic growth rate of $N(\mathcal{X},-K_{\mathcal{X}/\mathbb{P}^{1}},q,d)$ by modifying Batyrev's heuristic for the relative setting.  We will need the following constants:

\begin{defi}
Let $\pi: \mathcal{X} \to \mathbb{P}^{1}$ be a del Pezzo fibration.  We equip $N_{1}(\mathcal X_{\eta})_{\mathbb{R}}$ with the Lebesgue measure $\mu$ such that the fundamental domain for the lattice $N_1(\mathcal X_\eta)_{\mathbb Z}$ has volume $1$. Then we define the alpha-constant of $\mathcal X_\eta$ by 
\[
\alpha(\mathcal{X}_\eta, -K_{\mathcal{X}/\mathbb{P}^{1}}) := \dim N_1(\mathcal X_\eta)_{\mathbb{R}} \cdot \mu(\Nef_1(\mathcal X_\eta)\cap \{ \gamma \in N_1(\mathcal X_\eta)_{\mathbb{R}} \, | -K_{\mathcal X_\eta} \cdot \gamma \leq 1\}).
\]

\end{defi}

\begin{defi}
Let $\pi: \mathcal{X} \to \mathbb{P}^{1}$ be a del Pezzo fibration such that $\mathcal X$ is smooth.  We define
\begin{equation*}
\tau_{\mathcal{X}} = |\Gamma_{\mathcal X}| \cdot [N_{1}(\mathcal{X})_{\mathbb{Z}} \cap N_{1}(\mathcal{X}_{\eta})_{\mathbb{R}} : N_1(\mathcal{X}_\eta)_{\mathbb Z}]
\end{equation*}
where $\Gamma_{X}$ denotes the set of allowable intersection profiles for $\mathcal{X}$ as in Section \ref{sect:countingcomponents}.
\end{defi}

Putting everything together, Theorem \ref{theo:maintheorem1} and Theorem \ref{theo:maintheorem2} imply the following asymptotic formula.

\begin{theo} \label{theo:asymptoticformula}
Let $\pi: \mathcal{X} \to \mathbb{P}^{1}$ be a del Pezzo fibration such that $\mathcal X$ is smooth, $-K_{\mathcal{X}/\mathbb{P}^{1}}$ is relatively ample, and the general fiber is a del Pezzo surface of degree $\geq 2$ that is not $\mathbb{P}^{2}$ or $\mathbb{P}^{1} \times \mathbb{P}^{1}$.
Assume that Conjecture \ref{conj:irreducibility} holds for every intersection profile $\lambda$.  Then 
\begin{equation*}
N(\mathcal{X},-K_{\mathcal{X}/\mathbb{P}^{1}},q,d) \mathrel{\mathop{\sim}_{\mathrm{d \to \infty}}} \left( \tau_X \cdot \alpha (\mathcal{X}_\eta, -K_{\mathcal{X}/\mathbb{P}^{1}}) \cdot |\Br(\mathcal{X})| \cdot \frac{q}{q-1} \right) q^{d} d^{\rho(\mathcal{X}_{\eta})-1}.
\end{equation*}
\end{theo}

The restriction on the degree of the general fiber is not necessary: we impose this restriction in order to apply Lemma \ref{lemm:nefconedelpezzo}, but with more care one can prove a similar statement for degree $1$ del Pezzo fibrations.  (After appropriate modifications we can also find a formula for $\mathbb{P}^{2}$ or $\mathbb{P}^{1} \times \mathbb{P}^{1}$ fibrations.)

We expect the same asymptotic formula to hold even when the relative canonical divisor is not relatively ample.

\begin{proof}
For convenience we let $D$ denote the discriminant of $N_1(\mathcal{X}_\eta)_{\mathbb Z}$ in $N_{1}(\mathcal{X})_{\mathbb{Z}} \cap N_{1}(\mathcal{X}_{\eta})_{\mathbb{R}}$.

By Theorem \ref{theo:maintheorem1} each Manin component of sufficiently high degree parametrizes a dominant family of sections.  Thus we can estimate $N(\mathcal{X},-K_{\mathcal{X}/\mathbb{P}^{1}},q,d)$ by counting how many dominant families of sections parametrize curve classes of height at most $d$.  We execute this plan for each intersection profile $\lambda$ separately and then add the contributions at the end.

Recall that for each intersection profile $\lambda \in \Gamma_X$ there is a unique translate $N_{\lambda}$ of $N_1(\mathcal X_\eta)$ in $N_1(\mathcal X)$ parametrizing classes with this intersection profile.  For each $\lambda$ we will fix an isomorphism $\psi$ between $N_1(\mathcal X_\eta)$ and $N_{\lambda}$ by mapping $0$ to a some fixed $\mathbb{Z}$-class $v_{\lambda} \in N_{\lambda}$ and extending linearly.  Under this isomorphism $N_{1}(\mathcal{X}_{\eta})_{\mathbb{Z}}$ is identified with a full-rank sublattice of $N_{\lambda,\mathbb{Z}} := N_{\lambda} \cap N_{1}(\mathcal{X})_{\mathbb{Z}}$ of index $D$.

Let $\Sigma_\lambda$ denote the set of dominant components of $\Sec(\mathcal{X}/\mathbb{P}^{1})$ parametrizing sections with intersection profile $\lambda$.  By composing the map to $N_{1}(\mathcal{X})_{\mathbb{R}}$ with $\psi^{-1}$, we obtain
\[
\phi : \Sigma_\lambda \rightarrow N_1(\mathcal X_\eta)_{\mathbb{R}}.
\]
Fix any coset $\Xi$ of $N_{1}(\mathcal{X}_{\eta})_{\mathbb{Z}}$ in $N_{1}(\mathcal{X}_{\eta})_{\mathbb{R}} \cap N_{1}(\mathcal{X})_{\mathbb{Z}}$.  We claim that there is a movable section with class in $\phi^{-1}\Xi$.  Indeed since a general fiber $F$ is a del Pezzo surface $N_{1}(F)_{\mathbb{Z}}$ is generated by curves that are movable in $F$.  By \cite[\S 3 Proposition 3]{Khovanskii92} there is a set $\mathcal{T} \subset N_{1}(F)_{\mathbb{R}}$ which is a translate of a full dimensional cone such that every lattice point  $\mathcal{T}_{\mathbb{Z}}$ can be written as a non-negative sum of elements in this generating set.  By Conjecture \ref{conj:irreducibility} we know that there is a movable section $C$ with intersection profile $\lambda$.  Since $i_{*}(N_{1}(F)_{\mathbb{Z}}) = N_{1}(\mathcal{X}_{\eta})_{\mathbb{R}} \cap N_{1}(\mathcal{X})_{\mathbb{Z}}$ there is an element $\beta \in \mathcal{T}_{\mathbb{Z}}$ such that the sum of $\beta$ and the class of $C$ lies in $\phi^{-1}\Xi$.  By gluing $C$ to a union of free curves representing $\beta$ we find a movable section in $\phi^{-1}\Xi$.

We next claim that there are classes $v_{1},v_{2} \in \psi^{-1}(\Xi)$ such that
\begin{equation*}
v_{1} + \Nef_{1}(\mathcal{X}_{\eta})_{\mathbb{Z}} \subset \phi(\Sigma_{\lambda}) \cap \Xi \subset v_{2} + \Nef_{1}(\mathcal{X}_{\eta})_{\mathbb{Z}}.
\end{equation*}
To see the first inclusion, let $v_{1}$ be the class of a movable section in $\phi(\Sigma_{\lambda}) \cap \Xi$.  Lemma \ref{lemm:nefconedelpezzo} implies that for a general fiber $F$ every element in $\Nef_{1}(F)_{\mathbb{Z}}$ is represented by a union of free curves.  Using a gluing argument, we see that every element of $v_{1} +\Nef_{1}(\mathcal{X}_{\eta})_{\mathbb{Z}}$ is represented by a dominant family of free sections.

To see the second inclusion, note that Movable Bend and Break (more precisely Theorem~\ref{theo:precisemovablebandb}) shows that $\phi(\Sigma_{\lambda})$ is contained in the sum of $\Nef_{1}(\mathcal{X}_{\eta}) \cap N_{1}(\mathcal{X})_{\mathbb{Z}}$ with a finite set of classes.  This implies that $\phi(\Sigma_{\lambda}) \cap \Xi$ is contained in the sum of $\Nef_{1}(\mathcal{X}_{\eta})_{\mathbb{Z}}$ with a finite set of classes.   Thus $\phi(\Sigma_\lambda) \cap \Xi$ is contained in a translate of $\Nef_1(\mathcal X_\eta)_{\mathbb{Z}}$.

For each $\lambda \in \Gamma_{X}$, we can determine the contribution of $\lambda$ to the counting function $N(\mathcal{X},-K_{\mathcal{X}/\mathbb{P}^{1}},q,d)$ by taking a sum over the lattice points $\phi(\Sigma_{\lambda}) \subset N_{1}(\mathcal{X}_{\eta})_{\mathbb{R}}$, which can further be subdivided as a sum over cosets of $N_{1}(\mathcal{X}_{\eta})_{\mathbb{Z}}$.  As we determined above, for any coset $\Xi$ the set $\phi(\Sigma_{\lambda}) \cap \Xi$ is sandwiched between two translates of $\Nef_{1}(\mathcal{X}_{\eta})_{\mathbb{Z}}$.   We let $\mathcal{S}$ denote the set of numerical classes contained in the larger translate but not the smaller one. After perhaps increasing $\mathcal{S}$, by Conjecture \ref{conj:irreducibility} (2) we may also suppose that every fiber of $\phi$ over a point not in $\mathcal{S}$ is either empty or has size exactly $|\Br(\mathcal{X})|$.

Note that by Conjecture \ref{conj:irreducibility} (1) the contributions of the irreducible components in $\mathcal{S}$ to the counting function is asymptotically negligible.  In particular, the asymptotic behavior of the counting function is unchanged if we assume that each element of $\phi(\Sigma_{\lambda})$ contributes exactly $|\Br(\mathcal{X})| q^{\dim M}$ to the sum.

We can now estimate the contribution of sections of intersection type $\lambda$ to the asymptotic formula using the usual Tauberian methods for linear functionals on the cone $\Nef_{1}(\mathcal{X}_{\eta})$.  Using \cite[Corollary A.5]{CLT10}, each intersection profile $\lambda$ gives an asymptotic contribution of 
\begin{equation*}
D \cdot \alpha (\mathcal{X}_\eta, -K_{\mathcal{X}/\mathbb{P}^{1}}) \cdot |\Br(\mathcal{X})| \cdot \frac{q}{q-1} q^{d} d^{\rho(\mathcal{X}_{\eta})-1}.
\end{equation*}
We then multiply by $|\Gamma_{X}|$ to add up the contributions from every $\lambda \in \Gamma_{X}$, yielding the formula in the statement of the theorem.
\end{proof}

\subsection{Upper bounds}
Our results allow us to prove upper bounds on the counting function which approximate the expected form.

\begin{theo}
Let $\pi: \mathcal{X} \to \mathbb{P}^{1}$ be a del Pezzo fibration such that $-K_{\mathcal{X}/\mathbb{P}^{1}}$ is relatively ample and $\mathcal{X}$ is smooth. Then there is a positive integer $r$ such that
\begin{equation*}
N(\mathcal{X},-K_{\mathcal{X}/\mathbb{P}^{1}},q,d) = \mathcal{O}( q^{d} d^{r}).
\end{equation*}
\end{theo}

\begin{proof}
By Corollary \ref{coro:batyrev} we can choose $r$ sufficiently large so that the number of components of $\Sec(\mathcal{X}/\mathbb{P}^{1})$ with $-K_{\mathcal{X}/\mathbb{P}^{1}}$-degree at most $d$ is $\mathcal{O}(d^{r})$.  By arguing as in the proof of Theorem \ref{theo:asymptoticformula} but using  Corollary \ref{coro:batyrev} in place of Conjecture \ref{conj:irreducibility} one obtains the desired statement.
\end{proof}

\section{Stabilization of the Abel-Jacobi map}

In this section $\pi: \mathcal{X} \to \mathbb{P}^{1}$ will denote a del Pezzo fibration over $\mathbb{C}$ such that $\mathcal X$ is smooth.  We let $\IJ(\mathcal{X}) \cong H^{2,1}(X)^{\vee}/\mathrm{Im}\,  H_{3}(X,\mathbb{Z})$ denote the intermediate Jacobian of $\mathcal{X}$.  Since $h^{3,0}(\mathcal{X}) = h^{1,0}(\mathcal{X}) = 0$ the intermediate Jacobian is an abelian variety by \cite[Lemma 3.4]{CG72}.

Let $W$ be a smooth variety and let $Z \subset W \times X$ be a family of homologically equivalent effective $1$-cycles on $\mathcal{X}$.  After subtracting a constant $1$-cycle we obtain a family of homologically trivial $1$-cycles and thus an Abel-Jacobi morphism $\AJ_{W}: W \to \IJ(\mathcal{X})$ as constructed in \cite[Section 12.1.2]{Voi03}.  The choice of constant cycle only affects the map up to translation, and thus we frequently omit it from the discussion.
More generally, if $W$ is any projective variety and $Z$ is a family of homologically equivalent effective $1$-cycles on $\mathcal{X}$ over $W$, we can define the Abel-Jacobi map on the open locus where $W$ is smooth and then consider it as a rational map $\AJ_{W}: W \dashrightarrow \IJ(\mathcal{X})$.

Our goal is to analyze the Abel-Jacobi map for families of sections.  Recall that any component $M$ of $\Sec(\mathcal{X}/\mathbb{P}^{1})$ is a reduced subscheme of $\Hilb(\mathcal{X})$.  Thus there is a natural rational map $M \dashrightarrow \Chow(X)$ defined over the (semi)normal locus of $M$.  By pulling back the universal family over $\Chow(X)$, we obtain a family of cycles over the smooth locus of $M$, yielding $\AJ_{M}: M \dashrightarrow \IJ(\mathcal{X})$ that is well-defined on the smooth locus.

We need to analyze how the Abel-Jacobi map behaves under gluing of free curves.

\begin{prop} \label{prop:firstajgluing}
Let $\pi: \mathcal{X} \to \mathbb{P}^{1}$ be a del Pezzo fibration with $\mathcal{X}$ smooth whose general fiber is a del Pezzo surface with degree $\geq 3$.  Suppose that $M \subset \Sec(\mathcal{X}/\mathbb{P}^{1})$ is a component parametrizing a dominant family of sections such that $\AJ_{M}$ is dominant.  Suppose that we construct a family of sections $M'$ by gluing a family of $\pi$-vertical free conics or cubics to $M$ and smoothing.  Then $\AJ_{M'}$ is also dominant.

Furthermore, let $Z \to \IJ(\mathcal{X})$ denote the Stein factorization of the resolution of the Abel-Jacobi map for a projective closure of $M$, and let $Z' \to \IJ(\mathcal{X})$ denote the similar construction for $M'$.  Then we have a factorization $Z \to Z' \to \IJ(\mathcal{X})$.
\end{prop}

\begin{proof}
Choose a general fiber $F$ of $\pi: \mathcal{X} \to \mathbb{P}^{1}$.  Fix a component $W$ of the parameter space of conics or cubics in $F$.  By Lemma \ref{lemm:uniruledspaces} $W$ is rational; thus the Abel-Jacobi map for $W$ is constant.  Note that since $M$ parametrizes free curves a general section in $M$ will intersect a general curve parametrized by $W$.  Let $G \subset \overline{M}_{0,0}(\mathcal{X})$ denote the locus parametrizing the stable maps obtained by gluing curves parametrized by $M$ to curves parametrized by $W$.  Note that $G$ admits a dominant morphism $\psi: G \to \overline{M}$; from now on we will replace $G$ by a component for which $\psi$ is dominant.  Since $G$ intersects the smooth locus of $\overline{M}_{0,0}(\mathcal{X})$ the universal family induces an Abel-Jacobi map $\AJ_{G}: G \dashrightarrow \IJ(\mathcal{X})$.  Since the Abel-Jacobi map is constant on $W$ it is in particular constant on the fibers of $\psi$, so we see that $\AJ_{G}$ coincides with $\AJ_{\overline{M}} \circ \psi$ as a rational map.  Since the forgetful map $G \to M$ has connected fibers the Stein factorization of the resolution of a projective closure of $\AJ_{G}$ is the same as $Z$.

Let $M'$ denote the family of sections obtained by smoothing the curves parameterized by $G$ and let $\overline{M}'$ denote the corresponding component of $\overline{M}_{0,0}(X)$.  Let $Z' \to \IJ(\mathcal{X})$ denote the Stein factorization of a resolution of the Abel-Jacobi map $\AJ_{\overline{M}'}$.  Since $G$ is contained in the smooth locus of $\overline{M}'$, the restriction of $\AJ_{\overline{M}'}$ to $G$ is just $\AJ_{G}$.  Using the universal property of Stein factorizations, we see that the map $Z \to \IJ(\mathcal{X})$ must factor through the map $Z' \to \IJ(\mathcal{X})$.
\end{proof}

By combining Proposition \ref{prop:firstajgluing} with Lemma \ref{lemm:BBfordelPezzo} and an inductive argument, we obtain a similar statement when gluing $\pi$-vertical free curves of arbitrary degree.

\begin{coro} \label{coro:ajgluing}
Let $\pi: \mathcal{X} \to \mathbb{P}^{1}$ be a del Pezzo fibration with $\mathcal{X}$ smooth whose general fiber is a del Pezzo surface with degree $\geq 3$. Suppose that $M$ is a dominant family of sections such that the induced Abel-Jacobi map is also dominant.  Suppose that we construct a family of sections $M'$ by gluing a family of $\pi$-vertical free curves to $M$ and smoothing.  Then $\AJ_{M'}$ is also dominant.

Furthermore, let $Z \to \IJ(\mathcal{X})$ denote the Stein factorization of the resolution of the Abel-Jacobi map for a projective closure of $M$, and let $Z' \to \IJ(\mathcal{X})$ denote the similar construction for $M'$.  Then we have a factorization $Z \to Z' \to \IJ(\mathcal{X})$.
\end{coro}

By a combinatorial argument we can deduce a stabilization result for the Abel-Jacobi maps of the spaces of sections.

\begin{fact} \label{fact:combinatorics}
Equip $\mathbb{Z}^{m}_{\geq 0}$ with the partial order $\succeq$ such that $\vec{w} \succeq \vec{w}'$ when every coordinate of $\vec{w}$ is at least as large as the corresponding coordinate of $\vec{w}'$.  Suppose we have a function $h: \mathbb{Z}^{m}_{\geq 0} \to \mathbb{Z}$ such that $h$ preserves the partial order, i.e. $\vec{w} \succeq \vec{w}'$ $\implies$ $h(\vec{w}) \geq h(\vec{w}')$.

Then for any integer $c$, the set $h^{-1}(\mathbb{Z}_{\geq c})$ is either empty or a finite union of translates of $\mathbb{Z}^{m}_{\geq 0}$.  By a $c$-corner of $h$ we will mean any element of the minimal finite set of vectors $\{\vec{v}_{i}\}$ such that
\begin{equation*}
h^{-1}(\mathbb{Z}_{\geq c}) = \cup_{i} \left( \vec{v}_{i} + \mathbb{Z}_{\geq 0}^{m} \right).
\end{equation*}
\end{fact}

\begin{proof}[Proof of Theorem \ref{theo:maintheoremaj}:]
Let $\mathcal{F} \subset \Sec(\mathcal{X}/\mathbb{P}^{1})$ be the set of components which generically parametrize free curves of height $\geq \MBB(\mathcal{X})$. Since there are only finitely many dominant components of $\Sec(\mathcal{X}/\mathbb{P}^{1})$ not contained in $\mathcal{F}$, it suffices to restrict our attention to this set.

By applying Proposition \ref{prop:breakofflowdegree} repeatedly, we see that every family in $\mathcal{F}$ of height $\geq \MBB(\mathcal{X})+3$ can be obtained by gluing free $\pi$-vertical conics and cubics to a family of free sections of height at most $\MBB(\mathcal{X})+2$ and smoothing.  We let $\{ M_{i} \} \subset \mathcal{F}$ denote the finite set of sections of height at most $\MBB(\mathcal{X}) + 2$.  Let $\{R_{k}\}_{k=1}^{r}$ denote the families of free $\pi$-vertical conics on $\mathcal{X}$ and let $\{S_{l}\}_{l=1}^{s}$ denote the families of free $\pi$-vertical  cubics on $\mathcal{X}$.  We let $\vec{a}$ denote a vector with non-negative integer components whose $k$th coordinate records how many times we glue conics from $R_{k}$ to $M_{i}$; we let $\vec{b}$ denote the similar vector for the $S_{l}$.  By combining Corollary \ref{coro:evaluationmap2} with \cite[Lemma 5.11]{LT17} which allows us to reorder components, we see that if we fix $M_{i}$, $\vec{a}$, and $\vec{b}$ then by gluing the corresponding curves and smoothing we can only obtain one component of $\mathcal{F}$. 
In sum, we have a well-defined surjective map $\phi$ from the set of triples $\{ (M_{i},\vec{a},\vec{b}) | \, \vec{a} \in \mathbb{Z}_{\geq 0}^{r}, \vec{b} \in \mathbb{Z}_{\geq 0}^{s} \}$ to $\mathcal{F}$.

Let $\mathcal{F}_{dom}$ denote the subset of $\mathcal{F}$ such that $\AJ_{M}$ is dominant.  For each $M_{i}$, we define the function $h_{i}: \mathbb{Z}_{\geq 0}^{r+s} \to \{ 0, 1\}$ where $h_{i}(\vec{a},\vec{b}) = 1$ exactly when $\phi((M_{i},\vec{a},\vec{b})) \in \mathcal{F}_{dom}$.  By Proposition \ref{prop:firstajgluing} $h_{i}$ preserves the partial order.  By Fact \ref{fact:combinatorics}, for each fixed $M_{i}$ the set of pairs $(\vec{a},\vec{b})$ such that $\phi((M_{i},\vec{a},\vec{b})) \in \mathcal{F}_{dom}$ is a finite union of translates of $\mathbb{Z}_{\geq 0}^{r+s}$.  We conclude there is a finite subset $\{ N_{j} \} \subset \mathcal{F}_{dom}$ such that every element of $\mathcal{F}_{dom}$ can be obtained by gluing vertical conics and cubics onto the $N_{j}$.  Just as before, we represent this fact by a surjective map $\psi: \{ (N_{j},\vec{a},\vec{b}) \} \to \mathcal{F}_{dom}$.

For each fixed $j$, let $d_{j}$ denote the degree of the Stein factorization of a resolution of $\AJ_{N_{j}}$.  We define a function $g_{j}: \mathbb{Z}_{\geq 0}^{r+s} \to \mathbb{Z}$ sending a pair $(\vec{a},\vec{b})$ to $d_{j} - d(\vec{a},\vec{b})$ where $d(\vec{a},\vec{b})$ denotes the degree of the Stein factorization of a resolution of $\AJ_{\psi(N_{j},\vec{a},\vec{b})}$.  Fact \ref{fact:combinatorics} guarantees that there are only finitely many $c$-corners of $g_{j}$ as we vary $c \in \{0,\ldots,d_{j}\}$.  Furthermore, due to the factorization structure in Proposition \ref{prop:firstajgluing} we know that the Stein factorization of the resolution of any Abel-Jacobi map $\AJ_{\psi(N_{j},\vec{a},\vec{b})}$ will be dominated by the Stein factorization coming from one of these corners.  Since there are only finitely many $N_{j}$, we deduce the statement.
\end{proof}

\begin{exam}
Consider the del Pezzo fibration $\pi: \mathcal{X} \to \mathbb{P}^{1}$ consisting of a general pencil of hyperplane sections of the degree $5$ index $2$ Fano threefold $X_{5} \subset \mathbb{P}^{6}$.  As discussed in Example \ref{exam:blowupX5}, the spaces of sections of every height are irreducible.

Let $\phi: \mathcal{X} \to X_{5}$ be the birational map and let $E$ denote the exceptional divisor.  There is a $\mathbb{P}^{1}$-fibration $\pi: E \to Z$ where $Z$ is an elliptic curve in $X_{5}$. We have isomorphisms
\begin{equation*}
H^{3}(\mathcal X,\mathbb{Z})^{\vee} \mathrel{\mathop{\to}_{g}} H^{1}(E,\mathbb{Z})^{\vee} \mathrel{\mathop{\to}_{\pi^{*\vee}}} H^{1}(Z,\mathbb{Z})^{\vee}
\end{equation*}
where $g$ is the sequence of isomorphisms
\begin{equation*}
H^{3}(\mathcal X,\mathbb{Z})^{\vee} \mathrel{\mathop{\to}_{PD}} H_{3}(\mathcal X,\mathbb{Z}) \mathrel{\mathop{\to}_{\cap E}} H_{1}(E,\mathbb{Z}) \mathrel{\mathop{\to}_{PD}} H^{1}(E,\mathbb{Z})^{\vee}.
\end{equation*}
These operations all respect the Hodge structure, giving isomorphisms
\begin{equation*}
H^{2,1}(\mathcal X)^{\vee} \mathrel{\mathop{\to}_{g}} H^{1,0}(E)^{\vee} \mathrel{\mathop{\to}_{\pi^{*\vee}}} H^{1,0}(Z)^{\vee}.
\end{equation*}
Suppose $M_{d}$ is the family of sections of $\pi$ of height $d$.
Fix a general curve $C_{0}$ parametrized by $M_{d}$.  For a general point $b \in M_{d}$, we can identify a real $3$-cycle $\sigma$ on $\mathcal{X}$ such that $[C_{b}] - [C_{0}] = \partial \sigma$.  The Abel-Jacobi map sends $b$ to $[\sigma] \in H^{2,1}(X)^{\vee}/H_{3}(X,\mathbb{Z})$. Using the isomorphism $g$ we obtain an element of $H^{1,0}(E)^{\vee}/H_{1}(E,\mathbb{Z})$ which is represented by integration against $[\sigma \cap E]$.  Note that $\sigma \cap E$ is a $1$-cycle whose boundary is exactly $[C_{b} \cap E] - [C_{0} \cap E]$.  We now translate once more under the isomorphism $\pi^{*\vee}$; the corresponding element of $\Jac(Z)$ is given by integration against the $1$-cycle $\pi(\sigma \cap E)$ whose boundary is $[\phi(C_{b}) \cap Z] - [\phi(C_{0}) \cap Z]$.  In sum, under the chain of isomorphisms above, a general point $b \in M_{d}$ is mapped under the Abel-Jacobi map to the point $\phi(C_{b}) \cap Z - \phi(C_{0}) \cap Z$ in $\Jac(Z)$.

This map can be realized geometrically as follows.  Note that for any $d \geq 0$ a general height $d$ section on $\mathcal{X}$ is the same thing as a curve in $X_{5}$ of degree $d+1$ which meets $Z$ in $d$ distinct points.  This equips $M_{d}$ with a rational map to $\Sym^{d}(Z)$.  Since $Z$ is an elliptic curve $\Sym^{d}(Z)$ is a projective bundle over $\mathrm{Jac}(Z)$, and the composition $M_{d} \dashrightarrow \mathrm{Jac}(Z)$ is the Abel-Jacobi map (up to translation).

We can now analyze the Abel-Jacobi map for low height sections: 
\begin{itemize}
\item[] \textbf{Height 0:} The Abel-Jacobi map contracts $M_{0}$ to a point.  Note that since $M_{0}$ is birational to the space of lines in $X_{5}$ a projective closure is rationally connected by \cite{FN89}.
\item[] \textbf{Height 1:}  We claim that the Abel-Jacobi map is the MRC fibration for $M_{1}$.  Recall that an open subset of $M_{1}$ parametrizes conics on $X_{5}$ which meet $Z$ once.  By \cite{Iliev94} the space of conics on $X_{5}$ is isomorphic to $\mathbb{P}^{4\vee}$; the isomorphism sends a general conic $T \subset X_{5} \subset G(2,5)$ to the hyperplane in $\mathbb{P}^{4}$ spanned by the corresponding lines.

Suppose we now fix a quintic elliptic $Z$ in $X_{5}$.  Note that there is a morphism from $G(2,5)$ to the parameter space $P$ of Schubert varieties $\sigma_{1,1}$ in $\mathbb{P}^{4\vee}$ sending a line to the set of hyperplanes containing it.  The image of $Z$ under this map is an elliptic curve in $P$, and the corresponding Schubert varieties sweep out a $\mathbb{P}^{2}$-bundle over $Z$.  We claim that $M_{1}$ is birational to this $\mathbb{P}^{2}$ bundle.  Indeed, we know that $M_{1}$ has dimension $3$ and is contained in this locus, thus it must contain an open subset.

The Abel-Jacobi map for $M_{1}$ sends a curve to its intersection point with $Z$.  This agrees with the projection map from the $\mathbb{P}^{2}$-bundle to $Z$.  This proves our claim.

\item[] \textbf{Height 2:} The Abel-Jacobi map is a MRC fibration. Indeed, $\mathrm{Sym}^2(Z) \rightarrow \mathrm{Jac}(Z)$ is a projective bundle.  We claim that a general fiber of $M_2 \dashrightarrow \mathrm{Sym}^2(Z)$ is rationally connected.   Fix two general points $x_{1}, x_{2}$ on $Z$.  \cite[Theorem 4.9]{FGP19} shows that the general fiber of the map from the $2$-pointed cubic curves on $X_{5}$ to $X_{5} \times X_{5}$ is rationally connected.  (Note that this does not yet imply our desired statement since the fiber of interest in our situation is not general.)  Their argument uses the birational map $\phi: X_{5} \dashrightarrow Q$ to the three-dimensional quadric $Q$ given by projection from a line through $x_{1}$. Tracing through the proof, it suffices to show that the space of conics in the three-dimensional quadric $Q$ which:
\begin{itemize}
\item meet the twisted cubic where $\phi^{-1}$ is not defined,
\item meet the line in $Q$ defined by as the fiber over $x_{1}$, and 
\item meet $\phi(x_{2})$
\end{itemize}
is rationally connected. This is the same as the space of planes in $\mathbb{P}^{4}$ meeting these loci, and it is not difficult to see directly that this surface admits a rational map to $\mathbb{P}^{1}$ whose fibers are rational curves.
\end{itemize}

By the argument of Theorem \ref{theo:maintheoremaj} we deduce that the Abel-Jacobi map has connected fibers for the family of sections of any height $\geq 1$.  We conjecture that in this situation the Abel-Jacobi map always agrees with the MRC fibration for a family of sections.
\end{exam}

\subsection{Conic bundles over del Pezzo surfaces}
In this section we discuss a couple examples showing that the Abel-Jacobi map does not need to interact well with other possible stabilization results. 
These examples were suggested to us by Hassett.

The first example shows that there might be no components of $\Sec(\mathcal{X}/\mathbb{P}^{1})$ whose Abel-Jacobi map $M \dashrightarrow \IJ(\mathcal{X})$ is dominant and birationally equivalent to the MRC fibration.

\begin{exam} \label{exam: fibrationoverdegree2}
Let $S$ be a general del Pezzo surface of degree $2$ equipped with a conic fibration $g: S \to \mathbb{P}^{1}$.
Let $q : \mathcal X \rightarrow S$ be a conic bundle over $S$ such that its discriminant divisor $D$ is a very general element in $|-2K_S|$ (as constructed by \cite{AM72}). Then $\pi = g \circ q : \mathcal X \to \mathbb P^1$ is a del Pezzo fibration of degree $4$.   \cite[Theorem 1]{HKT16b} shows that $\mathcal{X}$ does not admit an integral Chow decomposition of the diagonal.  

The variety $\mathcal{X}$ also does not admit an integral homological decomposition of the diagonal.  To see this, let $\mathcal{M}$ be a component of the parameter space for the set of conic bundles over $S$ whose discriminant divisor lies in $|-2K_{S}|$.  Since $|-2K_{S}|$ is basepoint free and contains a union of two smooth elliptic curves meeting transversally in $2$ points, \cite[Proof of Theorem 1]{HKT16b} shows that $\mathcal{M}$ parametrizes a conic bundle which has only ordinary double point singularities and whose resolution has non-trivial Brauer group.  \cite[Corollary 4.5]{Voisin13} shows that this resolution does not admit an integral homological decomposition of the diagonal.  Finally, \cite[Theorem 2.1.(ii)]{Voisin15} (combined with a Hilbert scheme argument) applies in our situation to show that a very general $\mathcal{X}$ parametrized by $\mathcal{M}$ also does not admit an integral homological decomposition of the diagonal.

We next verify that the hypotheses of \cite[Theorem 4.9]{Voisin13} hold in this situation.  Since $\mathcal{X}$ is rationally connected, all the higher cohomology of the structure sheaf vanishes verifying Hypothesis (i).  \cite[Theorem 1.3]{Voisin13} shows that $Z^{4}(\mathcal{X}) = 0$ verifying Hypothesis (ii).  Hypothesis (iii) has two parts.  First, since $\IJ(\mathcal{X})$ is isomorphic to the Jacobian of a genus $2$ curve, there is a one-cycle $\Gamma$ of class $[\Theta]^{g-1}/(g-1)!$ on $\IJ(\mathcal{X})$.  Second, since $\mathcal{X} \to S$ has irreducible smooth discriminant locus, \cite[Theorem 2]{AM72} shows that the Brauer group of $\mathcal{X}$ is trivial. 

Since $\mathcal{X}$ does not admit an integral homological decomposition of the diagonal, \cite[Theorem 4.9]{Voisin13} implies that if $M \subset \Sec(\mathcal{X}/\mathbb{P}^{1})$ is a component then the Abel-Jacobi map is either non-dominant or the projective closure of a general fiber is not rationally connected.   In particular the MRC fibration for a resolution of a projective closure of $M$ cannot coincide with the Abel-Jacobi map whenever the Abel-Jacobi map is dominant.
\end{exam}

\begin{rema}
We expect that in Example \ref{exam: fibrationoverdegree2} the dimension of the base of the MRC fibration is unbounded.  Indeed, consider the map $h: \Sec(\mathcal{X}/\mathbb{P}^{1}) \to \Sec(S/\mathbb{P}^{1})$ induced by pushforward.  Let $M$ be a component of $\Sec(\mathcal{X}/\mathbb{P}^{1})$.  Every component of a general fiber of $h|_{M}$ is rationally connected and the finite part of the Stein factorization of $h|_{M}$ can have very large degree.  Thus it seems quite likely that the Stein factorization of $h|_{M}$ will often coincide with the MRC fibration for $M$.
\end{rema}

The following example shows that there can be an infinite set of components of $\Sec(\mathcal{X}/\mathbb{P}^{1})$ whose Albanese varieties have larger dimension than $\IJ(\mathcal{X})$.

\begin{exam}  \label{exam: fibrationoverdegree1}
Let $S$ be a general del Pezzo surface of degree $2$ equipped with a conic fibration $g: S \to \mathbb{P}^{1}$.  Let $q: \mathcal{X} \to S$ be a conic bundle over $S$ whose discriminant locus $D$ is a general element of $|-2K_{S}|$ (as constructed by \cite{AM72}).  By composing $q$ with $g$ we obtain a degree $4$ del Pezzo surface fibration $\pi: \mathcal{X} \to \mathbb{P}^{1}$.  In this case $\IJ(\mathcal{X})$ is the Jacobian of a genus $2$ curve.

Note that $q$ induces a map $h: \Sec(\mathcal{X}/\mathbb{P}^{1}) \to \Sec(S/\mathbb{P}^{1})$. If $C$ is a section of $g: S \to \mathbb{P}^{1}$ then $C$ is not contained in the discriminant locus, so that its $q$-preimage is generically a rationally connected fibration over $C$.  Thus the map $\Sec(\mathcal{X}/\mathbb{P}^{1}) \to \Sec(S/\mathbb{P}^{1})$ is surjective.

Let $C$ be a section of $g: S \to \mathbb{P}^{1}$ satisfying $-K_{S} \cdot C = 2$ and $C^2 = 0$.
Note that the corresponding component $N$ of $\Sec(S/\mathbb{P}^{1})$ is an open subset of $\mathbb{P}^{1}$.  Assuming $C$ is general, it meets the discriminant locus transversally in $4$ points.  Choose any component $M$ of $\Sec(\mathcal{X}/\mathbb{P}^{1})$ mapping dominantly onto $N$ and let $\overline{M}$ be a smooth projective variety birational to $M$ for which the Abel-Jacobi map is a morphism.    Denote by $h_{1}: \overline{M} \to Z$ and $h_{2}: Z \to \mathbb{P}^{1}$ the Stein factorization of the map $\overline{M} \to \mathbb{P}^{1}$ induced by $h$. We prove that $Z$ is a smooth curve of genus at least $5$.

Let $\gamma: S \to \mathbb{P}^{1}$ be the morphism which contracts the deformations of the conic $C$.  The restriction of $\gamma$ to the discriminant divisor $D$ has degree $4$.  We claim that $\gamma|_{D}: D \to \mathbb{P}^{1}$ is simply branched.  Indeed, by generality we may ensure that the branch points of $\gamma|_{D}$ avoid the singular fibers of $\gamma: S \to \mathbb{P}^{1}$. If we fix a fiber $C$ of $\gamma$, then the locus in $|-2K_{S}|$ parametrizing curves $D$ bitangent to $C$ has codimension $2$. Since there is only a $1$-parameter family of deformations of $C$ we conclude that a general $D$ will not be bitangent to any of $C$. 

We deduce that that $\gamma|_{D}$ has $12$ ramification points each of which is contained in a smooth fiber of $\gamma$.  This also shows that the monodromy of $D$ over $\mathbb{P}^{1}$ is transitive and generated by transpositions, and is thus the entire symmetric group $S_{4}$.

Let $I$ denote the divisor in $\mathcal{X}$ which is the preimage of $D$ in $S$.  Note that $I$ is irreducible because it arises from the construction of \cite{AM72}.
Let $I'$ denote a resolution of $I$ and let $B$ denote the Stein factorization of the map $I' \to \mathbb{P}^{1}$ induced by $\gamma$.  Let $G_{mon}$ denote the mondromy group of $\psi: B \to \mathbb{P}^{1}$.  First of all note that $\psi$ factors through $D$ as a $2:1$ cover.  Furthermore, the map $B \to D$ is \'etale: the discriminant divisor is smooth, so every fiber of $g|_{S}$ consists of two distinct lines and not a double line.  Using this covering map to $D$, we get an exact sequence
\begin{equation}
\label{equation:monodromyexact}
1 \to K \to G_{mon} \to S_{4} \to 1.
\end{equation}
where $K$ is the kernel of $G_{mon} \to S_4$.
Then since $G_{mon}$ is contained in the stabilizer of a conic fibration in the Weyl group $W(D_5)$, we have $\#K = 1, 2, 4$, or $8$. Since $I$ is irreducible, $\#K = 1$ is impossible. Moreover we can conclude that $K$ consists of involutions, since each element will either fix or switch the two irreducible components in each reducible fiber of $\gamma$.  

Suppose we fix a conic $C$ in our family on $S$.  The $g$-preimage of $C$ is a surface such that the map to $C$ has $4$ reducible fibers.  Thus there are 16 possible intersection profiles for curves lying over $C$ in a fixed numerical class on $\mathcal{X}$.  
We next study the action of $G_{mon}$ on these 16 possible intersection profiles. 
First of all note that there are $8$ intersection profiles realized by sections which are $(-1)$-curves, and there are $8$ other intersection profiles which cannot be realized by sections which are $(-1)$-curves. These are preserved by the monodromy action $G_{mon}$.
Our goal is to prove that there are two orbits under the action of $G_{mon}$ of size $8$ given by these intersection profiles.

First we will prove that $K$ has a unique non-trivial involution $\sigma$ in the center of $G_{mon}$. Suppose we fix a section in our family $M$.  This section intersects the $4$ reducible fibers along components $Q_{1}$, $Q_{2}$, $Q_{3}$, $Q_{4}$, and does not intersect the other $4$ components $Q_{1}'$, $Q_{2}'$, $Q_{3}'$, $Q_{4}'$ of the reducible fibers.  We can identify the action of $G_{mon}$ on intersection profiles by analyzing how it permutes these $8$ curves.  Note that the quotient $S_{4}$ of $G_{mon}$ acts on the set of pairs $P_{i} := \{Q_{i}, Q_{i}'\}$ via the natural symmetric group action.  If we have a non-trivial involution $\sigma$ in the center then the action of $\sigma$ is given by $Q_{i} \leftrightarrow Q_{i}'$. The conjugacy classes of other involutions have size $4$, $6$, $4$ respectively. So the only possibilities are that $\#K = 2$ and the decomposition of conjugacy classes on $K$ is $(1,1)$ or $\#K = 8$ and the decomposition of conjugacy classes on $K$ is given by $(1, 1, 6)$. Thus $K$ has a unique non-trivial involution $\sigma$ in the center of $G_{mon}$.  Also note that $K$ does not intersect non-trivially with the stabilizer of any intersection profile and elements in $K$ act trivially on the set of pairs $\{ P_{i} \}$. 

If $K$ has order $2$, the sequence~(\ref{equation:monodromyexact}) cannot split. If it did then we would have a factoring $B \to D' \to \mathbb P^1$ where $D' \to \mathbb P^1$ has degree $2$. Since $D \to \mathbb P^1$ is simply branched and $B \to D$ is \'etale, the fibers of $B\to \mathbb P^1$ would consist of either $8$ points or $6$ points. This implies that $D' \to \mathbb P^1$ is \'etale. But this is impossible. 


%

Now the possible sizes of the $G_{mon}$-orbit of an intersection profile are $2, 4, 6, 8$.  A combinatorial argument shows that if there is more than two orbits then there must be an orbit of size $2$ or $4$, so it suffices to prove there is no orbit of these types. When $K$ has order $8$, orbits of size $2$ and $4$ are impossible as $K$ does not intersect non-trivially with the stabilizer of any intersection profile.

When $K$ has order $2$, an orbit of size $2$ is impossible: since $\sigma$ is not contained in the stabilizer, the map $G_{mon} \to S_4$ would restrict to an isomorphism on the stabilizer group, contradicting the fact that there is no splitting of this map.  An orbit of size $4$ is also impossible: if we had such an orbit, then the stabilizer $G_{stab}$ would be identified with $A_{4}$ under the quotient map $G_{mon} \to S_{4}$.  Let $\tau_{ij}$ for $1 \leq i < j \leq 4$ be an element of $G_{mon}$ mapping to the transposition $(i j)$ in $S_{4}$.  Any product of an even number of the $\tau_{ij}$ must either lie in $G_{stab}$ or $\sigma G_{stab}$.  By writing down the possible actions of the $\tau_{ij}$ on the components of the singular fibers which satisfy this constraint, it is not hard to see that either some $\tau_{ij}$ or some $\sigma \tau_{ij}$ will also be contained in the stabilizer, a contradiction.  




To summarize, we have shown that the action of $G_{mon}$ on the set of intersection profiles of sections lying above $C$ has two orbits.  This implies that there are at most two components $M$ of $\Sec(\mathcal{X}/\mathbb{P}^{1})$ which maps to the component $N$ and which represents a fixed numerical class on $\mathcal{X}$.   The Stein factorization of a resolution of a projective closure of the map $M \to N$ has degree $8$.

Consider again the conics $C$ which are fibers of $\gamma$.  If $C$ is general, then the $g$-preimage $\Gamma$ is birationally ruled over $C$ with $4$ singular fibers.  But when $C$ is tangent to $D$, the singular fibers of $\Gamma$ behave differently.  This will happen at the $12$ simple ramification points of $\gamma|_{D}$. At each such point, formally locally in a neighborhood of the fiber over the tangency point $\Gamma$ looks like the subset of $\mathbb{A}^{2}_{s,t} \times \mathbb{P}^{3}_{x,y, z}$ defined by the equations $s=0$ and
\begin{equation*}
(s-t^{2})z^2 + xy.
\end{equation*}
 Setting $s=0$, we see that $\Gamma$ has a single $A_{1}$-singularity.  Its resolution has three singular fibers under the map to $S$ which have $2$, $2$, and $3$ components. Thus there are possibly $12$ intersection profiles for such a surface, but $6$ of them are realized by sections of $(-1)$-curves and other $6$ of them are not realized by sections of $(-1)$-curves. We deduce that the fiber of $h_{2}$ corresponding to this singular surface has at most $6$ components.  Thus the contribution to the ramification divisor of $h_{2}$ at this point is at least $2$.  Altogether we see that the contribution of such fibers to the ramification divisor $R$ for $h_{2}$ has degree at least $24$, so that the genus of $Z$ is at least $5$.

We have a morphism $\Alb(\overline{M}) \to \Jac(Z)$ induced by $h_{1}$.  Since $h_{1}: \overline{M} \to Z$ is surjective, the image of $\Alb(\overline{M})$ generates $\Jac(Z)$.  Thus the map from $\Alb(\overline{M}) \to \Jac(Z)$ is surjective.  We conclude that the dimension of $\Alb(\overline{M})$ is at least $5$, so it cannot coincide with $\IJ(\mathcal{X})$.
\end{exam}

\section{Gromov-Witten invariants}

Let $\pi: \mathcal{X} \to \mathbb{P}^{1}$ denote a del Pezzo fibration such that $\mathcal X$ is smooth.  We will consider Gromov-Witten invariants on $\mathcal{X}$ with $n$ point insertions, i.e.~$\langle [pt]^{n} \rangle_{0,n}^{\mathcal{X},\beta}$ where $\beta$ is a curve class satisfying $F \cdot \beta = 1$ for a general fiber $F$ of $\pi$.  In order to obtain the correct virtual dimension, we must have $-K_{\mathcal{X}/\mathbb{P}^{1}} \cdot \beta  = 2n-2$.

\begin{prop} \label{prop:gwenumerative}
Let $\pi: \mathcal{X} \to \mathbb{P}^{1}$ be a del Pezzo fibration such that $-K_{\mathcal{X}/\mathbb{P}^{1}}$ is relatively ample and $\mathcal X$ is smooth.  Suppose that \begin{equation*}
n \geq \maxdef(\mathcal{X}) + 2 + \sup\{0, - \neg(\mathcal{X},-K_{\mathcal{X}/\mathbb{P}^{1}})\}.
\end{equation*}
Let $\beta \in N_{1}(\mathcal{X})_{\mathbb{R}}$ denote a curve class satisfying $-K_{\mathcal{X}/\mathbb{P}^{1}} \cdot \beta = 2n-2$ and $F \cdot \beta = 1$ for a general fiber $F$ of $\pi$.  Then the GW invariant $\langle [pt]^{n} \rangle_{0,n}^{\mathcal{X},\beta}$ is enumerative: it counts the number of rational curves of class $\beta$ through $n$ general points of $\mathcal{X}$.
\end{prop}

This extends \cite[Theorem 4.1]{Tian12} which establishes the non-vanishing of certain GW-invariants for sections of a del Pezzo fibration.  The argument is very similar to the proof of Theorem \ref{theo:maintheorem2}.

\begin{proof}
Fix $n$ general points of $\mathcal{X}$.  If a component $M \subset \overline{M}_{0,n}(\mathcal{X},\beta)$ contributes to the GW invariant, then there must be a stable map parametrized by $M$ whose image $C$ contains all $n$ points.  By Proposition \ref{prop:evendegreebound} we know that $M$ must generically parametrize free curves.  In particular, such components $M$ have the expected dimension, are generically smooth, and do not admit a generic stabilizer.  This suffices to yield the result.
\end{proof}

By essentially the same argument, we have:

\begin{prop} \label{prop:gwenumerative2}
Let $\pi: \mathcal{X} \to \mathbb{P}^{1}$ be a del Pezzo fibration such that $-K_{\mathcal{X}/\mathbb{P}^{1}}$ is relatively ample and $\mathcal X$ is smooth.  Suppose that \begin{equation*}
n \geq \maxdef(\mathcal{X}) + \sup\{0, - \neg(\mathcal{X},-K_{\mathcal{X}/\mathbb{P}^{1}})\}.
\end{equation*}
Let $\beta \in N_{1}(\mathcal{X})_{\mathbb{R}}$ denote a curve class satisfying $-K_{\mathcal{X}/\mathbb{P}^{1}} \cdot \beta = 2n-1$ and $F \cdot \beta = 1$  for a general fiber $F$ of $\pi$.  Fix a $\pi$-vertical curve $C$ that is a general member of a basepoint free linear series in a general fiber.  Then the GW invariant $\langle [pt]^{n}, C \rangle_{0,n+1}^{\mathcal{X},\beta}$ is enumerative: it counts the number of rational curves of class $\beta$ that intersect $C$ and contain $n$ general points of $\mathcal{X}$.
\end{prop}

Finally, we prove an existence result which can be used to show that certain GW invariants are non-zero.  Recall that a smooth rational curve $C$ contained in a smooth threefold $\mathcal{X}$ is said to have balanced normal bundle if $N_{C/X} \cong \mathcal{O}(a) \oplus \mathcal{O}(b)$ where $|b-a| \leq 1$.

\begin{prop} \label{prop:balancednormalbundle}
Let $\pi: \mathcal{X} \to \mathbb{P}^{1}$ be a del Pezzo fibration such that $-K_{\mathcal{X}/\mathbb{P}^{1}}$ is relatively ample and $\mathcal X$ is smooth.   Suppose that $n \geq \frac{Q(\mathcal{X})+2}{2}$ where $Q(\mathcal{X})$ is defined as in Theorem \ref{theo:mbbmain}.  There exists a free section with balanced normal bundle containing $n$ general points of $\mathcal{X}$.

If the general fiber of $\pi$ is not $\mathbb{P}^{2}$ or $\mathbb{P}^{1} \times \mathbb{P}^{1}$, then for any $a \geq \frac{Q(\mathcal{X})+8}{2}$ there exists a free section with normal bundle $\mathcal{O}(a) \oplus \mathcal{O}(a)$ and a free section with normal bundle $\mathcal{O}(a) \oplus \mathcal{O}(a+1)$.  In particular, the  Gromov-Witten invariant as in Proposition \ref{prop:gwenumerative} or Proposition \ref{prop:gwenumerative2} is positive for the corresponding curve class.  
\end{prop}

\cite[Proof of Theorem 4.1]{Tian12} gives a similar statement with no explicit restriction on the size of $n$ or $a$.

\begin{proof}
We know there exists a very free section $C$ through $n$ general points of $\mathcal{X}$.  Let us choose such a curve of minimal height.  Write $N_{C/\mathcal{X}} = \mathcal{O}(a) \oplus \mathcal{O}(b)$ with $a \leq b$.  Furthermore since $C$ contains $n$ general points we know that if we replace $C$ by a general deformation then $a \geq n-1$ by Lemma \ref{lemm:normalbundleestimate}.

Suppose for a contradiction that the normal bundle is not balanced, so that $a +1 < b$.  By repeating the arguments of \cite{Shen12} as used in Theorem \ref{theo:mbbmain}, we obtain a smooth model $\widetilde{\Sigma}$ of a surface $\Sigma$ swept out by deformations of $C$ through $a+1$ general points.  We further blow up the preimages of the general $a + 1$ points so that the preimages are divisorial.  There is an induced generically $\mathbb{P}^{1}$-fibration $\psi: \widetilde{\Sigma} \to \mathbb{P}^{1}$, and the preimage of each of the general points contains a component of a fiber of $\psi$ that meets the strict transform $\widetilde{C}$ of $C$.  Continuing the argument of Theorem \ref{theo:mbbmain}, we can break the curve $\widetilde{C}$ on $\widetilde{\Sigma}$ into $\widetilde{C}_{1} + F$ where $F$ denotes a fiber of $\psi$. 
Note that $\widetilde{C}_{1}$ must have the same intersection against $\psi$-vertical curves in $\widetilde{\Sigma}$ as $\widetilde{C}$ does, and in particular, the image of $\widetilde{C}_{1}$ must also go through $a+1$ general points of $\mathcal{X}$ and avoid singularities of $\mathcal X$.  This contradicts the minimality of $C$.  Thus $C$ must have balanced normal bundle, proving the first statement.

Now suppose that a general fiber of $\pi$ is not $\mathbb{P}^{2}$ or $\mathbb{P}^{1} \times \mathbb{P}^{1}$.  In particular, a general fiber $F$ contains $-K_{\mathcal{X}/\mathbb{P}^{1}}$-cubics which are the strict transforms of lines on $\mathbb{P}^{2}$ and contains $-K_{\mathcal{X}/\mathbb{P}^{1}}$-quartics $T$ such that $T_{F}|_{T} \cong \mathcal{O}(2) \oplus \mathcal{O}(2)$.  Some elementary deformation arguments show:
\begin{itemize}
\item Suppose we take a general section $C$ with normal bundle $\mathcal{O}(a) \oplus \mathcal{O}(a)$, fix a general fiber $F$, and set $T$ to be a curve in $F$ through $C \cap F$.  If $T$ is a $-K_{\mathcal{X}/\mathbb{P}^{1}}$-cubic with a general tangent direction at $C \cap F$, then a general smoothing of $C \cup T$ has normal bundle $\mathcal{O}(a+1) \oplus \mathcal{O}(a+2)$. 
If $T$ is a $-K_{\mathcal{X}/\mathbb{P}^{1}}$-quartic with $T_F|_{T} = \mathcal{O}(2) \oplus \mathcal{O}(2)$ and a general tangent direction at $C \cap F$, then a general smoothing of $C \cup T$ has normal bundle $\mathcal{O}(a+2) \oplus \mathcal{O}(a+2)$.
\item Suppose we take a general section $C$ with normal bundle $\mathcal{O}(a) \oplus \mathcal{O}(a+1)$, fix a general fiber $F$, and set $T$ to be a curve in $F$ through $C \cap F$.  If $T$ is a $-K_{\mathcal{X}/\mathbb{P}^{1}}$-cubic with a general tangent direction at $C \cap F$, then a general smoothing of $C \cup T$ has normal bundle $\mathcal{O}(a+2) \oplus \mathcal{O}(a+2)$.  If $T$ is a $-K_{\mathcal{X}/\mathbb{P}^{1}}$-quartic with $T_F|_{T} = \mathcal{O}(2) \oplus \mathcal{O}(2)$ and a general tangent direction at $C \cap F$, then a general smoothing of $C \cup T$ has normal bundle $\mathcal{O}(a+2) \oplus \mathcal{O}(a+3)$.
\end{itemize}
By applying this gluing argument twice to a curve with balanced normal bundle, we obtain curves of the desired type with any height at least $6$ more than the original curve.  Since the balanced curve constructed in the first part has height $Q(\mathcal{X})$ or $Q(\mathcal{X})+1$, we obtain the desired statement.
\end{proof}

\bibliographystyle{alpha}
\bibliography{Sections}

\end{document}